\documentclass[10pt,a4paper]{article}

\usepackage[utf8]{inputenc}
\usepackage{amssymb}
\usepackage{graphicx}
\usepackage{amsmath,amsfonts,amsthm}
\usepackage{array,cases}
\usepackage{tikz}
\usepackage{color}
\usepackage{bm}
\numberwithin{equation}{section}
\def\eps{\varepsilon}
\def\calC{\mathcal{C}}

\def\calL{\mathcal{L}}
\def\calP{\mathcal{P}}
\def\loc{\mathrm {loc}}
\def\ext{\mathrm {ext}}
\def\intt{\mathrm {int}}
\def\R{\mathbb{R}}
\def\N{\mathbb{N}}

\def\calH{\mathcal{H}}
\newcommand\Div{\mathrm{div}}
\newcommand\Lip{\mathrm{Lip}}
\providecommand{\dcalH}{\, \mathrm{d}\calH}
\providecommand{\dcalL}{\, \mathrm{d}\calL}

\providecommand{\supp}{\operatorname{supp}}
\providecommand{\dr}{\, \mathrm{d} r}

\providecommand{\dt}{\, \mathrm{d} t}
\providecommand{\ds}{\, \mathrm{d} s}
\providecommand{\dx}{\, \mathrm{d} x_1}
\providecommand{\dxy}{\, \mathrm{d} x}
\providecommand{\dy}{\, \mathrm{d} x_2}

\providecommand{\dxy}{\, \mathrm{d} \mathbf{x}}

\newcommand{\dist}{\operatorname{dist}}
\newcommand{\Tr}{\operatorname{Tr}}
\newcommand{\rank}{\operatorname{rank}}
\newtheorem{thrm}{Theorem}[section]
\newtheorem{lmm}[thrm]{Lemma}

\newtheorem{prpstn}[thrm]{Proposition}

\newtheorem{rmrk}[thrm]{Remark}
\begin{document}
\title{Energy scaling laws for geometrically linear elasticity models for microstructures in shape memory alloys$^*$}
\author{Sergio Conti$^1$, Johannes Diermeier$^1$, David Melching$^2$ and Barbara Zwicknagl$^3$}
\date{March 8, 2020}
%
%
%
%
%
\maketitle
\renewcommand*{\thefootnote}{*}
\footnotetext{This work was mainly done while all authors were at the University of Bonn and 
was partially supported by the {\em Deutsche Forschungsgemeinschaft} via
project 211504053 - SFB 1060/A{06}.}
\renewcommand*{\thefootnote}{\arabic{footnote}}
\footnotetext[1]{
 Institut f\"ur Angewandte Mathematik, Universit\"at Bonn,
    53115 Bonn, Germany}
\footnotetext[2]{Fakult\"at f\"ur Mathematik, Universit\"at Wien, 
1090 Wien, Austria}
\footnotetext[3]{
Institut f\"ur Mathematik,    {Humboldt-Universit\"at zu Berlin,
10117} Berlin, Germany}

\begin{abstract} We consider a singularly-perturbed two-well problem in the context of planar geometrically linear elasticity to model a rectangular martensitic nucleus in an austenitic matrix. We derive the scaling regimes for the minimal energy in terms of the problem parameters, which represent the {shape} of the nucleus, the quotient of the elastic moduli of the two phases, the surface energy constant, and the volume fraction of the two martensitic variants. We identify several different scaling regimes, which are distinguished either by the exponents in the parameters, or by logarithmic corrections, for which we have matching upper and lower bounds.
\end{abstract}

\section{Introduction}
Solid-solid phase transitions are a classical model problem in the variational study of pattern formation in solids, both in the context of the theory of relaxation and in the study of singularly perturbed problems.
Their study has led on the one side to many important abstract developments in the calculus of variations, on the other side to a mathematical explanation of the physical behavior of shape-memory alloys and other materials with peculiar properties \cite{ball-james:87,ball-james:92,bhattacharya:92,MuellerLectureNotes,ball:04,kohn:06,James2019}.
The basic model is a vectorial, nonconvex variational problem, where the integrand depends on the gradient of the deformation field.
The study of the macroscopic material behavior is strongly coupled to the development of the theory of quasiconvexity and relaxation \cite{MuellerLectureNotes,Dacorognabuch},
and focuses on average properties of the microstructures without resolving the geometric details and the microscopic length scales.

A finer analysis requires the introduction of a length scale, typically in the form of a small parameter times a convex function of a second gradient, which penalizes interfaces.
The resulting singularly-perturbed  nonconvex problem contains  a 
scale dependence and is much more difficult to study in detail, a numerical treatment is in most cases not feasible {either}.
Starting with the papers by Kohn and Müller  \cite{kohn-mueller:92,kohn-mueller:94} it has become clear that the key
property is the scaling of the optimal energy in terms of the parameters present in the problem, and that it is appropriate
to start by focusing on the exponents and ignoring the prefactor. One obtains mesoscopic phase diagrams which characterize the different 
regimes of material behavior and the qualitative properties of the microstructure \cite{kohn:06,CapellaOtto2009,knuepfer-kohn:11,knuepfer-kohn-otto:13}.
The techniques developed for singularly-perturbed functionals modeling martensitic microstructures have proven useful also in the study
of a variety of other physical problems, such as {for example} magnetic microstructures \cite{choksi-kohn:98,choksi-et-al:98,knuepfer-muratov:11}, flux tubes in superconductors \cite{CKO03,conti-et-al:15,ContiGoldmanOttoSerfaty2018},
diblock copolymers  \cite{Choksi01}, wrinkling in thin elastic films \cite{JinSternberg2,belgacem-et-al:02,bella-kohn:14}, 
and compliance minimization \cite{Kohn-Wirth:15}.

One aspect which is very important {for} practical applications of materials with solid-solid phase transitions is the detailed study of the transformation path from austenite to martensite and the corresponding hysteresis. 
It is known that the amplitude of the hysteresis cycle crucially depends on the microstructures that emerge during nucleation \cite{cui-et-al:06,zjm:09,zwicknagl:14}.
Specifically, transition-state theory explains that the transformation from austenite to martensite is strongly influenced by the energetics of the critical nucleus, which is a small inclusion of martensite in an austenitic matrix. 
{It is known that stress-free inclusions with interfaces of finite total area (or length, in two dimensions) are possible only for special material parameters \cite{dolzmann-mueller:95,kirchheim:03,knuepfer-kohn:11,knuepfer-kohn-otto:13,rueland:16,rueland:16_2,ContiKlarZwicknagl2017,cesana-et-al:19}.}

We investigate here a variational model for the formation of microstructures in a martensitic nucleus embedded in an austenitic matrix. 
The mechanical framework is the theory of geometrically linear elasticity, the mathematical framework is a singularly perturbed nonconvex vectorial functional. 
Before discussing the large body of mathematical literature that has been devoted to variants of this problem in the last decades, let us briefly introduce the setting. 
For simplicity we work in two spatial dimensions and consider a large body, identified with $\R^2$, 
which is mostly austenitic with a bounded martensitic inclusion $\omega\subset\subset\R^2$.
The inclusion $\omega$ is selected by a process slower than elastic equilibration and therefore, for the present purposes, fixed.

We take the austenite state as reference configuration and denote by $u:\R^2\to\R^2$ the elastic displacement.
We assume that two variants of martensite are relevant, which are characterized by strains $A, B\in\R^{2\times 2}_\text{sym}$.
Experimentally it is known that martensitic transformations are to a very good approximation volume preserving \cite{bhattacharya:92}, therefore we assume $\Tr A=\Tr B=0$.
Relaxation theory predicts zero macroscopic energy if $A$ and $B$ are compatible 
and austenite can be realized as a weighted average of the two martensitic variants. This means
that 
there are matrices $\hat A, \hat B\in\R^{2\times 2}$
with $\hat A-A$ and $\hat B-B$ skew-symmetric 
such that
$\rank (\hat A-\hat B)=1$ and 
$(1-\theta) \hat A+\theta \hat B=0$ for some $\theta\in (0,1)$.
In this situation, a finer analysis, which includes a singular perturbation regularizing the microstructure, is necessary in order to understand the detailed material behavior.
Since austenite/martensite interfaces which are not aligned with the rank-one direction have very large energy, one {expects} the nucleus to be elongated in the rank-one direction.
For mathematical simplicity it is convenient to further restrict the geometry. Since $\hat A$ and $\hat B$ are rank-one connected, we have $\hat A-\hat B=c\otimes n$ for some $c,n\in\R^2$. By scaling we can assume $|c|=|n|=1$.
From $\Tr A=\Tr B=0$ one obtains 
$\Tr \hat A=\Tr \hat B=0$ and therefore
$c\cdot n=0$, and from $(1-\theta) \hat A+\theta \hat B=0$ one obtains
 $\hat A=\theta c\otimes n$, $\hat B=(\theta-1)c\otimes n$. By a change of variables one can reduce to the case that
 $\theta\le\frac12$, $c=e_1$ and $n=e_2$. We shall then assume that the martensitic domain is a rectangle elongated along $e_1$, and by scaling it suffices to consider
\begin{eqnarray}\label{eq:Omega}
\Omega_{2L}:=(0,2L)\times(0,1)\subset\R^2.
\end{eqnarray}
The same pair of matrices allows for a second rank-one connection, rotated by 90 degrees. Therefore we can assume without loss of generality that
\[L\geq \frac 12. \]
In particular, both edges of $\Omega_{2L}$ are aligned with the habit planes of exact austenite/martensite interfaces.
In the austenite the elastic energy vanishes if 
the strain $e(u)$, defined by
\begin{equation}\label{eqdefeu}
 e(u):=\left(\nabla u\right)_{\text{sym}}:=\frac{1}{2}(\nabla u+\nabla^Tu), 
\end{equation}
vanishes;  in the martensite if $e(u)\in\{A,B\}$, and (assuming sufficient regularity) grows quadratically close to these minima. The relevant constructions have strains which are 
not larger than a multiple of the order parameter, hence we do not expect the behavior of the energy at infinity to be important for the scaling results we shall derive, provided sufficient coercivity is present. For simplicity we restrict to quadratic energies, characterized as the squared distance from the energy wells. 
We use $|a|:=(\text{Tr } a^Ta)^{1/2}=(\sum a_{ij}^2)^{1/2}$ for the Euclidean norm of a matrix and $\dist(a,M):=\inf\{|a-m|:m\in M\}$ for the distance of a matrix to a set and consider the 
functional $J:W_\loc^{1,2}(\R^2;\R^2)\to\R\cup\{\infty\}$ given by
\begin{eqnarray}\label{eq:funcfull}
&&J(u):=\mu\int_{\R^2\setminus\Omega_{2L}}|e(u)|^2\text{ d} \mathcal{L}^2+\int_{\Omega_{2L}}\dist^2\left(e( u),K\right)\text{ d} \mathcal{L}^2+\eps|D^2u|(\Omega_{2L}).
\end{eqnarray}
Let us briefly explain the terms in the functional.
The first term in $J(u)$ represents the elastic energy of the surrounding austenite, where $\mu$ stands for the ratio of typical elastic moduli of austenite and martensite.
This term favors configurations whose gradients are approximately skew symmetric. The second term measures the elastic energy inside the martensitic nucleus, which vanishes on
\begin{eqnarray}\label{eq:K}
K:=\left\{\frac{1}{2}\left(\begin{array}{c c}
0&\theta\\
\theta&0
\end{array}\right),\ \frac{1}{2}\left(\begin{array}{c c}
0&-1+\theta\\
-1+\theta&0
\end{array}\right)\right\}{.}
\end{eqnarray}
The parameter $\theta\in(0,1/2]$ measures the compatibility between this majority martensitic variant and the surrounding austenite.
The so-measured compatibility has been found to play an important role in the control of the thermal hysteresis of the phase transition (see e.g. \cite{james-zhang:05,cui-et-al:06,zjm:09} and the references therein). 
Of particular interest is the almost compatible case $\theta\ll1$ which corresponds to particularly low hysteresis
 \cite{cui-et-al:06,zjm:09,zwicknagl:14}.
The third term in \eqref{eq:funcfull} 
is a singular perturbation that regularizes the nonconvex part of the functional. It 
prevents  too fine oscillations between the martensitic variants in $\Omega_{2L}$, and can be related to an interfacial energy, $\eps>0$ being a typical surface energy constant per unit length.
Whereas one could physically imagine that similar terms are present also in the austenitic phase, they are normally not included 
since the convex austenitic energy does not need regularization. Although we expect most of our results to carry over to a setting in which this term is extended to $\R^2$, for brevity we do not pursue this investigation here.
The energy of the austenite/martensite interface depends only on the shape of the inclusion, which is fixed here, and is hence irrelevant for the present purposes.
We denote by $D^2u$ the second distributional derivative of $u$. If it is a measure, then we denote by
$|D^2u|(\Omega_{2L})$ {the} total variation of $D^2u$, otherwise we set
$|D^2u|(\Omega_{2L}):=\infty$.
Existence of minimizers can be readily established by the direct method of the calculus of variations and will not be discussed explicitly, as it is not important for the study of the scaling of the energy.
\\
To determine exact minimizers of functionals like \eqref{eq:funcfull} is generally not possible, and we follow the strategy to determine the scaling regimes of the minima in terms of the problem parameters $L$, $\mu$, $\theta$ and $\eps$.
{We remark that $L\to\infty$ corresponds to the long-inclusion limit, which is relevant due to the compatibility condition; $\theta\to0$ is the almost-compatible limit, which is the one of low-hysteresis materials; $\eps\to0$ is the large-body limit, in which complex structures arise.}

Our main result is the following scaling law for the minimal energy.
For an overview over the individual regimes we refer to Section {\ref{sec:tabulars}} below. Explicit constructions are given in Section {\ref{sec:ub}}.
{We remark that the same result {holds for}  nonlinear energy densities $W_A(e(u))$, $W_M(e(u))$ with $\frac1c|a|^2\le W_A(a)\le c |a|^2$ and $\frac1c \dist^2(a,K)\le W_M(a)\le c \dist^2(a,K)$ for all $a\in\R^{2\times 2}$.}

\begin{thrm}\label{th:main}

There exists a constant $c>0$ such that for all {$\mu>0$, $\varepsilon>0$, $\theta\in (0,1/2]$, and $L\geq1/2$}
\begin{align*}
  &\frac{1}{c}\mathcal{I}({\mu,\eps,\theta,L})
   \leq \min_u J(u) 
 \leq  c\mathcal{I}({\mu,\eps,\theta,L}),
\end{align*}
where {$J$ was defined in (\ref{eq:Omega}), (\ref{eqdefeu}), (\ref{eq:funcfull}), (\ref{eq:K}) and}
\begin{align*}
{\mathcal{I}({\mu,\eps,\theta,L})}:= \min\Big{\{ }&
 \theta^2 L,  & {\text{(constant)}}\\
 &\mu \theta^2 \ln (3+L), & {\text{(affine)}} \\
 &\mu \theta^2 \ln (3+ \frac L  \mu ) + \varepsilon\theta ,&{\text{(linear interpolation)}}\\
 &{\mu \theta^2 \ln (3+ \frac { \varepsilon {L}}{\mu \theta^{2}}) 
 +\mu\theta^2\ln (3+\frac{\eps}{\mu^2\theta^2})
 + \varepsilon^{1/2}\theta^{3/2},}   & {\text{(single truncated branching)}}\\
 &{\mu \theta^2 \ln (3+ \frac { \varepsilon {L}}{\mu \theta^{2}}) 
 +\mu\theta^2\ln (3+\frac{\theta}{\mu})}
,  & {\text{(corner laminate)}}\\
 &\varepsilon^{2/3}\theta^{2/3} L^{1/3} +\varepsilon L,  & {\text{(branching)}}\\
 &\mu^{1/2}\varepsilon^{1/2}\theta L^{1/2}  (\ln (3+ \frac 1 {\theta^{2}}))^{1/2}+\varepsilon L,  & {\text{(laminate)}} \\
 &\mu^{1/2}\varepsilon^{1/2}\theta L^{1/2} (\ln (3+ \frac{\varepsilon}{ \mu^{3}\theta^{2}{L}}))^{1/2}+\varepsilon L
 \Big{\}}  & {\text{(two-scale branching).}}
 \end{align*}
\end{thrm}
{
\begin{proof}
 {The upper bound follows from  Theorem \ref{th:upperbound}, and the lower bound follows from Theorem \ref{th:lowerbound}.}

\end{proof}
}
The proof of Theorem \ref{th:main} is split into two main steps. 
In Section \ref{sec:upperbound}
we combine constructions from the literature with some new ones for the upper bound. 
The ansatz-free lower bound is proven in Section \ref{sec:lowerbound}.

{We remark that the constant $3$ inside the terms of the form $\ln(3+x)$ is, to a certain degree, arbitrary. We choose 3 so that $\ln (3+x)\ge 1$ for all $x\ge0$. This simplifies some estimates in the proofs. The constant 3 could be replaced by 2 or by any number larger than 1, changing correspondingly the constant $c$ in the statement.}\\
{We point out that in contrast to previous works on a single austenite/martensite interface (see e.g. \cite{conti:06,zwicknagl:14,conti-zwicknagl:16}) there is no relevant regime which corresponds to a construction with a single laminate near the left and right boundaries of the nucleus.}

\subsection{Comparison to the literature and new contributions}
 We point out that the study of microstructures in shape memory alloys by means of the Calculus of Variations has a long history, and we generalize and build on several earlier works that we will briefly discuss to point out our new contributions. Our proof uses techniques developed in the study of scalar-valued models taking into account all problem parameters on the one hand,  and of vectorial models in specific parameter regimes on the other hand, combined with several new arguments. As will be outlined in more detail below, the latter include in particular
 \begin{itemize}
 \item the careful treatment of the elastic energy in the austenite part {which completely surrounds a} martensitic inclusion{. This} introduces new difficulties in both the upper and the lower bound of Theorem \ref{th:main}; and 
 \item the explicit use of the full geometrically linearized energy instead of the scalar simplification, {which requires e.g.} a $BD$-type slicing argument in the proof of the lower bound.
 \end{itemize}
  In the '90s, Kohn \& M\"uller proposed a reduced scalar-valued model for the formation of microstructures near interfaces  between austenite and twinned martensite \cite{kohn-mueller:92,kohn-mueller:94}. These models are by now well understood in terms of scaling of the minimal energy and more quantitative properties of minimizers in specific regimes ({see} e.g. \cite{conti:00,conti:06,diermeier:10,zwicknagl:14,cdz:15,diermeier:16,conti-zwicknagl:16}). Roughly speaking, depending on the problem parameters, minimizers are expected to be uniform, or show laminated structures, or branched patterns, where the different martensitic variants finely mix close to the interface. Compared to the setting we consider here, in these earlier works there were two main simplifications: \\
  First, only one component of the displacement has been taken into account, i.e., only functions $u=(u_1,u_2)$ with $u_2=0$ are considered, which makes the problem scalar (similarly for $u_1=0$). On a more technical level, the {first two terms of the} functional \eqref{eq:funcfull} then in particular provide control on the full gradient of the displacements. In our more general setting, only the symmetric part of the gradient is {directly} controlled by the functional, and this introduces several additional difficulties in the proof of the lower bound (see also the discussion of vectorial models below). Nevertheless, the constructions we use to prove the upper bound here, are in fact scalar valued. {Some of them build upon constructions introduced in the above mentioned references, 
  but others are new, as for example the one for the corner laminate and the single truncated branching, see the proof of Theorem \ref{th:upperbound}  in}
  Section \ref{sec:upperbound}. \\
  Second, in the above references, only one austenite/martensite interface is considered. That is, in the elastic energy of the austenite part (the first term in \eqref{eq:funcfull}), only the contribution from $(-\infty,0)\times(0,1)$ is taken into account. If we restricted the energy in \eqref{eq:funcfull} to this strip, there would be configurations with vanishing total energy, e.g.,
\[u(x)=\begin{cases}
0,&\text{\ if \ }x\in (-\infty,0]\times (0,1),\\
(0,\theta x_1),&\text{\ if\ }x\in(0,2L) \times (0,1).
\end{cases} \]
The functional \eqref{eq:funcfull} is more nonlocal than the scalar valued models in the sense that interactions between the traces at the upper and lower boundaries of $\Omega_{2L}$ (captured by the elastic energy of the austenite part) make it sometimes more favorable to pay elastic energy inside $\Omega_{2L}$ to release elastic energy in the austenite part. Let us consider a typical example. Deep in $\Omega_{2L}$, in the simplified setting one expects the affine configuration $u_1(x)=\theta x_2$, which has zero energy in $\Omega_{2L}$. In our setting instead, by Rellich's trace theorem, this configuration bears elastic energy in the austenite part. It therefore competes with a single laminate, which requires only surface energy in the interior of $\Omega_{2L}$.
The proof of the lower bound correspondingly needs a treatment of the interplay of the energy in the austenite and the martensite on many different scales. This is done by line integrals, inspired by the arguments used for proving Korn-Poincar\'e inequalities in $BD$, 
see for example Lemma \ref{lemmaystin}
and Lemma \ref{lem:costsu2inA} below.
One important ingredient is a separate treatment of the austenite part, where one controls
three of the four components of $\nabla u$ (but with a coefficient $\mu$), and the martensite part, where only the two diagonal entries are 
controlled independently of the variant.
\\[1mm]
On the other hand, we extend techniques developed in \cite{chan:13,chan-conti:14,chan-conti:14-1}. In these works, the geometrically nonlinear analogue to \eqref{eq:funcfull} has been considered for the case $\theta=\frac{1}{2}$ and hard austenite $\mu=\infty$. Some of their techniques, in particular related to localization in the proof of the lower bound, have been adopted to the geometrically linear setting and refined in two of the authors' Masters's theses \cite{diermeier:13,melching:15} on which  we build here. A main difficulty in our setting compared to those works (in addition to the {`non-locality'} due to the elastic energy in the austenite part discussed above) lies in the treatment of small $\theta$. As pointed out in \cite{conti-zwicknagl:16}, such localization techniques are not sufficient to obtain the precise scaling of laminated structures since the logarithmic corrections require a rather precise understanding of the geometry of the set in which 
the minority variant is active. Here, a careful $BD$-type slicing argument for almost diagonal slices allows us to combine the techniques from the studies of vectorial models with techniques developed to treat small volume fractions in the scalar valued case. {In particular, test functions need to be obtained that reproduce the fine-scale structure of the martensite and have a controlled behavior at the boundaries, see Lemma \ref{lem:thetasmallbranching} where for example separate test functions on the top and bottom boundaries need to be constructed  for $\theta\ge\mu$ (the {parameter range} with corner laminates) and $\theta\le \mu$ (without corner laminates){, and also} Lemma \ref{lemmabdryln} where the corner logarithm is treated by a test function on the boundary.
At the same time the ``horizontal'' interpolation between different variants needs to be localized in order to capture the optimal power of $\theta$ (Lemma \ref{lemmainterpolationestim}).}\\
Analytical results on microstructures for related three-dimensional models based on geometrically linearized elasticity functionals were obtained in \cite{CapellaOtto2009,CapellaOtto2012,TS} for a cubic-to-tetragonal phase transition and in \cite{rueland:16,rueland:16_2} for a cubic-to-orthorhombic transition. As in our case, the focus there lies on planar austenite/martensite interfaces. Some results that take into account also the volume dependence of the energy of a martensitic inclusion (by penalizing the area of the austenite/martensite interfaces) and the resulting optimal shapes of nuclei (which typically differ significantly from a rectangle) were given in \cite{knuepfer-kohn:11} for a two-well potential  and in \cite{knuepfer-kohn-otto:13,bella-goldman:15} for the cubic-to-tetragonal transition in whole space and domains with generic corners, respectively. We note that these works predict a different scaling behavior. We hope that a precise understanding of microstructures in a fixed domain as derived {here}
provides also a step towards a better understanding of the full nucleation problem.

\subsection{Notation}
Throughout the text, we denote by $c$ positive constants that may change from expression to expression, we use
 $x\lesssim y$ to state that there is $c>0$ such that $x\le cy$. We use capitalized letters and $c_i$ with indices $i\in\N$ to denote specific fixed constants that will not be changed throughout the text.\\
For a measurable set $A\subset \R^d$ {with $\calL^d(A)\neq 0$} and a function $w\in L^1(A)$, we denote the average by $\langle w\rangle_{A}:=\calL^d(A)^{-1}\int_Aw\text{\,d}\calL^d$. 

{{\bf Energy.}}

Let us first fix a notation for  the function space
\[
 \mathcal{X}:=\left\{ u\in W_{\text{loc}}^{1,2}(\R^2,\R^2)\,:\,\partial_i u_j \in BV(\Omega_{2L}) \text{ for } i,j \in \{1,2\},\, \nabla u\in L^2(\R^2,\R^{2\times 2})\right\}
\]
on which the energy is finite. In the proofs it will be convenient to consider a slightly modified energy functional where the symmetrized gradient in the elastic energy of the austenite part is replaced by the full gradient. This does not change the scaling regimes of the minimal energy due to 
Korn's inequality {since} 
the constant in Korn's inequality in $\R^2 \setminus \Omega_{2L}$ can be chosen independently of $L$, i.e., there is a constant $C_K>0$ independent of $L$ such that 
\begin{eqnarray}\label{eq:korn}
\min_{A\in\text{Skew}(2)}\|\nabla u-A\|_{L^2(\R^2\setminus \Omega_{2L})}\leq C_K\|e(u)\|_{L^2(\R^2\setminus \Omega_{2L})} \text{\ for all\ }u\in W_{\text{loc}}^{1,2}(\R^2,\R^2). 
\end{eqnarray}
To see this, we use the decomposition
\[\R^2\setminus\Omega_{2L}=[(-\infty,0)\times\R]\cup[\R\times(1,\infty)]\cup[(2L,\infty)\times\R]\cup[\R\times(-\infty,0)]. \]
Each one of the sets on the right-hand side is a half-space, and therefore on each of them, a Korn's inequality holds with a constant independent of $L$ (see \cite{kondratev-oleinik:88}). Further, every set intersects another one on a set of infinite measure{. Hence, for any function $u$, the Korn's inequality in each one of the four parts holds with the same skew symmetric matrix $A$, and therefore, \eqref{eq:korn} holds.}
We may therefore without changing the qualitative scaling behavior replace the symmetrized gradient in the first term in (\ref{eq:funcfull}) by the full gradient and define
\[
I(u):=\mu\int_{\R^2\setminus\Omega_{2L}}|\nabla u|^2\text{ d} \mathcal{L}^2+\int_{\Omega_{2L}} \min\left\{ |e(u)-\theta e_1 \odot e_2|^2,|e(u)+(1-\theta) e_1 \odot e_2|^2\right\}\text{ d} \mathcal{L}^2+\eps|D^2u|(\Omega_{2L})
\]
where 
\[e_1\odot e_2 := (e_1\otimes e_2)_{\text{sym}}.\]
Sometimes it will be useful to consider the energy only on parts of the domain. For any Borel set $A\subset \R^2$ we define 
\begin{align*}
 I_{A}(u)&:=\mu\int_{A\setminus\Omega_{2L}}|\nabla u|^2\text{ d} \mathcal{L}^2+\\
 &+\int_{A\cap\Omega_{2L}} \min\left\{ |e(u)-\theta e_1 \odot e_2|^2,|e(u)+(1-\theta) e_1 \odot e_2|^2\right\}\text{ d} \mathcal{L}^2+\eps|D^2u|(A\cap \Omega_{2L}),\\
  I^\intt(u)&:= I_{\Omega_{2L}}(u) \qquad  \text{ and } \qquad
  I^\ext(u):= I_{\R^2\setminus \Omega_{2L}}(u) .
\end{align*}
{{\bf The $H^{1/2}$-norm.}}\\
It has proven useful to interpret the energetic contribution in the austenite region as a trace norm at the austenite/martensite interface. 
For $\rho >0$ and $u_0 \in L^2((0,\rho))$ we define the $H^{1/2}$-seminorm by
\[
 [u_0]_{H^{1/2}((0,\rho))}^2:= \inf \Big\{ \int_{-\infty}^0\int_0^{\rho} |\nabla v (x_1,x_2)|^2 \text{ d}x_2\text{ d}x_1 \,:\, v(0,x_2)=u_0(x_2), v\in W^{1,2}_{\text{loc}}((-\infty,0)\times(0,\rho)) \Big\}.
\]
The subspace of $L^2((0,\rho))$ on which this seminorm is finite is called $H^{1/2}((0,\rho))$. 
%
{We state a variant of } Lemma 4.1 from \cite{conti-zwicknagl:16}.
\begin{lmm}\label{lem:interpolH12}
{Let $\omega\subset\subset(0,1)$.} Then there is $c=c(\omega)>0$  such that for all $v \in H^{1/2}((0,1))$ and $\psi \in H^{1/2}((0,1))\cap H^1((0,1))$ {with $\supp \psi\subset\omega$ one has}
 \[
  \int_0^1 v(t) \psi'(t) \text{ d} t \leq c [v]_{H^{1/2}((0,1))} [\psi]_{H^{1/2}((0,1))} .
 \]
\end{lmm}
{{\bf Remark.} Lemma 4.1 from \cite{conti-zwicknagl:16} incorrectly does not state
that $c$ depends on $\omega$ (or on $\supp\psi$). This is not relevant for the usage in  
 \cite{conti-zwicknagl:16}, since the test function $\psi$ can be constructed to be supported in $(1/12, 11/12)$.}
\begin{proof}
If $\supp v\subset\subset(0,1)$, then the assertion is readily proven by Fourier series,
\begin{equation*}
 |\sum_k \hat v_k^* i k \psi_k | \le (\sum_k |k|\, |v_k|^2)^{1/2} (\sum_k |k|\, |\psi_k|^2)^{1/2}.
\end{equation*}
Otherwise, one fixes $\varphi\in C^\infty_c((0,1))$ with $\varphi=1$ on $\omega$, applies 
the Poincar\'e estimate for $H^{1/2}$ to obtain $\|v-v_0\|_{L^2((0,1))}\le  [v]_{H^{1/2}((0,1))}$ for some $v_0\in\R$,
and {then applies} 
the previous assertion to $\tilde v:=\varphi (v-v_0)$ and $\psi$.
\end{proof}

In our proof of the lower bound, we shall use a related estimate given in the next lemma.

\newcommand{\myvarphi}{v}
\begin{lmm}\label{lemmah12app} Let $\omega\subset\R^2$ be a bounded Lipschitz set, $\Psi\in\Lip(\bar\omega)$, $\myvarphi\in W^{1,2}(\omega)$. Then
 \begin{equation*}
\left|  \int_{\partial\omega} \myvarphi\partial_\tau \Psi  \dcalH^1 \right| \le \|\nabla\Psi\|_{L^2(\omega)} \|\nabla\myvarphi\|_{L^2(\omega)},
 \end{equation*}
 where $\partial_\tau$ denotes the tangential derivative and $v$ is identified with its trace on $\partial\omega$.
\end{lmm}
\begin{proof}
Assume first that $\Psi\in C^2(\bar\omega)$. We define $f\in W^{1,2}(\omega;\R^2)$ by
 $f:=\myvarphi\,(\nabla\Psi)^\perp =(-\myvarphi\partial_2\Psi, \myvarphi\partial_1\Psi)$ and observe that $\Div f= \nabla\Psi\times \nabla\myvarphi{=\partial_2v\partial_1\Psi-\partial_1v\partial_2\Psi}$. Then
 \begin{equation*}
  \left|\int_{\partial\omega} \myvarphi \, \partial_\tau \Psi \dcalH^1\right|
  =\left|\int_{\partial\omega} f\cdot \nu\dcalH^1\right|
  =\left|\int_{\omega} \Div f\dxy\right|
 =\left|\int_{\omega} \nabla\Psi\times \nabla\myvarphi\dxy\right|
 \end{equation*}
implies the result. 
In the general case, we choose a sequence of smooth functions $\Psi_\eps\in C^\infty(\R^2)$ such that $\|\nabla \Psi_\eps\|_{L^\infty(\omega)}
\le \|\nabla \Psi\|_{L^\infty(\omega)}$, with $\Psi_\eps$ converging uniformly and strongly in $W^{1,2}(\omega)$ to $\Psi$. 
Then $\partial_\tau \Psi_\eps$ converges weakly{-$\ast$} to $\partial_\tau \Psi$ in $L^\infty(\partial\omega)$, therefore
 \begin{equation*}
  \left|\int_{\partial\omega} \myvarphi \, \partial_\tau \Psi \dcalH^1\right|
=\left|\lim_{\eps\to0}   \int_{\partial\omega} \myvarphi \, \partial_\tau \Psi_\eps  \dcalH^1\right|
  =\left|\lim_{\eps\to0}   \int_{\omega} \nabla\Psi_\eps\times \nabla\myvarphi\dxy\right|
  =\left|\int_{\omega} \nabla\Psi\times \nabla\myvarphi\dxy\right|.
 \end{equation*}
\end{proof}

We will moreover use the following variant of Lemma 2 in \cite{conti:06} that has also been used in the proof of Theorem 1 in \cite{zwicknagl:14}:
\begin{lmm}\label{lem:interpolSC}
 There is $c>0$ such that for all $\rho>0$, $v\in H^{1/2}((0,\rho))$,  $a\in\R$, and $b\in \R$ there holds
 \[
  c\rho^2 a^2 \leq  |a|\int_0^\rho |v(y)-a y - b|\text{ d}y + [v]^2_{H^{1/2}((0,\rho))}.
 \]

\end{lmm}

{\bf{BD-type slicing.}}

For any  $u:\R^2\rightarrow \R^2$, $x_1\in \R$, and  $\xi\in \R^2$ with $\xi_1\neq 0\neq \xi_2$, we define  the one-dimensional, scalar-valued function on the slice $(x_1,0)+\R\xi$ as
\begin{align}\label{eq:defuxi}
u_{x_1}^{\xi}(s):= \frac 1 {\xi_1\xi_2}u((x_1,0) +s\xi)\cdot \xi.
\end{align}
Notice that this definition is {motivated by the characterization of $BD$ via suitable one-dimensional sections} as introduced by Ambrosio, Coscia and Dal Maso  {\cite[Proposition 3.2]{ambrosio-et-al:97}}. {For convenience, our definition} differs {from the one given in the above reference} by the prefactor $\frac 1 {\xi_1\xi_2}$.

If $u\in W_{\text{loc}}^{1,2}(\R^2,\R^2)$ we have $u_{x_1}^\xi\in W_{\text{loc}}^{1,2}(\R)$ for almost every $x_1\in  \R$, and 
\begin{align}\label{eq:estduxiv}
{u_{x_1}^\xi}'(s) 
&= \left(\frac {\xi_1}{\xi_2} \partial_1 u_1 + \partial_2 u_1+\partial_1 u_2 +\frac{\xi_2}{\xi_1} \partial_2 u_2\right)((x_1,0)+s\xi).
\end{align}
Throughout this work we will fix the direction  $\xi:=(\frac 14,1) $ and define for some $x_1\in (0,2L-\xi_1)$ the almost diagonal segment with base point $(x_1,0)$  as
\begin{eqnarray}\label{eq:sliceDelta}
\Delta^\xi_{x_1}:=\left\{(x_1,0)+s\xi\,:\, s\in (0,1)\right\}. 
\end{eqnarray}
With a small abuse of notation we shall write, for $f:\Delta^\xi_{x_1}\to\R$,
\begin{equation}\label{eqdeffl2delta}
 \|f\|_{L^2(\Delta^\xi_{x_1})}^2:=\int_0^1 |f|^2((x_1,0)+s\xi) \text{ d}s,
\end{equation}
and the same for $L^1$. This definition differs from the usual one, in which one integrates with respect to $\calH^1$, by a factor of $\sqrt{17/16}$, which is irrelevant for our argument but would make notation cumbersome.


For any $a\in\R$, almost every $x_1\in (0,2L-\xi_1)$ and almost every $s \in (0,1)$
we compute 
\begin{align}\label{eqestduxiv1}
| {u_{x_1}^\xi}'(s)-a | &=
\Big|\frac{1}{\xi_1\xi_2}\sum_{ij}\xi_i\xi_j( e(u)-ae_1\odot e_2)_{ij}   \Big|((x_1,0)+s\xi) \nonumber\\
&\le \frac{|\xi|^2}{|\xi_1\xi_2|} |e(u)-ae_1\odot e_2|((x_1,0)+s\xi)\le 5 |e(u)-ae_1\odot e_2|((x_1,0)+s\xi)
\end{align}
where in the last step we used the specific choice $\xi=(1/4,1)$.
This implies in particular 
\begin{align}\label{eq:estduxi}
&\min\left\{|{u_{x_1}^\xi}'-\theta|^2,|{u_{x_1}^\xi}'+(1-\theta)|^2\right\}(s) 
 \leq 25 \min \left\{|e(u)-\theta e_1 \odot e_2|^2,|e(u)+(1-\theta) e_1 \odot e_2|^2\right\}((x_1,0)+s\xi).  
 \end{align}
{We remark that vertical slices cannot be used {to obtain similar estimates}, since $\partial_2 u_2$ does not distinguish between the two variants and $\partial_2 u_1$ cannot be controlled without an independent estimate on $\partial_1 u_2$.} 
\clearpage
 
\section{Upper bound}\label{sec:upperbound}
In this section we will prove the upper bound, i.e., the second inequality in Theorem \ref{th:main}. For that, we provide explicit constructions for the different energy {scaling regimes} in Subsection \ref{sec:ub}. In Subsection \ref{sec:tabulars} we give an overview {over typical} parameter {ranges} to illustrate that indeed all scalings are attained.
\subsection{Explicit constructions}\label{sec:ub}
Let us point out that all our constructions will be scalar valued. 
{W}e define for every $u\in W^{1,2}_{\text{loc}}(\R^2,\R)$ with $\nabla u \in BV(\Omega,\R^2)$
\[
 {E}(u):={I((u,0))}
 \leq  \int_{\Omega_{2L}}\min\left\{|\nabla u - \theta e_2|^2, |\nabla u +(1-\theta)e_2|^2\right\} \text{ d} \mathcal{L}^2 + \mu \int_{\R^2\setminus\Omega_{2L}}|\nabla u|^2 \text{ d} \mathcal{L}^2 +\varepsilon |D^2u|(\Omega_{2L})
\]
and correspondingly ${E}_A(u):=I_A({(u,0)})$, ${E}^\intt$, and ${E}^\ext$.
Some of the test functions we consider below are taken from the literature, 
{some constructions have to be modified, and some are new.} 
We shall use only constructions that are symmetric with respect to the axis $\{x_1=L\}$, {working explicitly in $(-\infty,L]\times\R$ and then extending each construction by symmetry.}
This introduces an additional term $|D^2u|(\{L\}\times(0,1))$. In some cases, it vanishes (constant, affine, linear interpolation). In the other cases, we use that 
 the relevant gradients are bounded, 
 and hence
the term ${\eps} |D^2u|(\{L\}\times(0,1))$ is bounded by $\eps$. Then this term can be incorporated in the regimes using $\eps\leq\mu\theta^2$, see the proof of Theorem \ref{th:upperbound}. We will therefore not explicitly mention this term in the discussion of the constructions below.
  We shall use the following short-hand notation: For $a<b$, we set  
 \begin{eqnarray}
 \label{eq:iota}
 \iota_{a,b}:[a,b]\rightarrow\R , \quad{\iota_{a,b}}(t):=\frac{t-a}{b-a},
 \end{eqnarray}
 i.e., $\iota_{a,b}$ is the affine function with $\iota_{a,b}(a)=0$, $\iota_{a,b}(b)=1$.\\
 We start with auxiliary lemmata to estimate the energy contribution from $\R^2\setminus \Omega_{2L}$. 
\begin{lmm}
\label{lem:lemout}
Let $\bar{L}\in[1/2,\infty)$, $1\leq\alpha<\beta\leq \bar{L}$. Then there exists $u_{\alpha,\beta,\bar{L}}\in
W^{1,2}_\loc{\left(\left((-\infty,\bar L]\times \R\right)\setminus\Omega_{\bar{L}}\right)}\cap C^0{\left(\left((-\infty,\bar L]\times \R\right)\setminus\Omega_{\bar{L}}\right)}$ such that
\begin{itemize}
\item[(i)] on the lower boundary $u_{\alpha,\beta,\bar{L}}(x_1,0)=0$ for all $x_1\in[0,\bar{L}]$;
\item[(ii)] on the upper boundary
\[u_{\alpha,\beta,\bar{L}}(x_1,1)=\begin{cases}
0,&\text{\qquad if\ }x_1\in[0,\alpha],\\
\theta\iota_{\alpha,\beta}(x_1),&\text{\qquad if\ }x_1\in(\alpha,\beta],\\
\theta,&\text{\qquad if\ }x_1\in(\beta,\bar{L}];
\end{cases} \]
\item[(iii)] on the interface $u_{\alpha,\beta,\bar{L}}(0,x_2)=0$  {for all $x_2\in[0,1]$;}
\item[(iv)] and there is a constant $c>0$ independent of $\alpha$, $\beta$, $\bar{L}$ and $\theta$ such that  
\[\int_{\R^2\setminus\Omega_{2\bar{L}}}\left|\nabla u_{\alpha,\beta,\bar{L}}\right|^2\dcalL^2\leq c\theta^2\left(\ln\left(3+\frac{\bar{L}}{\alpha}\right)+\frac{\beta+\alpha}{\beta-\alpha}\right). \]
\end{itemize}
\end{lmm}
\begin{proof}
{We use polar coordinates, denoting by $\phi(x)\in(0,2\pi)$ and $r(x)\in[0,\infty)$ the coordinates of
$x=(x_1,x_2)\in\R^2\setminus\left((0,\infty)\times\{0\}\right)$ so 
that}
\[x_1=r(x)\cos(\phi(x)),\qquad x_2=r(x)\sin(\phi(x)). \]
We define $\hat{f}_{\alpha,\beta}:(0,2\pi)\times[0,\infty) \rightarrow \R$ by
\[
{\hat{ f}_{\alpha,\beta}(\phi,r)}:= 
 \begin{cases}
  0,							&\text{if } r\in [0,\alpha],\\
    (1-\frac {\phi}{2\pi} )\iota_{\alpha,\beta}(r)\theta,	&\text{if }r\in (\alpha, \beta], \\
  (1- \frac {\phi}{2\pi} )\theta 			,	&\text{if }r \in (\beta, {\bar{L}}],\\
 \frac {{\bar{L}}^{1/2}}{r^{1/2}}  (1-\frac {\phi}{2\pi}) \theta,	&\text{if }r \in ({\bar{L}},\infty),\\
 \end{cases}
\]
and consider the  transformation $T:\R^2\rightarrow \R^2$ given by
\[
T(x_1,x_2):=
 \begin{cases}
  (x_1,x_2),			&\text{\qquad if\ } x_2\le 0,\\
  (x_1,0),			&\text{\qquad if\ }0<x_2\le 1,\\
  (x_1,x_2-1),		& \text{\qquad if\ }x_2> 1. 
   \end{cases}
\]
We set for $x\in\left((-\infty,\bar{L}]\times\R\right)\setminus\overline{\Omega_{2\bar{L}}}$
\[u_{\alpha,\beta,\bar{L}}(x):=\hat{f}_{\alpha,\beta}\left( \phi(T(x)),r(T(x))\right). \]
Note that for $x_1>0$ and $x_2\searrow 1$, we have $r(T(x_1,x_2))\to x_1$, and $\phi(T(x_1,x_2))\searrow 0$. Similarly, for $x_1>0$ and $x_2\nearrow0$, we have 
$r(T(x_1,x_2))\to x_1$ and $\phi(T(x_1,x_2)\to2\pi$. Finally, for $x_2\in(0,1)$ and $x_1\nearrow 0$, we have {$r(T(x_1,x_2))\to 0$}. Therefore  $u_{\alpha,\beta,\bar{L}}$ {has a continuous extension} to $\left((-\infty,\bar{L}]\times\R\right)\setminus\Omega_{\bar{L}}$ {and it satisfies} (i), (ii) and (iii). 
It remains to verify (iv). From the definition of $T$ we obtain, with
 $f_{\alpha,\beta}( x):=\hat{f}(\phi(x),r(x))$,
\[
\|\nabla u_{\alpha,\beta,\bar{L}}\|^2_{{{\bar{L}}}^2({(}(-\infty,\bar{L}) \times \R{)} \setminus \Omega_{\bar L})} = \|\nabla f_{\alpha,\beta}\|^2_{L^2({(}(-\infty, \bar{L}) \times \R{)} \setminus {(}[0,\bar{L}]\times \{0\}{)}) } +\|\partial_1 f_{\alpha,\beta}(\cdot,0)\|^2_{L^2((-\infty,0))} 
\]
where $\partial_1$ refers to the usual derivative in direction $e_1$.
Using polar coordinates we compute
\begin{align*}
 \|\nabla f_{\alpha,\beta}\|^2_{L^2({B_{\bar{L}}(0)}\setminus{(} [0,\bar{L}]\times\{0\}{)})}
 &= \int_0^{{\bar{L}}}\int_0^{2\pi}  \frac 1 r |\partial_\phi{\hat f}_{\alpha,\beta}|^2 + r |\partial_r {\hat f}_{\alpha,\beta}|^2 \text{ d}\phi \text{ d}r  
 \leq c \theta^2\left( \int_{\alpha}^{{{\bar{L}}}}  \frac 1 {r}\text{ d}r + \int_\alpha^\beta \frac {r}{(\beta-\alpha)^2} \text {d}r\right) \\
 & \le c {\theta^2 }\left(  \ln {\frac{{\bar{L}}}\alpha} + \frac {\beta +\alpha}{\beta - \alpha}\right)
\end{align*}
and
\begin{align*}
  \|\nabla f_{\alpha,\beta}\|^2_{L^2({(}(-\infty,\bar{L})\times \R{)}\setminus B_{\bar{L}}(0))}
 &\leq \int_{ \bar{L}}^\infty \int_0^{2\pi}  \frac 1 r |\partial_\phi {\hat f}_{\alpha,\beta}|^2 + r |\partial_r \hat f_{\alpha,\beta}|^2 \text{ d}\phi \text{ d}r  
 \leq c \theta^2 \int_{\bar{L}}^\infty \frac{\bar{L}}{r^{2}}\text{ d}r
 = c\theta^2.
\end{align*}
We also estimate, 
\[
 \|\partial_1 f_{\alpha,\beta}(\cdot,0)\|^2_{L^2((-\infty,0))} = \frac{\theta^2}{4}\int_{\alpha}^\beta \frac 1 {(\beta-\alpha)^2} \text{ d}x_1+ \frac {\theta^2}{16}\int_{\bar{L}}^\infty \frac {\bar{L}}{x_1^3}\text{ d}x_1 
 {=\frac{\theta^2/4}{\beta-\alpha}+\frac{\theta^2/{32}}{\bar{L}}\le \frac{\theta^2}{\beta-\alpha}.}
\]
{Recalling that $\alpha\ge 1$, we conclude}
\[
 \|\nabla u_{\alpha,\beta,\bar{L}}\|^2_{L^2({(}(-\infty,\bar{L}) \times \R) \setminus ((0,\bar{L})\times (0,1)){)}} \leq  c \theta^2 \left( \ln {(3+\frac{\bar{L}}\alpha)} + \frac {\beta +\alpha}{\beta - \alpha}\right) .
\]
\end{proof}
{
\begin{rmrk}
We will often use that the upper bound in Lemma \ref{lem:lemout}(iv) is monotonically increasing in $\bar{L}$ and decreasing in $\beta$.
\end{rmrk}
}
The following lemma has been used in \cite{conti-zwicknagl:16} without being explicitly stated. We refer to Fig.~\ref{figlamoutside} for a sketch.
\begin{lmm}
\label{lem:lamoutside}
Let $N\in\N$, $N\ge 1$, $h\in(0,\frac{1}{2}]$, $\theta\in[0,\frac{1}{2}]$. Then there exists $v_{N,h}\in W^{1,2}_\loc\left((-\infty,0)\times(0,1)\right)\cap{C^0}\left((-\infty,0]\times[0,1]\right)$ such that
\begin{itemize}
\item[(i)] $v_{N,h}(x_1,x_2)=v_{N,h}(x_1, x_2+\frac{1}{N})$ for all $x_2\in[0,1-\frac{1}{N}]$ and all $x_1{\in(-\infty,0]}$;
\item[(ii)] for $x_1=0$ we have
\[ v_{N,h}(0,x_2)=\begin{cases}
\theta x_2,&\text{\qquad if \ }x_2\in[0,\frac{1-h}{N}],\\
{\theta\frac{1-h}{h} (\frac1N-x_2)},&\text{\qquad if \ }x_2\in(\frac{1-h}{N},\frac{1}{N}];
\end{cases}
 \]
 \item[(iii)] for all $x_1\in(-\infty,0]$, we have
 $v_{N,h}(x_1,0)=v_{N,h}(x_1,1)=0$;
 \item[(iv)] {for all $x_1\in(-\infty,-\frac1N]$ and all $x_2$, we have
 $v_{N,h}(x_1,x_2)=0$;}
 \item[(v)] and there exists $c>0$ independent of $N$, {$\theta$} and $h$ such that
 \[\int_{(-\infty,0)\times(0,1)}\left|\nabla v_{N,h}\right|^2\,\dcalL^2\leq c\frac{\theta^2}{N}\ln\left(3+\frac{1}{h}\right). \]
\end{itemize}
\end{lmm} 
\begin{figure}
 \begin{center}
  \includegraphics[width=7cm]{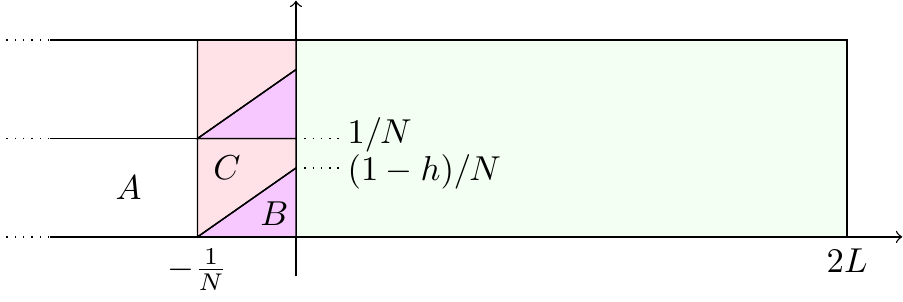}
 \end{center}
\caption{Construction in the proof of Lemma \ref{lem:lamoutside} for $N=2$.}
\label{figlamoutside}
\end{figure}
\begin{proof}
{We write $v$ for $v_{N,h}$ and set}
\[v(x):=\begin{cases}
0,&\text{\qquad if\ } {x\in A:=(-\infty,-\frac{1}{N}]\times[0,\frac1N)},\\
\theta x_2,&\text{\qquad if\ }{x\in B:=\{x_1\in(-\frac{1}{N},0], 0\leq x_2\leq (1-h)(x_1+\frac{1}{N})\}},\\
\frac{(1-Nx_2)\theta(1-h)(Nx_1+1)}{N\left(h-(1-h)Nx_1\right)},&\text{\qquad if\ }
{x\in C:=\{
x_1\in(-\frac{1}{N},0], (1-h)(x_1+\frac{1}{N})< x_2<\frac{1}{N}\}},
\end{cases} \]
and extend it periodically in $x_2$ to $(-\infty,0]\times[0,1]$ as stated in (i) {(see Fig.~\ref{figlamoutside})}.
One easily checks that $v$ is continuous and satisfies (ii), (iii) { and (iv)}. To show (v), we 
{first work in $(-\frac1N,0]\times[0,\frac1N){=B\cup C}$. In region $C$ we have
$(1-h)(Nx_1+1)\le Nx_2<1$, and therefore $0< 1-Nx_2\le h-(1-h)Nx_1$, so that
\begin{equation*}
 |\nabla v|(x)\le \theta(1-h)\Big( \frac{2}{h-(1-h)Nx_1}+
  \frac{1-Nx_2}{(h-(1-h)Nx_1)^2}\Big)
  \le \frac{3 \theta (1-h)}{h-(1-h)Nx_1} \text{ for all }x\in C.
\end{equation*}
Since $|\nabla v|=\theta$ in $B$ we obtain
\begin{equation*}
\begin{split}
 \int_{-1/N}^0\int_0^{1/N} |\nabla v|^2\dy\dx&\le
 \frac{\theta^2}{N^2} +9 \theta^2 \int_{-1/N}^0\int_{(1-h)(x_1+\frac{1}{N})}^{1/N}
 \frac{{(1-h)^2}}{(h-(1-h)Nx_1)^2}\dy\dx\\
 &=
 \frac{\theta^2}{N^2} +\frac{9 \theta^2}N \int_{-1/N}^0 \frac{{(1-h)^2}}{h-(1-h)Nx_1}\dx
 {\leq}
 \frac{\theta^2}{N^2} +\frac{9 \theta^2}{N^2} \ln \frac1h\le
  c\frac{\theta^2}{N^2}\ln\left(3+\frac{1}{h}\right).
\end{split}
\end{equation*}
By periodicity the proof is concluded.
}
\end{proof}
We now recall the basic branching construction from \cite{conti-zwicknagl:16}, which refines the one in \cite{kohn-mueller:94} (see Fig.~\ref{fig:branching}).
\begin{figure}
 \begin{center}
  \includegraphics[height=3.2cm]{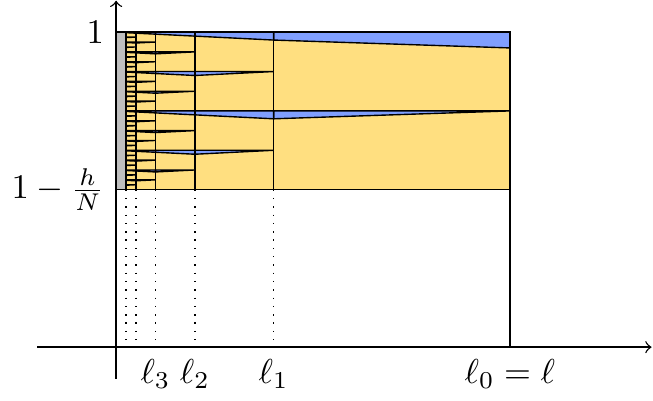}\hskip.1cm
  \includegraphics[height=2.9cm]{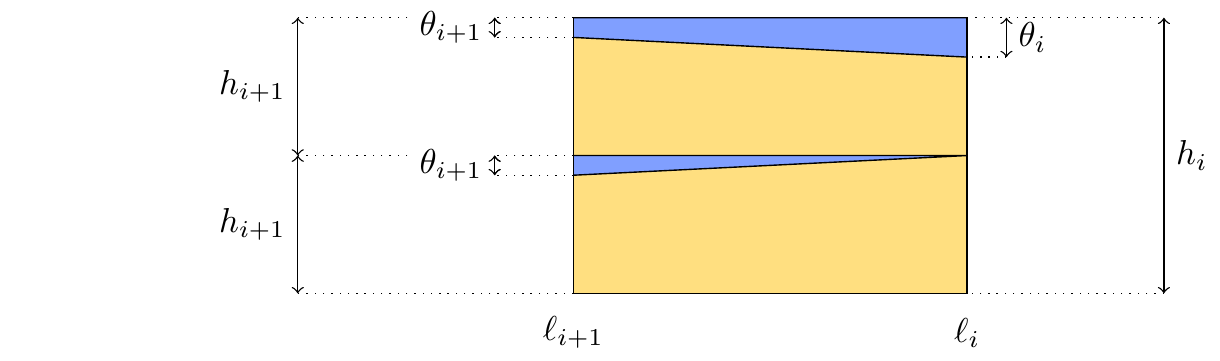}
 \end{center}
\caption{Sketch of the branched construction in Lemma \ref{lem:2scbrobenbranchN}. Left: global construction, with $h=1/2$, $N=1$, and four refinement steps. The gray region $(0,\ell_{k+1})\times(1-\frac hN,1)$ is the one where linear interpolation is used. Right: enlargement of a unit cell $(\ell_{i+1},\ell_i)\times(1-h_i,1)$ in the refinement. The colors mark the regions  $\{\partial_2u=\theta\}$ {(yellow)} and $\{\partial_2u=\theta-1\}$ {(blue)}.}
\label{fig:branching}
\end{figure}

\begin{lmm}
\label{lem:2scbrobenbranchN}
Suppose that $\theta\in(0,1/2]$, $N\in\N$, $N\ge 1$, $h\in[\theta,1]$, $\ell\in[\theta,\infty)$. Then there exists $u:=u_{h,\ell,N}\in W^{1,\infty}\left((0,\ell)\times(1-\frac hN,1)\right)\cap C^0\left([0,\ell]\times[1-\frac hN,1]\right)$ with the following properties:
\begin{itemize}
\item[(i)] $u(0,x_2)=\frac{\theta}{h}(1-h)(1-x_2)$ for all $x_2\in[1-\frac hN,1]$;
\item[(ii)] 
\[u(\ell,x_2)=\begin{cases}
\theta x_2 -\theta(1-\frac1N),&\text{\qquad if \ }x_2\in [1-\frac hN,1-\frac\theta N],\\
(1-\theta)(1-x_2),&\text{\qquad if \ }x_2\in (1-\frac\theta N,1];
\end{cases} \]
\item[(iii)] $u(x_1,1-\frac hN)=\frac1N\theta(1-h)$ for all $x_1\in [0,\ell]$;
\item[(iv)] $u(x_1,1)=0$ for all $x_1\in[0,\ell]$;
\item[(v)] $|D^2u|\left((0,\ell)\times (1-\frac hN,1)\right)\leq c{(\ell+\frac hN)}${and $\|\nabla u\|_{L^\infty}\le c$;}
\item[(vi)]
\[\int_{(0,\ell)\times(1-\frac hN,1)}\left(\partial_1u\right)^2+\min\left\{(\partial_2u-\theta)^2,(\partial_2u+1-\theta)^2\right\}\dcalL^2\leq c\frac{\theta^2 h}{N^3\ell}. \]
\end{itemize}
\end{lmm}

\begin{proof}
One can use the finite branching construction given in \cite{conti-zwicknagl:16} building on Lemma \cite[Lemma 5.2]{conti-zwicknagl:16} and truncation parameter $I\in\N$ such that $(3/2)^I\sim\ell/\theta$. The estimates then follow from the considerations in the proof of \cite[Proposition 6.1]{conti-zwicknagl:16}.

For the convenience of the reader we sketch the main steps of the construction, referring to 
Figure \ref{fig:branching} for an illustration.
 For $i\in\N$ we set $h_i:=2^{-i}h/N$,  $\theta_i:=2^{-i}\theta/N$, $\ell_i:=3^{-i}\ell$, and let $k$ be the largest integer such that $\theta_k\le \ell_k$.
For $0\le i\le k$, the function $x_2\mapsto \partial_2 u(\ell_i,x_2)$ is $h_i$-periodic, with
$\partial_2u(\ell_i,x_2)=\theta$ for $x_2\in (1-h_i,1-\theta_i)$ and 
 $\partial_2u(\ell_i,x_2)=\theta-1$ for $x_2\in(1-\theta_i,1)$.  
 In $(\ell_{i+1}, \ell_{i})\times(1-\frac hN,1)$, $0\le i<k$, the function $x_2\mapsto u(x_1,x_2)-\frac\theta{h}(1-h)(1-x_2)$ is $h_i$-periodic in the $x_2$ direction, {$u$} obeys (iii) and (iv),  and is defined interpolating the boundary values as sketched in Figure \ref{fig:branching}.
By construction $\partial_2 u\in\{\theta,\theta-1\}$ almost everywhere in this set. One checks that
$|\partial_1u|\le c\theta_i/\ell_i$, leading to
$|D^2u|([\ell_{i+1},\ell_i]\times(1-\frac hN,1))\le c 2^i \ell_i+c2^i \frac{h_i\theta_i}{\ell_i}$ 
and
 $\|\partial_1u\|_{L^2([\ell_{i+1},\ell_i]\times(1-\frac hN,1))}^2\le c (\frac{\theta_i}{\ell_i})^2 \ell_i\frac hN$. Summing the two geometric series, this leads to the bounds in (v) and (vi) on $(\ell_k,\ell)\times(1-\frac hN,1)$. In $(0,\ell_k)\times(1-\frac hN,1)$ we use an affine interpolation. The condition $\theta_k\le\ell_k$ gives $|\nabla u|\le c$ in this region, so that the elastic energy is bounded by $c\frac hN\ell_k$. 
 Since $\ell_{k+1}<\theta_{k+1}$ we have $\ell\le \frac\theta N (\frac32)^{k+1}\le 2 \frac\theta N 3^{k/2}$, which implies
 $\ell_k\ell \le 4 \frac{\theta^2}{N^2}$ and hence (v). In turn, $|D^2u|((0, \ell_k]\times(1-\frac hN,1))\le c (2^k\ell_k+\frac hN)\le c(\ell+\frac hN)$.  
 This concludes the proof.
\end{proof}

{Finally, we shall frequently use a function that interpolates between a single laminate and an affine function.
\begin{lmm}\label{lem:tildev}
Let $\theta\in(0,1/2]$ and $\tilde{\beta}>0$. There exists a function $\tilde{v}_{\tilde{\beta}}{=}\tilde{v}{\in C^0([0,\tilde{\beta}]\times[0,\theta])}$ such that
\begin{itemize}
\item[(i)] $\tilde{v}(0,x_2)=(-1+\theta)x_2$ and $\tilde{v}(\tilde{\beta},x_2)=\theta x_2$ for all $x_2\in[0,\theta]$;
\item[(ii)] $\tilde{v}(x_1,0)=0$ for all $x_1\in[0,\tilde{\beta}]$;
\item[(iii)] $\tilde{v}(x_1,\theta)=\theta \iota_{0,\tilde{\beta}}(x_1)+\theta(-1+\theta)$, with $\iota_{0,\tilde{\beta}}$ as in (\ref{eq:iota});
\item[(iv)]{$\|\nabla \tilde v\|_{L^\infty}\le 1+\frac \theta{\tilde\beta}$} and the energy is estimated by
\[\int_{(0,\tilde \beta)\times (0,\theta)}{\min}\left\{|\nabla \tilde{v} -
\theta e_2
|^2, 
|\nabla \tilde{v} +
(1-\theta)e_2
|^2\right\} \dcalL^2  +
\varepsilon |D^2\tilde{v}|((0,\tilde\beta)\times{(0,\theta))}\leq c\left( \frac {\theta^3}{\tilde  \beta} + \varepsilon (\tilde \beta + \frac{\theta^2}{\tilde\beta})\right).\]
\end{itemize}
\end{lmm}
\begin{proof} The standard interpolation 
\begin{eqnarray}
\label{eq:tildev}
 \tilde v(x):= 
 \begin{cases}
  \theta x_2,				&\text{if }x_2 \leq \frac\theta {\tilde \beta} x_1,\\
  -(1-\theta) x_2 +\frac \theta {\tilde \beta} x_1,	&\text{if }x_2 > \frac \theta {\tilde \beta} x_1
 \end{cases}
\end{eqnarray}
satisfies all required properties.
{We simplified the {estimate}  using $\theta\le \tilde\beta+\frac{\theta^2}{\tilde\beta}$.}\end{proof}
}
We shall now provide the proof of the upper bound in Theorem \ref{th:main}. We proceed in two steps: In the first step, we adapt a result from the literature (Proposition \ref{prop:upperbound}), and in the second step, we provide test functions for the remaining regimes (Theorem \ref{th:upperbound}).

\begin{figure}
 \begin{center}
  \includegraphics[width=14cm]{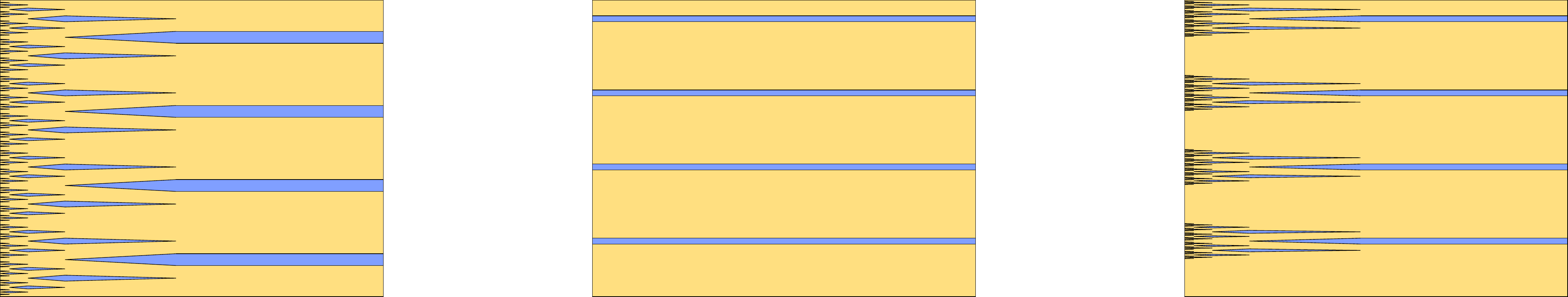}
 \end{center}
\caption{Sketch of the three regimes from Proposition \ref{prop:upperbound}. Only the left half of the martensite, in $\Omega_L$, is plotted, for parameters for which the $\eps L$ term is not relevant.
Left: branching construction from Prop.~\ref{prop:upperbound}(a).
Middle:  laminate from Prop.~\ref{prop:upperbound}(b).
Right: two-scale branching from Prop.~\ref{prop:upperbound}(c).
}
\label{fig-br2-disc}
\end{figure}

\begin{prpstn}[Upper bound: constructions from the literature]\label{prop:upperbound}
There exists $c>0$ such that for  all $\mu>0$, $\varepsilon>0$, $\theta\in(0,1/2]$, and  $L\geq1/2$ there holds
\begin{align*}
 \min_u E(u) 
 \leq c \min \Big{\{} 
 &\varepsilon^{2/3}\theta^{2/3} L^{1/3} +\varepsilon L,\,\, \mu^{1/2}\varepsilon^{1/2}\theta L^{1/2}  (\ln (3+ \frac 1 {\theta^{2}}))^{1/2}+\varepsilon L, \\
 &\mu^{1/2}\varepsilon^{1/2}\theta L^{1/2} (\ln (3+ \frac{\varepsilon}{ \mu^{3}\theta^{2}{L}}))^{1/2}+\varepsilon L
 \Big{\}}.
\end{align*}
\end{prpstn}
\begin{proof}
{The assertion follows from the proof of \cite[Proposition 5.1]{conti-zwicknagl:16}, checking carefully that the differences between the two functionals are not relevant. 
{For clarity} we provide a short self-contained argument, based on 
Lemma \ref{lem:lamoutside} and  Lemma \ref{lem:2scbrobenbranchN} (both taken from \cite{conti-zwicknagl:16}).
The three constructions are illustrated in Figure \ref{fig-br2-disc}.
\begin{enumerate}
\item[(a)] Let $N\ge 1$, $h:=1$, $\ell:=L$, and let $u\in C^0([0,L]\times[1-\frac1N,1])$ be as in Lemma \ref{lem:2scbrobenbranchN}, extended periodically in $x_2$ to
$[0,L]\times [0,1]$, symmetrically for $x_1\in (L,2L]$ and then by zero outside $\overline{\Omega_{2L}}$. We have
\begin{equation*}
 E(u)\le c N \eps L + c\frac{\theta^2}{N^2L} {+\eps}.
\end{equation*}
Choosing $N$ to be the smallest integer above $\theta^{2/3}\eps^{-1/3}L^{-2/3}$, we obtain \\
$E(u)\le c(\varepsilon^{2/3}\theta^{2/3} L^{1/3} +\varepsilon L)$. 

 \item[(b)] 
Let $N\ge 1$, $h:=\theta$, and let $v_{N,\theta}$ be as in Lemma \ref{lem:lamoutside}. 
We extend it setting 
$v_{N,\theta}(x)=v_{N,\theta}(0,x_2)$ for $x_1\in(0,L]$,  symmetrically for $x_1\ge L$, and by zero on $\R\times (\R\setminus[0,1])$. By
Lemma \ref{lem:lamoutside}(ii) we have $\partial_1 v_{N,\theta}=0$ and $\partial_2 v_{N,\theta}\in\{\theta,\theta-1\}$ in $\Omega_{2L}$, so that
$I_{\Omega_{2L}}(v_{N,\theta})\le 4N\eps L$.
By Lemma \ref{lem:lamoutside}(v),
\begin{equation*}
 E(v_{N,\theta})\le c\mu\theta^2 N^{-1} \ln(3+\frac1{\theta^2}) + cN\eps L.
\end{equation*}
Choosing $N$ as the smallest integer above $(\eps^{-1}L^{-1}\mu\theta^2\ln(3+\frac1{\theta^2}))^{1/2}$, we obtain \\
$E(v_{N,\theta})\le c(\mu^{1/2}\varepsilon^{1/2}\theta L^{1/2} (\ln (3+ \frac 1 {\theta^{2}}))^{1/2}+\varepsilon L)$.

\item[(c)]
It remains to show that $I(c)\le c\left(
 \mu^{1/2}\varepsilon^{1/2}\theta L^{1/2}  (\ln (3+ \frac{\varepsilon}{ \mu^{3}\theta^{2}{L}}))^{1/2}+\varepsilon L\right)$. 
If $\mu^{3/2}\eps^{-1/2} \theta L^{1/2}\le \theta$, this follows from (b). Otherwise, we 
fix again $N\ge 1$, $h\in[\theta,1]$, $\ell:=L$ and use
 Lemma \ref{lem:2scbrobenbranchN} and  Lemma \ref{lem:lamoutside} to obtain a function $u$ with 
\begin{equation*}
 E(u)\le c\mu\theta^2 N^{-1} \ln(3+\frac1{h^2}) + cN\eps L +c\frac{\theta^2h}{N^2L}{+\eps}{,}
\end{equation*}
{where we used that $\ln (3+\frac{1}{h})\leq \ln(3+\frac{1}{h^2})$ since $h\leq 1$.}
We choose $h:=\min\{1,\mu L (\mu\theta^2/(\eps L))^{1/2}\}=\min\{1,\mu^{3/2}\eps^{-1/2}\theta L^{1/2}\}{\in[\theta,1]}$  and 
$N$ to be the smallest integer above $(\mu\theta^2\ln(3+\frac{\eps}{\mu^3\theta^2 L}) /(\eps L))^{1/2}$.
Using $h\le \mu^{3/2}\eps^{-1/2}\theta L^{1/2}$ {and $\ln (3+\frac{\eps}{\mu^3\theta^2 L}) \geq 1$},
the proof is concluded.
\end{enumerate}
}
\end{proof}

\begin{figure}
 \begin{center}
  \includegraphics[width=8cm]{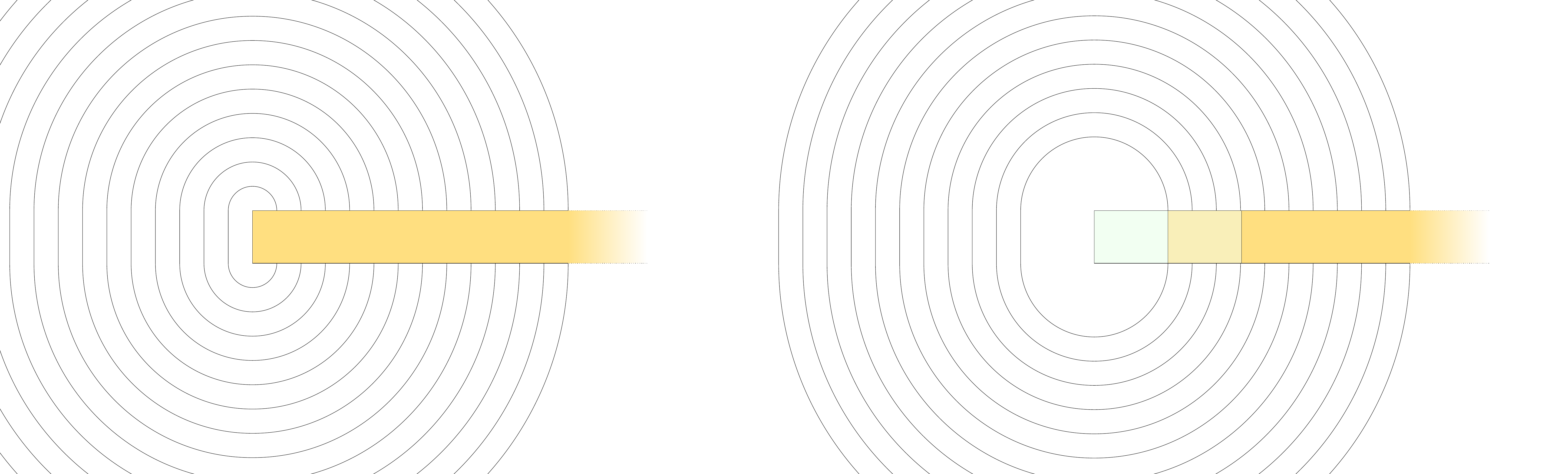}
 \end{center}
\caption{
Left: Sketch of the {affine} regime 
(see proof of Theorem \ref{th:upperbound}(ii)). 
The martensite in $\Omega_{2L}$ has an affine deformation $u=\theta x_2$; in the austenite the field lines of $\nabla u$ are sketched. As usual, we only plot the region with $x_1\le L$. 
Right: Sketch of {{the linear} interpolation} regime 
(see proof of Theorem \ref{th:upperbound}(iii)). 
The martensite has $u=\theta x_2$ only for $x_1\in [2\mu,L]$, it has $u=0$ for $x_1\in[0, \mu]$, and the affine interpolation in between. Correspondingly the field lines start at $x_1=\mu$. This construction is relevant for $\mu\gg 1$, for simplicity $\mu=1.4$ is plotted.}
\label{fig-regimes1}
\end{figure}

\begin{thrm}[Upper bound: conclusion]\label{th:upperbound}
 {There is a constant $c>0$ such that for all $\mu>0$, $\varepsilon>0$, $\theta\in (0,1/2]$, and $L\geq1/2$}
 \begin{align*}
  &\min_u J(u) 
 \leq  c\mathcal{I}({\mu,\eps,\theta,L}),
\end{align*}
where $\mathcal{I}({\mu,\eps,\theta,L})$ is given in Theorem \ref{th:main}.
\end{thrm}

\begin{proof}
{By the estimate $|e(u)|\leq |\nabla u|$, it suffices to show an upper bound for $I$, which in turn follows from an upper bound for $E$.}
We provide test functions for the respective regimes separately. Some constructions are used for several test functions. We will describe them in detail the first time we use them and refer to the arguments in the path of the proof.
\begin{itemize}
\item[(0)] 
{Branching, laminate, two-scale branching:}
By Proposition \ref{prop:upperbound}, we have
\begin{align*} \min_uE(u) \leq c 
\min  \Big{\{}\varepsilon^{2/3}\theta^{2/3} L^{1/3}+\varepsilon L,\, & \mu^{1/2}\varepsilon^{1/2}\theta{L^{1/2} } (\ln (3+ \frac 1{\theta^{2}}))^{1/2} +\varepsilon L,\\
&\mu^{1/2}\varepsilon^{1/2} \theta {L^{1/2}}(\ln (3+\frac {\varepsilon}{ \mu^{3}\theta^{2}{L}}))^{1/2}+\varepsilon L
\Big{\}}.
\end{align*}
\item[(i)]{Constant:} Set ${u^{(1)}}:=0$ in $\R^2$. This shows $\min_{u}  E(u)\leq {E}({u^{(1)}}) \le2 \theta^2L$.
\item[(ii)]{Affine:} We aim to show that $\min_u E(u) \leq c \mu \theta^2 \ln (3+L)$. Choose $\bar{L}:=L+2$, $\beta:=2$ and $\alpha:=1$. We note that the assumptions of Lemma \ref{lem:lemout} are satisfied, and we shall use the function $u_{\alpha,\beta,\bar{L}}$. Precisely, we set
 \[{u^{(2)}}(x):=\begin{cases}
 \theta x_2,&\text{if }x\in\Omega_{L},\\
 u_{{1,2,L+2}}(x_1+2,x_2),&\text{if }x\in ((-\infty,L]\times\R)\setminus ((-2,L]\times(0,1)),\\
 0,&\text{if }x\in (-2,-1]\times (0,1),\\
 (1+x_1)\theta x_2,&\text{if }x\in (-1,0]\times(0,1),\\
 u^{(2)}(2L-x_1,x_2),&{\text{if }x_1>L}
  \end{cases} \]
{see Fig.~\ref{fig-regimes1}, left panel.}
We obtain, by Lemma \ref{lem:lemout} {(using that $\frac{\beta+\alpha}{\beta-\alpha}=3\leq c\ln(3+L)$)} and an explicit computation in $(-1,0)\times(0,1)$,
\[\min_u E(u)\leq E(u^{(2)})={E}^\ext({u^{(2)}}) \leq c \mu \theta^2 \ln (3+L).\]
\item[(iii)]
{{Linear} interpolation: We}
 aim to show that $\min_u E(u) \leq c \left(\mu \theta^2 \ln (3+L/\mu) +\varepsilon \theta\right)$.\\
We distinguish some cases.\\
a)  If $\mu\le1$, this holds by (ii). \\
b) If $\mu\ge L/3$, this holds by (i). \\
c) If $\mu\in(1,L/3)$, we choose $\alpha :=\mu$, $\beta:= 2\mu$ and $\overline{L}:=L$. Note that these choices are admissible for Lemma \ref{lem:lemout} since $\alpha=\mu>1$ and  $L\ge 3\mu\ge \beta$. We set (with $\iota$ as defined in \eqref{eq:iota})
 \[u^{(3)}(x):=\begin{cases}
 0,								& \text{if }x\in [0,\mu]\times[0,1]{,}\\
\iota_{\mu, 2\mu}(x_1)\theta x_2, 				& \text{if }x \in (\mu, 2\mu]\times[0,1]{,}\\
\theta x_2, 							& \text{if }x \in (2\mu, L] \times[0,1]{,}\\
u_{\mu,2\mu,L}(x),&\text{if\ } x\in((-\infty,L]\times\R)\setminus\overline{\Omega_L},\\
u^{(3)}(2L-x_1,x_2),&{\text{if }x_1>L}
 \end{cases}
  \]
  {see Fig.~\ref{fig-regimes1}, right panel.}
 Then ${E}^\ext({u^{(3)}}) \leq c\mu \theta^2 \ln (3+\frac{L}{\mu}) $ by Lemma \ref{lem:lemout}, and hence {by an explicit computation in $\Omega_L$ using that $\mu<1$}
 \[{E}({u^{(3)}}) \leq c\left( \mu \theta^2 \ln \big(3+\frac{L}{\mu}\big) +\varepsilon \theta  +{\frac {\eps\theta} \mu+\mu \theta^2}\right)\leq c\left( \mu \theta^2 \ln \big(3+\frac{L}{\mu}\big) +\varepsilon \theta\right).\]
 \item[(iv)] 
 Next we aim to show an auxiliary result, namely that
 \begin{equation}\label{eqivE}
 \min_u E(u) \leq c\left(\mu \theta^2 \ln (3+ \frac { \varepsilon {L}}{\mu \theta^{2}}) 
 {+\mu\theta^2\ln (3+\frac{1}{\theta^2})}
 + \varepsilon^{1/2}\theta^{3/2}\right).
 \end{equation}
 This bound is not needed for the proof of the theorem, but it introduces a new construction method that will be used for cases (v) and (vi) below. {The idea behind the construction is a single laminate close to the left and right boundaries of the nucleus, interpolated to an affine function deep in the bulk. Related estimates play also a role in the proof of the lower bound, see e.g. the assumptions of Proposition \ref{lem:lbbranching}.}
 \\
 We distinguish three cases:\\
 a) If $\mu\theta^2\eps^{-1}\leq 1$, then (\ref{eqivE}) follows from (ii).\\
  b) If $\mu\theta^2\eps^{-1}\geq  L$, we use the function ${v_{1,\theta}}$ from Lemma \ref{lem:lamoutside} with $N=1$ and $h=\theta$, and set
 \[u^{(4)}(x):=\begin{cases}
 \theta x_2,&\text{if }x\in[0,L]\times[0,1-\theta],\\
 {(1-\theta)(1-x_2)},&\text{if }x\in[0,L]\times(1-\theta,1],\\
 {v_{1,\theta}(x)},&\text{if }x\in(-\infty,0)\times [0,1],\\
 0, &{\text{if } x\in (-\infty,L]\times (\R\setminus[0,1]),}\\
 u^{(4)}(2L-x_1,x_2),&{\text{if }x_1>L}.
 \end{cases} \]
 Then by Lemma \ref{lem:lamoutside}, 
 \[E(u^{(4)})\leq c\left(\mu\theta^2\ln(3+{\frac{1}{\theta}})+\eps L\right)\leq c\mu\theta^2\ln(3+\frac{1}{\theta^2}) \]
 which concludes the proof of (\ref{eqivE}).\\
c) If $\mu\theta^2 \eps^{-1}\in(1, L)$, we use {$v_{1,\theta}$ as above} and Lemma \ref{lem:lemout} with ${\alpha:=\mu\theta^2\eps^{-1}}$, $\beta:=\alpha+\tilde{\beta}$ with $\tilde{\beta}\geq \alpha$ chosen below, 
    and {$\bar L:=\max\{L+1, \beta+1\}$.}  Precisely, we set with {$\tilde{v}_{\tilde{\beta}}$ from Lemma \ref{lem:tildev} and $v_{1,\theta}$ from Lemma \ref{lem:lamoutside}}
    \[U^{(4)}(x):=\begin{cases}
    \theta x_2, 									& \text{if } x\in [0,L]\times [0, 1-\theta]{,}\\
(1-\theta)(1-x_2),								& \text{if } x\in[0,\alpha]\times (1-\theta,1],\\
\tilde v_{{\tilde{\beta}}}(x_1-{\alpha},x_2-(1-\theta)) + \theta(1-\theta),		& \text{if } x\in(\alpha,{\alpha}+ \tilde \beta]\times
(1-\theta,1],\\
\theta x_2,									& \text{if }  x\in ( {\alpha}+ \tilde \beta,{L}]\times(1-\theta,1],\\
{ v_{1,\theta}(x),} &\text{if } x\in (-1,0)\times[0,1],\\
{u_{\alpha+1,\beta+1 ,\bar{L}}(x_1+1,x_2),}&\text{if }x\in((-\infty,L]\times\R)\setminus ((-1,L]\times[0,1]),\\
U^{(4)}(2L-x_1,x_2),&{\text{if }x_1>L}.
    \end{cases} \]
 One readily checks that $U^{(4)}$ is continuous. By Lemma \ref{lem:lemout} {and  Lemma \ref{lem:lamoutside}}, we have
\begin{equation*}
E^\ext({U^{(4)}}) \leq c\mu \theta^2\left(\ln \Big(3+\frac{L+\alpha+\tilde\beta{+1}}{\alpha}\Big) + \frac{\beta+\alpha}{\beta-\alpha} 
+\ln(3+\frac1{\theta^2})
\right){.}
\end{equation*}
Altogether, we obtain {using Lemma \ref{lem:tildev}} {and $\beta\ge 2\alpha$}
\begin{align*}
 {E}({U^{(4)}})&\leq c\left(\frac{\theta^3}{ {\tilde{\beta}}} + \varepsilon {\tilde{\beta}}
 +\frac{\eps\theta^2}{{\tilde{\beta}}} 
 \right)
+ c\mu \theta^2
{\ln\Big(3+ \frac{L}{\alpha}+\frac{\tilde{\beta}}{\alpha}\Big)}
+c\mu\theta^2\ln (3+\frac{1}{\theta^2}){+\eps}
\\
&\le c \left(\frac{\theta^3}{ {\tilde{\beta}}} + \varepsilon {\tilde{\beta}}
 \right)
+ c\mu \theta^2\ln\Big(3+ \frac{\eps L }{\mu\theta^2}+\frac{{\tilde{\beta}}}{\alpha}\Big) {+c\mu\theta^2\ln (3+\frac{1}{\theta^2})}
\end{align*}
where in the second step we used 
that $\tilde\beta\ge1$ implies $\eps\theta^2/\tilde{\beta}\leq\eps\tilde\beta${, and the assumption  $\eps\leq\mu\theta^2$.}
At this point we distinguish two further subcases. If $\mu\theta^2\le \eps^{1/2}\theta^{3/2}$ then we set $\tilde{\beta}:= \varepsilon^{-1/2} \theta^{3/2}{\geq \mu\theta^2\eps^{-1}}=\alpha$ and obtain
 \begin{equation*}
 {E}({U^{(4)}})\le  c\left(\varepsilon^{1/2}\theta^{3/2}
+ \mu \theta^2\ln\Big(3+ \frac{\eps L }{\mu\theta^2}+\frac{\eps^{1/2}\theta^{3/2}}{\mu\theta^2}\Big) \right)
{+c\mu\theta^2\ln (3+\frac{1}{\theta^2})}. 
 \end{equation*}
We treat the first logarithm using $\ln (3+x+y)\le \ln (3+x)+\ln (1+y)\le \ln (3+x)+y$ {for $x,y\geq 0$}, leading to
 \begin{equation*}
 {E}({U^{(4)}})\le  c\left(\varepsilon^{1/2}\theta^{3/2}
+ \mu \theta^2\ln\Big(3+ \frac{\eps L }{\mu\theta^2}\Big) \right){+c\mu\theta^2\ln (3+\frac{1}{\theta^2})},
 \end{equation*}
 which concludes the proof of (\ref{eqivE}).
If instead $\mu\theta^2> \eps^{1/2}\theta^{3/2}$ then we set {$\tilde{\beta}:= \alpha\ge \eps^{-1/2}\theta^{3/2}$} and obtain
 \begin{equation*}
 {E}({U^{(4)}})\le  c\left(\eps^{1/2}\theta^{3/2} 
+ \mu \theta^2\ln\Big(3+ \frac{\eps L }{\mu\theta^2}\Big) \right) {+c\mu\theta^2\ln (3+\frac{1}{\theta^2})}.
 \end{equation*}

\begin{figure}
 \begin{center}
  \includegraphics[width=8cm]{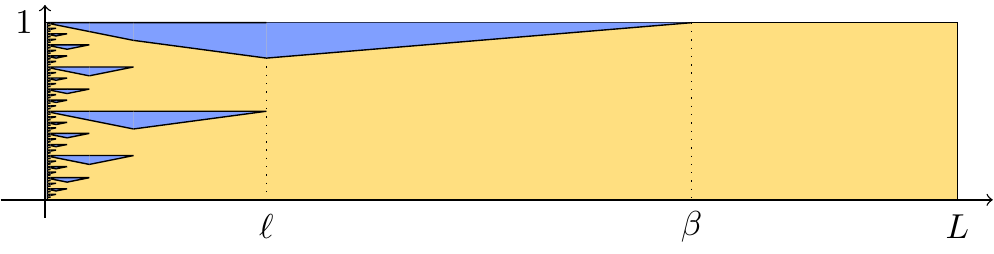}
 \end{center}
\caption{{Sketch of the single truncated branching construction. 
We refer to Fig.~\ref{fig:branching} for details of the branching construction on the left, to Fig.~\ref{fig-regimes1} for the field lines in the austenite.}}
\label{figsingletruncbra}
\end{figure}

\item[(v)]{Single truncated {branching}: We} aim to show that
$\min_u E(u)\leq c\left(\mu \theta^2 \ln (3+ \frac { \varepsilon {L}}{\mu \theta^{2}}) 
 +\mu\theta^2\ln (3+\frac{\eps}{\mu^2\theta^2})
 + \varepsilon^{1/2}\theta^{3/2}\right)$. Again, we distinguish several cases.\\
 a)  If  $\mu\theta^2\eps^{-1}\leq 1$, this follows from (ii).\\
 b) If  $\mu^2\theta^2\eps^{-1}{<} \theta$, then this follows from (iv).\\
c) If $\mu^2\theta^2\eps^{-1}{\geq\theta} $ {and $\mu\theta^2\eps^{-1}> 1$}, we 
use the truncated branching construction 
  from \cite[Proofs of Theorems 3.1, 3.2]{zwicknagl:14}, in the version of Lemma \ref{lem:2scbrobenbranchN} {and $v_{1,h}$ from Lemma \ref{lem:lamoutside}}. Precisely, we {choose $h:=\min\{1,\mu^2\theta^2\eps^{-1}\}\in[\theta,1]$}, $N:=1$, and 
  $\ell:=\mu\theta^2\eps^{-1}{\geq 1{\geq}h}$ {and set
    \begin{equation}\label{eqdefU5}
{u}^{(5)}(x):=\begin{cases}
\theta x_2,&\text{\ if\ } x\in[0,\min\{\ell,L\}]\times[0,1-h],\\
u_{h,\ell,1}(x),&\text{\ if\ }x\in[0,\min\{\ell,L\}]\times(1-h,1],\\
v_{1,h}(x),&\text{\ if\ }x\in[-1,0)\times[0,1],
\end{cases}
\end{equation}
which satisfies} 
  \begin{equation}\label{eqestw}
E_{{[-1,\ell)}\times(0,1)}(w)\leq {c\frac{\theta^2{h}}{\ell}+c\eps{(\ell+{h})}}{+c\mu\theta^2\ln(3+\frac{1}{h}) \le c\mu\theta^2\ln(3+\frac{\eps}{\mu^2\theta^2})}.  
  \end{equation}
  {In the second inequality we used that $\frac{\theta^2 h}{\ell}=\frac{\eps h}{\mu}\leq \mu\theta^2$ and $\ell\geq h$.
  
  If $\ell\geq L/2$, we extend $u^{(5)}$ inside $\Omega_L$ by a simple laminate, i.e., 
\[u^{(5)}(x):=\begin{cases}
\theta x_2,&\text{\ if\ }x\in(\ell,L]\times[0,1-\theta],\\
(1-\theta)(1-x_2),&\text{\ if\ }x\in(\ell, L]\times(1-\theta,1],\\
0,&\text{\ if\ }x\in((-\infty,L]\times\R)\setminus([-1,L]\times[0,1]),\\
u^{(5)}(2L-x_1,x_2),&{\text{\ if }x_1>L}.
\end{cases} \]
Note that $(\ell,L]=\emptyset$ if $L<\ell$. We have $E(u^{(5)})\leq c\left(E_{{[-1,\ell)}\times(0,1)}(u^{(5)})+\eps \ell\right)$, and the assertion follows.\\
Otherwise, if $\ell< L/2$, we proceed as in (iv)c) using Lemma \ref{lem:tildev} and Lemma \ref{lem:lemout}
with $\tilde{\beta}:=\alpha:=\ell$, $\beta:=\ell+\tilde\beta$, and $\bar{L}:=\max\{\beta+1,L+1\}$ and set 
\[u^{(5)}(x):=\begin{cases}
\theta x_2,									& \text{if }  x\in ( {\ell},{L}]\times[0,1-\theta],\\
\tilde v_{\tilde{\beta}}(x_1-{\ell},x_2-(1-\theta)) + \theta(1-\theta),		& \text{if } x\in(\ell,{\ell}+ \tilde \beta]\times
(1-\theta,1],\\
\theta x_2,									& \text{if }  x\in ( {\ell}+ \tilde \beta,{L}]\times(1-\theta,1],\\
u_{\ell+1,\ell+\tilde\beta+1 ,\bar{L}}(x_1+1,x_2),&\text{if }x\in((-\infty,L]\times\R)\setminus ((-1,L]\times[0,1]),\\
u^{(5)}(2L-x_1,x_2),&{\text{if }x_1>L}.
\end{cases} \]
This leads to
\[
 E(u^{(5)})\le c \left(\mu\theta^2\ln\left(3+\frac{\eps}{\mu^2\theta^2}\right)+\eps \ell +\frac{\theta^3}{\tilde\beta} + \eps\tilde\beta+\frac{\eps\theta^2}{\tilde\beta} + \mu\theta^2\ln
 \left(3+\frac{L+\ell+1}{\ell}\right){+\eps}\right),
\]
and the assertion follows as in (iv)c).}\\
\begin{figure}
 \begin{center}
  \includegraphics[width=10cm]{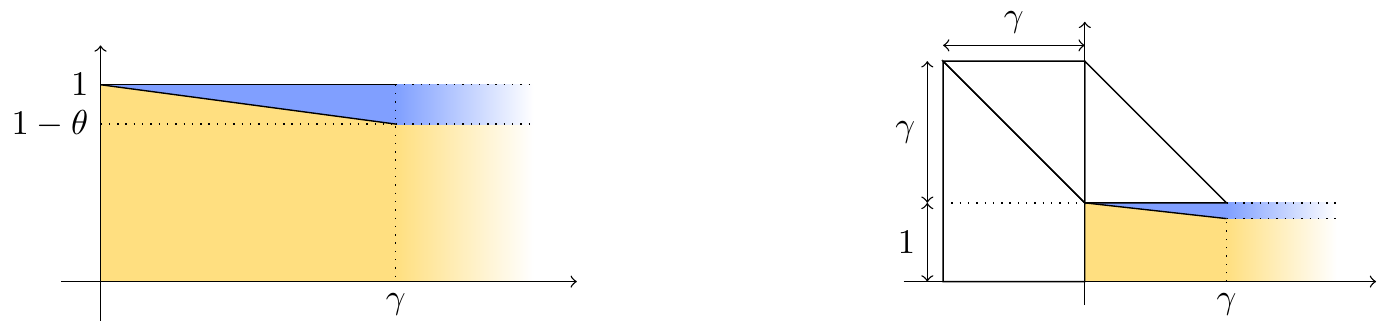}
 \end{center}
\caption{Construction for $\tilde u^{(6)}$. The left panel shows the construction {in the martensite,} the right panel the subdivision of the domain in the austenite phase.}
\label{figlnthetamu}
\end{figure}
\begin{figure}
 \begin{center}
  \includegraphics[height=3cm]{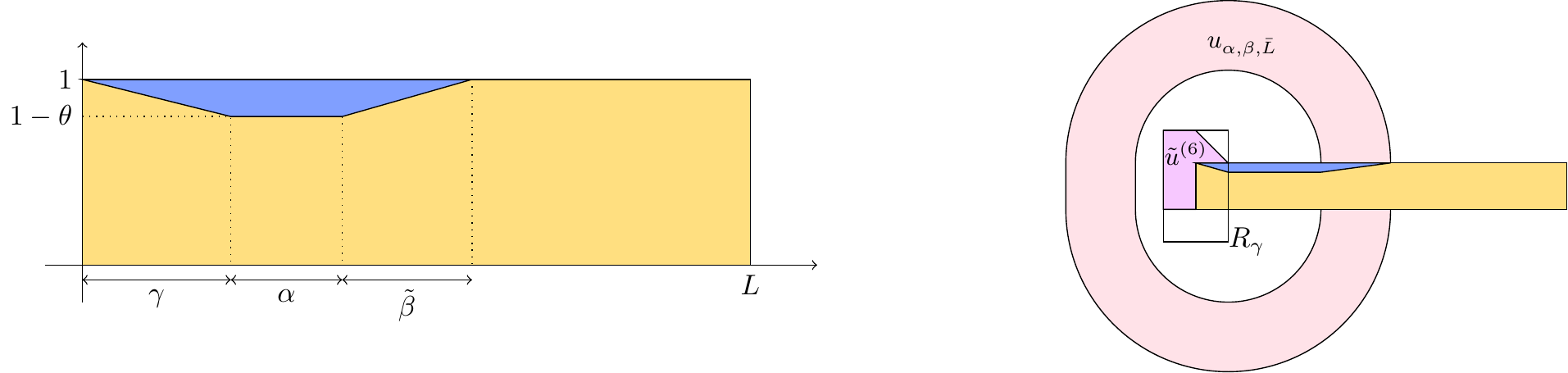}
 \end{center}
\caption{{Corner-laminate} construction for $u^{(6)}$. The left panel shows the construction in the martensitic region $\Omega_L$, the right panel (on a different scale, and with different parameters) the construction in the austenite. The shaded regions are those where $\tilde u^{(6)}\ne 0$ and $u_{\alpha,\beta,\bar{L}}(x_1-\gamma,x_2)\ne0$, respectively (see also Fig.~\ref{figlnthetamu}).}
\label{figlnthetamu2}
\end{figure}
\item[(vi)]{Corner laminate:} We show that
{\[\min_u E(u)\leq c\left({\mu \theta^2 \ln (3+ \frac { \varepsilon {L}}{\mu \theta^{2}}) 
 +\mu\theta^2\ln (3+\frac{\theta}{\mu})}
\right).\] }
Again, we distinguish several cases.\\
a) If $\eps\geq\mu\theta^2$, then the assertion follows from (ii).\\
b) If $\theta\le\eps<\mu\theta^2$, then {$\frac{\eps L}{\mu\theta^2}\geq \frac{L}{\mu\theta}>\frac{L}{\mu}$, and $\mu\theta^2> \eps> \eps\theta$, and the assertion} follows from (iii).\\
c) If $\theta/\mu\geq L$, this follows from (ii).\\
d) If $\mu^2\theta\le \eps$, then
$\frac{\eps L}{\mu\theta^2}\frac\theta\mu\ge L$. The assertion follows from (ii) using that $\ln(3+a)+\ln(3+b)\ge\ln(3+ab)$ for any $a,b\ge0$.\\
e) It remains to consider the case that $\eps<\min\{\mu\theta^2,\,\theta, {\mu^2\theta}\}$ and $\theta/\mu< L$. 
We choose {$\gamma:=\max\{\frac{1}{4},\frac{\theta}{\mu\ln(3+\frac\theta\mu)}\}$} and note that $\gamma< L$ since $\theta/ \mu< L$  and $L\geq\frac{1}{2}$.
We first construct a function $\tilde u^{(6)}$ in $R_\gamma:=[-\gamma,\gamma]\times[-\gamma,\gamma+1]$, 
\begin{subnumcases}{\tilde 
u^{(6)}(x):=}
\theta x_2, & $\text{if }x_1\in[0, \gamma],\  0\le x_2\le 1-{\theta\frac{x_1}\gamma}$,\nonumber\\
(1-\theta)(1-x_2) + \theta - \theta\frac{x_1}\gamma, &
${\text{if\ }}  x_1\in[0, \gamma],\  1- \theta\frac{x_1}\gamma< x_2\le1$,\nonumber\\
 \theta(1-\frac{x_1+x_2-1}\gamma), & $\text{if } x_1{\in[0,\gamma], \ 1<x_2\leq 1+\gamma-x_1}$,\label{eq:6B}\\
  \theta(1-\frac{x_2-1}\gamma), & ${\text{if }} x_1{\in[-\gamma,0], \ 1-x_1\leq x_2\le \gamma+1}$,\label{eq:6C}\\
  \theta(1+\frac{x_1}\gamma)\frac{x_2}{1-x_1}, &
  ${\text{if }} x_1{\in[-\gamma,0]}$ \text{ and } $0\le x_2< 1-x_1$,\label{eq:6A}\\
  0, & $\text{elsewhere in }R_\gamma$,\nonumber
\end{subnumcases}
see Fig.~\ref{figlnthetamu}. One easily checks that $\tilde u^{(6)}$ is continuous and that 
\[E_{(0,\gamma)\times(0,1)}(\tilde u^{(6)}){\leq c\left(\frac{\theta^3}{\gamma}+\eps(\gamma+\theta) + \frac{\eps\theta^2}{\gamma}\right)}
\leq c\left(\frac{\theta^3}{\gamma}+\eps \gamma+\eps\theta\right), \]
{where we used in the last estimate that $\eps\leq\theta$.}
To estimate the energy outside $\Omega_{2L}$, we observe that $|\nabla \tilde u^{(6)}|\leq c\theta/\gamma$ in the parts given in \eqref{eq:6B} and \eqref{eq:6C}, and {$|\nabla \tilde u^{(6)}(x)|\leq c\frac{\theta}{|1-x_1|}$} in the part given in \eqref{eq:6A} {(recall that $\gamma\geq 1/4$)}, which yields
\[E_{R_\gamma\setminus(0,\gamma)\times(0,1)}(\tilde u^{(6)})\leq c\mu\theta^2\ln\left(3+\gamma\right).\]
We then proceed as in (iv)c), using $\tilde{v}_{\tilde{\beta}}$ from Lemma \ref{lem:tildev} with $\tilde\beta:=3\mu\theta^2\eps^{-1}$.
Since $\eps<\mu\theta^2$ we have $1\le\tilde\beta$. We use Lemma \ref{lem:lemout} with $\alpha:=\tilde{\beta}$, 
 $\beta:=\alpha+\tilde{\beta}{=2\alpha}$ and $\bar{L}:=\max\{\beta,L\}$ and 
 define $u^{(6)}$ by
\[{u}^{(6)}(x):=\begin{cases}
{\tilde u^{(6)}(x)},&{\text{\ if\ }x\in R_\gamma,}\\
\theta x_2,&\text{\ if\ }x\in (\gamma, \gamma+\alpha+\tilde\beta]\times[0,1-\theta],\\
(1-\theta)(1-x_2),&\text{\ if\ }x\in (\gamma, \gamma+\alpha]\times(1-\theta,1],\\
\tilde{v}_{\tilde{\beta}}(x_1-(\gamma+\alpha),x_2-(1-\theta))+\theta(1-\theta),& \text{\ if\ }x\in(\gamma+\alpha,\gamma+\alpha+\tilde{\beta}]\times(1-\theta,1],\\
\theta x_2,&\text{\ if\ }x\in(\gamma+\alpha+\tilde{\beta},L]\times[0,1],\\
u_{\alpha,\beta,\bar{L}}(x_1-\gamma,x_2),&\text{\ otherwise in } (-\infty,L]\times\R,\\
 u^{(6)}(2L-x_1,x_2),&\text{\ if\  }x_1> L,
\end{cases} \]
see Fig.~\ref{figlnthetamu2}. The condition $\eps\le {\min\{\mu^2\theta,\mu\theta^2\}}$ implies $3\gamma\le \alpha$, so that the construction for $\tilde u^{(6)}$ and the one for $u_{\alpha,\beta,\bar{L}}(x_1-\gamma,x_2)$ match continuously.
The function $u^{(6)}$ is continuous and 
\begin{eqnarray*}
E(u^{(6)})&\leq& c\left(\frac{\theta^3}{\gamma}+\eps\gamma+{\eps\theta+}\mu\theta^2\ln(3+\gamma)+\eps\alpha+\frac{\theta^3}{\tilde{\beta}}+\eps(\tilde{\beta}+\frac{\theta^2}{\tilde{\beta}})+\mu\theta^2\ln(3+\frac{\bar{L}}{\alpha})+\mu\theta^2\frac{\beta+\alpha}{\beta-\alpha}{+\eps}\right)\\
&\leq&\mu\theta^2\ln(3+\frac{\theta}{\mu})+\mu\theta^2\ln\left(3+\frac{\eps L}{\mu\theta^2}\right).
\end{eqnarray*}
We used here {$\eps\gamma\leq\max\{\eps,\frac{\eps \theta}{\mu}\}\leq\mu\theta^2$ since $\eps\leq\min\{\mu\theta^2,\mu^2\theta\}$; and similarly  $\eps({1}+\theta+\alpha+\tilde \beta)\le c\mu\theta^2$ and $\frac{\eps\theta^2}{\tilde\beta}\le \frac{\theta^3}{\tilde \beta}=\frac{\eps\theta^3}{3\mu\theta^2}\le \mu\theta^2$.}
\end{itemize}
This concludes the proof of the upper bound.
\end{proof}

 \subsection{Comments on the scaling law}\label{sec:tabulars}
The purpose of this subsection is two-fold: On the one hand, in Subsection \ref{sec:regnec} we shall prove that all terms in the definition of $\mathcal{I}(\mu,\eps,\theta,L)$ are relevant in the sense that the statement is false if we remove one of them. Furthermore, we give some intuition on the constructions used in the proof of the upper bound. On the other hand, in Subsection \ref{sec:regimesrough}, we shall explain the different parameter regimes and motivate why they are  treated separately in the proof of the lower bound.
\subsubsection{Do all regimes really exist?}
\label{sec:regnec} 
{We} will use the following abbreviatory notation:
We denote as
{\em scaling} an expression like $\eps^{1/2}\theta^{3/2}$, and 
as {\em regime} something like $\mu\theta^2\ln(3+\frac L\mu)+\eps\theta$ (which is the sum of a few scalings). {In particular, $\mathcal I$ is defined in Theorem \ref{th:main} as the minimum of eight regimes.}

{We show below that no regime  $R=R(\mu,\eps,\theta, L)$ can be eliminated from the definition of $\mathcal I$ in Theorem \ref{th:main}. To do this,}
 we shall exhibit a sequence of parameters $\mu_j,\eps_j,\theta_j,L_j$ such that
$\lim_{j\to\infty} \frac{\mathcal I_R(\mu_j,\eps_j,\theta_j,L_j)}{R(\mu_j,\eps_j,\theta_j, L_j)}\to\infty$, where $\mathcal I_R$ is the minimum in Theorem \ref{th:main} without regime $R$.

Additionally, we show that no scaling can be eliminated {in the regimes that consist of more than one scaling}. Consider a regime $R$ which consists of the scalings $S^{(k)}$, in the sense that  $R=S^{(1)}+\dots +S^{(K)}$ {for $K\geq 2$}.  For any $i\in\{1, \dots, K\}$ let 
$R_i:=\sum_{k\ne i} S^{(k)}$ be the regime  $R$ without $S^{(i)}$. We shall provide a sequence  $(\mu_j,\eps_j,\theta_j,L_j)$ such that $\frac{R_i{(\mu_j,\eps_j,\theta_j,L_j)}}{\mathcal I{(\mu_j,\eps_j,\theta_j,L_j)}}\to0$, proving that 
$R$ cannot be replaced by $R_i$. In most cases, this will be done constructing a sequence with 
$\frac{\mathcal I_R}{R}\to\infty$,
$\frac{S^{(i)}}{R}\to1$ and $\frac{S^{(i)}}{S^{(k)}}\to0$ if $k\ne i$ along that sequence, which additionally shows that {the scaling} $S^{(i)}$ dominates the {regime} $R$.


To briefly sketch the ideas behind the constructions in the proof of the upper bound, we describe them only {inside} the martensitic nucleus. They should be considered to be extended optimally (in the sense of trace) to the austenite part. The precise constructions and references {to the literature} are given in Section \ref{sec:upperbound}. 
{We recall that we} write $\ln^\alpha x$ for $(\ln x)^\alpha$, and the same for $\ln^\alpha\ln x$. We write $a_j\sim b_j$ if there is a constant $c>0$ such that $\frac{1}{c} a_j\leq b_j\leq c a_j$ for all $j\in\N$. 

\begin{enumerate}
 \item $R=\theta^2L$ {(constant)}: This regime is attained by a constant test function, corresponding to austenite. This regime is the only one that does not depend on $\eps$ nor $\mu$. We take $\theta_j=L_j=\frac12$, $\eps_j=\mu_j\to\infty$. Then all other regimes have diverging energy. 
 \item $R=\mu\theta^2\ln(3+L)$ {(affine)}: This regime is attained by using an affine function inside the nucleus corresponding to the majority variant of martensite {(see Figure \ref{fig-regimes1} (left)).} We take $\theta_{{j}}=\mu_{{j}}=\frac12$, {$L_j\to\infty$, $\eps_j=e^{L_j}$}. All regimes which contain one of the scalings $\eps L$, $\eps\theta$, $\eps^{1/2}\theta^{3/2}$, $\theta^2 L$ have energ{ies} which diverge at least as a power of {$L_j$}, 
 and also $\mu_j\theta_j^2\ln(3+\frac{\eps_j L_j}{\mu_j\theta_j^2})\gg L_j$, and only $\mu\theta^2\ln(3+L)$ is logarithmic.
 \item $R=\mu\theta^2\ln(3+\frac L\mu)+\eps\theta$ {(linear interpolation)}. This regime is (when relevant) attained by a test function that is constant near the left and the right boundaries of $\Omega_{2L}$ (corresponding to austenite), and affine near the middle $\{x_1=L\}$ of the nucleus (corresponding to the majority variant of martensite). There is a competition of the energy inside the nucleus, which favours the test function to be in the martensitic variant on a large part, and the energy contribution from the austenite part, which favours the function to be constant in a large neighbourhood of the left and right boundaries {(see Figure \ref{fig-regimes1} (right)).}
 \begin{itemize}
\item[(a)]  $S=\mu\theta^2\ln(3+\frac L\mu)$: 
 We take {$L_j\to\infty$,} $\theta_j=\frac12$, $\mu_j=\frac{L_j}{\ln L_j}$, $\eps_j=\mu_j\theta_j$. Then $\eps_j\theta_j=\mu_j\theta_j^2$, $\frac {L_j}{\mu_j}=\ln L_j$, so that $R(\mu_j,\eps_j,\theta_j,L_j)\sim S(\mu_j,\eps_j,\theta_j,L_j)\sim L_j \frac{\ln\ln L_j}{\ln L_j}$, whereas $\theta_j^2 L_j \sim L_j$, $ \mu_j\theta_j^2\ln(3+L_j)\sim L_j$,
 $\eps_j L_j \sim L_j^2/\ln L_j$, 
 and $\frac{\eps_j L_j}{\mu_j\theta_j^2}=\frac{L_j}{\theta_j}$ implies
 $\mu_j\theta_j^2\ln(3+\frac{\eps_j L_j}{\mu_j\theta_j^2})\sim \mu_j\ln L_j\sim L_j$.
 \item[(b)]  $S=\eps\theta$: As above, 
 we take {$L_j\to\infty$,} $\theta_j=\frac12$, $\mu_j=\frac{L_j}{\ln L_j}$, but this time $\eps_j=\mu_j\theta_j (\ln \ln L_j)^2$. Then $\frac{L_j}{\mu_j}=\ln L_j$, so that $R(\mu_j,\eps_j,\theta_j,L_j)\sim S(\mu_j,\eps_j,\theta_j,L_j)\sim \theta_j^2L_j \frac{(\ln\ln L_j)^2}{\ln L_j}$, in the other terms the $\ln\ln L_j$ correction does not change the argument.
 \end{itemize}
\item $R=\mu\theta^2\ln(3+\frac{\eps L}{\mu\theta^2})+\mu\theta^2(3+\frac{\eps}{{\mu^2\theta^2}})+{\eps^{1/2}\theta^{3/2}}$
(single truncated branching). This regime is (when relevant) attained by a test function that consists of roughly three parts: Close to the left and right boundaries of $\Omega_{2L}$, a branching construction is used, which goes over to a single laminate, and then interpolates to an affine function (which corresponds to the majority variant of martensite) near the vertical middle $\{x_1=L\}$ of the nucleus, {see Fig.~\ref{figsingletruncbra}}. 
\begin{itemize}
\item[(a)] $S=\eps^{1/2}\theta^{3/2}$: We take $\theta_j=\frac12$, $L_j\to\infty$, $\mu_j=\frac{1}{L_j}$, $\eps_j=\frac{\mu_j\theta_j^2}{L_j}\ln L_j=\frac{\ln L_j}{4L_j^2}$.
 Then $\mu_j\theta_j^2=\frac{1}{ 4L_j}$, $\eps_j L_j=\frac{\ln L_j}{ 4 L_j}$. 
 In particular, $\theta_j^2L_j\sim L_j$, $\mu_j\theta_j^2 \ln L_j\sim\mu_j\theta_j^2\ln \frac{L_j}{\mu_j} \sim \frac{\ln L_j}{L_j}$.
 We have $\frac{\eps_j L_j}{\mu_j\theta_j^2}=\ln L_j$, $\frac{\eps_j}{\mu_j^2\theta_j^2}
 =\ln L_j$, 
 and $S(\mu_j,\eps_j,\theta_j,L_j)=\eps_j^{1/2}\theta_j^{3/2}\sim \frac{1}{L_j}\ln^{1/2} L_j$, so that $R(\mu_j,\eps_j,\theta_j,L_j)=O(\frac1L_j\ln\ln L_j)+S(\mu_j,\eps_j,\theta_j,L_j)
 =O(\frac{1}{L_j}\ln\ln L_j)+\frac{1}{L_j} \ln^{1/2} L_j$, and $S(\mu_j,\eps_j,\theta_j,L_j)/R(\mu_j,\eps_j,\theta_j,L_j)\to1$.
 All regimes which contain the scaling $\eps_j L_j\sim\frac{\ln L_j}{L_j}$ can be ignored.
 Finally, $\theta_j/\mu_j=\theta_j L$, hence 
 $\mu_j\theta_j^2 \ln\frac{\theta_j}{\mu_j}\sim \frac{1}{L_j}\ln L_j\gg S(\mu_j,\eps_j,\theta_j,L_j)\sim R(\mu_j,\eps_j,\theta_j,L_j)$. This concludes the proof.
 
 \item[(b)] $S=\mu\theta^2\ln(3+\frac{\eps L}{\mu\theta^2})$:
 We take $\theta_j=\frac12$, $L_j\to\infty$, $\mu_j=\frac1{L_j^{1/2}}$, $\eps_j = \frac{\mu_j\theta_j^2}{L_j}\ln L_j=\frac1{ 4L_j^{3/2}}\ln L_j$. Then $\mu_j\theta_j^2=\frac1{ 4L_j^{1/2}}$, $\eps_j L_j=\frac1{ 4L_j^{1/2}}\ln L_j$, 
 $\frac{\eps_j L_j}{\mu_j\theta_j^2}=\ln L_j$, 
 $\frac{\eps_j}{\mu_j^2\theta_j^2}=\frac{\ln L_j}{L_j^{1/2}}\to0$,
 and $\eps_j^{1/2}\sim \frac{1}{L_j^{3/4}}\ln^{1/2}L_j$. Therefore $S(\mu_j,\eps_j,\theta_j,L_j)$ dominates $R(\mu_j,\eps_j,\theta_j,L_j)$, 
 and $S(\mu_j,\eps_j,\theta_j,L_j)\sim  R(\mu_j,\eps_j,\theta_j,L_j)\sim \frac{1}{L_j^{1/2}}\ln\ln L_j$. 
 All regimes with $\eps_j L_j$ are higher, as is obviously $\theta_j^2L_j$. 
 {Further, $\frac{\theta_j}{\mu_j}\sim L_j^{1/2}$ implies $\mu_j\theta_j^2\ln(3+\frac{\theta_j}{\mu_j})\sim \frac{1}{L_j^{1/2}}\ln L_j$}, and finally, 
 $\mu_j\theta_j^2\ln L_j\sim \mu_j\theta_j^2\ln \frac{L_j}{\mu_j}\sim \frac{1}{L_j^{1/2}}\ln L_j\gg R{(\mu_j,\eps_j,\theta_j,L_j))}$.
 
 \item[(c)] $S=\mu\theta^2\ln (3+\frac{\eps}{\mu^2\theta^2})$:
 We take $L_j\to\infty$, $\theta_j=\frac1{\ln^{2/5}L_j}$, $\mu_j=\frac{\ln\ln L_j}{L_j\ln^{1/5} L_j}$, $\eps_j=\frac{\ln^5\ln L_j}{L_j^2\ln L_j}$. Then $\eps_j L_j=\frac{\ln^5\ln L_j}{L_j\ln L_j}$, $\mu_j\theta_j^2=\frac{\ln\ln L_j}{L_j\ln L_j}$, $(\eps_j L_j\mu_j\theta_j^2)^{1/2}=\mu_j\theta_j^2\ln^2\ln L_j$. 
 Further, $\frac{\eps_j L_j}{\mu_j\theta_j^2}=\ln^4\ln L_j$, $\frac{\eps_j}{\mu_j^2\theta_j^2}=\ln^{1/5} L_j\ln^3\ln L_j$, 
 and $S(\mu_j,\eps_j,\theta_j,L_j)
 \sim {\mu_j\theta_j^2 \ln\ln  L_j}\gg \mu_j\theta_j^2\ln\frac{\eps_j L_j}{\mu_j\theta_j^2}
 \sim {\mu_j\theta_j^2 \ln\ln\ln  L_j}$.
For the third scaling in this regime, 
 $\eps_j^{1/2}\theta_j^{3/2}=\frac{\ln^{5/2}\ln L_j}{L_j\ln^{1/2}L_j \ln^{3/5}L_j}=
{\mu_j\theta_j^2 \frac{\ln^{3/2}\ln L_j}{\ln^{1/10}L_j}}\ll S(\mu_j,\eps_j,\theta_j,L_j)$. For the {corner laminate} regime, we estimate $\frac{\theta_j}{\mu_j}=\frac{L_j}{\ln^{1/5}L_j\ln\ln L_j}$ which gives
 $\mu_j\theta_j^2\ln(3+\frac{\theta_j}{\mu_j})\sim {\mu_j\theta_j^2}\ln L_j\gg S(\mu_j,\eps_j,\theta_j,L_j)$. The other regimes are simpler. $\theta_j^2L_j$ is linear in $L_j$, $\mu_j\theta_j^2\ln(3+L_j)$ and {$\mu\theta^2\ln(3+\frac{L}{\mu})$} behave as $\mu_j\theta_j^2\ln L_j$, which is much larger than $S(\mu_j,\eps_j,\theta_j,L_j)\sim\mu_j\theta_j^2\ln\ln L_j$, and the last ones {(branching, laminate and two-scale-branching)} are eliminated by $\eps_j L_j$.
 \end{itemize}
\item {$R=\mu\theta^2\ln(3+\frac\theta\mu)+\mu\theta^2\ln(3+\frac{\eps L}{\mu\theta^2})$}
 {(corner laminate)} : This scaling is (when relevant) attained by a construction sketched in Figure \ref{figlnthetamu2} (left). Note that this leads to two relevant contributions from the austenite part as sketched in Figure \ref{figlnthetamu2} (right){.}
 \begin{itemize}
\item[(a)]  
 $S=\mu\theta^2\ln(3+\frac\theta\mu)$: 
 We take $L_j\to\infty$, $\theta_j=\frac1{L_j^2}$, $\mu_j=\frac{1}{L_j^2\ln L_j}$, $\eps_j=\frac{\ln^4\ln L_j}{L_j^7\ln L_j}$.
 Then $\frac{\theta_j}{\mu_j}=\ln L_j$, $\mu_j\theta_j^2=\frac{1}{L_j^6\ln L_j}$, $\eps_j L_j=\frac{\ln^4\ln L_j}{L_j^6\ln L_j}$.
 In particular, $\ln \frac{\eps_j L_j}{\mu_j\theta_j^2}\sim \ln\ln\ln L_j \ll \ln\frac{\theta_j}{\mu_j}= \ln\ln L_j$, and 
 $S(\mu_j,\eps_j,\theta_j,L_j)\sim \frac{1}{L_j^6\ln L_j}\ln\ln L_j$. 
 Therefore $S(\mu_j,\eps_j,\theta_j,L_j)$ dominates $R(\mu_j,\eps_j,\theta_j,L_j)$.
 To eliminate the other regimes, we observe that $\frac{\eps_j}{\mu_j^2\theta_j^2}=L_j {(}\ln L_j{)}{(}\ln^4\ln L_j {)}$ shows that 
 $\mu_j\theta_j^2\ln(3+\frac{\eps_j}{\mu_j^2\theta_j^2})/S(\mu_j,\eps_j,\theta_j,L_j)\to\infty$. 
  Since $\eps_j L_j=\frac{\ln^4\ln L_j}{L_j^6\ln L_j}\gg S(\mu_j,\eps_j,\theta_j,L_j)$, branching, laminates and two-scale branching are ruled out. Since
 $\mu_j\theta_j^2\ln L_j\sim \mu_j\theta_j^2L_j\ln \frac{L_j}{\mu_j}\sim \frac{1}{L_j^6}\gg S(\mu_j,\eps_j,\theta_j,L_j)${, and $\theta_j^2 L_j=\frac{1}{L_j^3}$, all remaining regimes are eliminated}. 
 \item[(b)] $S=\mu\theta^2\ln(3+\frac{\eps L}{\mu\theta^2})$.
 {    
 We take $L_j\to\infty$, $\theta_j=\mu_j=\frac1{L_j^2}$, $\eps_j=\frac{\ln L_j}{L_j^7}$.
 Then $\frac{\theta_j}{\mu_j}=1$, $\mu_j\theta_j^2=\frac{1}{L_j^6}$, $\eps_j L_j=\frac{\ln L_j}{L_j^6}$.
 In particular, $\ln \frac{\eps_j L_j}{\mu_j\theta_j^2}\sim \ln\ln L_j \gg \ln(3+\frac{\theta_j}{\mu_j})=\ln 4$, and 
 $S(\mu_j,\eps_j,\theta_j,L_j)\sim \frac{1}{L_j^6}\ln\ln L_j$. 
 Therefore $S(\mu_j,\eps_j,\theta_j,L_j)$ dominates $R(\mu_j,\eps_j,\theta_j,L_j)$.
 To eliminate the other regimes, we observe that $\frac{\eps_j}{\mu_j^2\theta_j^2}=L_j \ln L_j$ shows that 
 $\mu_j\theta_j^2\ln(3+\frac{\eps_j}{\mu_j^2\theta_j^2})/S(\mu_j,\eps_j,\theta_j,L_j)\to\infty$. 
 The bottom ones {(branching, laminates and two-scale branching)} are eliminated by $\eps_j L_j/S(\mu_j,\eps_j,\theta_j,L_j)
 \sim \frac{\ln L_j}{\ln\ln L_j}\to\infty$, the top ones {(constant, affine and linear interpolation)} by $\mu_j\le 1$ {which implies} 
 ${\min\{\theta_j^2 L, \mu_j\theta_j^2\ln(3+\frac{L_j}{\mu_j})\}\geq }\mu_j\theta_j^2\ln L_j\sim  \frac{\ln L_j}{L_j^6}\gg S(\mu_j,\eps_j,\theta_j,L_j)$. 
 }
 \end{itemize}
\item $R=\eps^{2/3}\theta^{2/3}L^{1/3}+\eps L$ {(branching)}: This regime is (when relevant) attained by a branching construction sketched in Figure \ref{fig-br2-disc} (left).
 \begin{itemize} 
\item[(a)]  $S=\eps^{2/3}\theta^{2/3}L^{1/3}$. We take  $\theta_j=\mu_j=L_j=\frac12$, $\eps_j=\frac1j\to0$. Only the last three regimes {(branching, laminates and two-scale branching)} have infinitesimal energy. 
 We have $\eps_j^{2/3}\theta_j^{2/3}L_j^{1/3}{\sim j^{-2/3}}$, {whereas ${\mu_j^{1/2}\eps_j^{1/2}\theta_j L_j^{1/2}}\sim j^{-1/2}$}. At the same time, $\eps_j^{2/3}\gg \eps_j${, and hence $S(\mu_j,\eps_j,\theta_j,L_j)$ dominates $R(\mu_j,\eps_j,\theta_j,L_j)$.}
 \item[(b)] $S=\eps L$. We take $\theta_j=\mu_j=\frac12$, 
 $L_j=j^{2/3}\to\infty$,
 $\eps_j=\frac1j\to0$. 
 Then $\eps_j L_j=j^{-1/3}\to0$, ${\eps_j}^{2/3}L_j^{1/3}=j^{-4/9}\ll j^{-1/3}$, 
 $\mu_j\theta_j^2=\frac18$, ${\mu_j^{1/2}\eps_j^{1/2}\theta_j L_j^{1/2}}\sim j^{-1/6}\gg j^{-1/3}$.
  \end{itemize}
\item $R={\mu^{1/2}\eps^{1/2}\theta L^{1/2}\ln^{1/2}(3+\frac1{\theta^2})}{+}\eps L$ {(laminate)}: This regime is attained by a laminate construction as sketched in Figure \ref{fig-br2-disc} (middle).
\begin{itemize}
\item[(a)] 
{$S=\mu^{1/2}\eps^{1/2}\theta L^{1/2}\ln^{1/2}(3+\frac1{\theta^2})$:}
We take $L_j\to\infty$, $\theta_j=\frac 1{L_j}$, $\mu_j=L_j^2e^{-L_j}$, $\eps_j=\frac1{L_j^2}e^{-L_j}$.
Then $\theta_j^2L_j=\frac1L_j$, $\mu_j\theta_j^2=e^{-L_j}$, $\eps_j L_j=\frac{1}{L_j}e^{-L_j}$, and ${\mu_j^{1/2}\eps_j^{1/2}\theta_j L_j^{1/2}}= \frac{1}{L_j^{1/2}}e^{-L_j}$.
We compute in detail the last two regimes. Since $\ln(3+\frac{1}{\theta_j^2})\sim\ln L_j$, {we have} $S(\mu_j,\eps_j,\theta_j,L_j)\sim \frac{1}{L_j^{1/2}}e^{-L_j} \ln L_j$ and $\eps_j L_j/S(\mu_j,\eps_j,\theta_j,L_j)\to0$.
Since
$\ln(3+\frac{\eps_j}{\mu_j^3L_j\theta_j^2})=\ln(3+\frac{e^{2L_j}}{L_j^7})\sim L_j$, {we have} 
${\mu_j^{1/2}\eps_j^{1/2}\theta_j L_j^{1/2}}\ln^{1/2}(3+\frac{\eps_j}{\mu_j^3L_j\theta_j^2})\sim e^{-L_j}$. With $\mu_j\theta_j^2/S(\mu_j,\eps_j,\theta_j,L_j)\to\infty$ the proof is concluded.

\item[(b)]{$S=\eps L$:} {We take $L_j\to\infty$, $\theta_j=\frac 1{\ln^{1/2} L_j}$, $\mu_j=\frac{1}{L_j^{3/2}}$, $\eps_j=\frac1{L_j^{5/2}\ln^{1/2}L_j}$.
Then $\mu_j\theta_j^2=\frac{1}{L_j^{3/2}\ln L_j}$, $\eps_j L_j=
\frac{1}{L_j^{3/2}\ln^{1/2} L_j}$, and ${\mu_j^{1/2}\eps_j^{1/2}\theta_j L_j^{1/2}}=\frac{1}{L_j^{3/2}\ln^{3/4} L_j}$. 
We compute in detail the last two regimes. Since $\ln(3+\frac{1}{\theta_j^2})\sim\ln\ln L_j$, {we have} 
${\mu_j^{1/2}\eps_j^{1/2}\theta_j L_j^{1/2}} \ln^{1/2}(3+\frac{1}{\theta_j^2})\sim\frac{1}{L_j^{3/2}\ln^{3/4} L_j}\ln^{1/2}\ln L_j\ll \eps_j L_j$.
Since
$\ln(3+\frac{\eps_j}{\mu_j^3L_j\theta_j^2})=\ln(3+L_j\ln^{1/2}L_j)\sim \ln L_j$, {we have} 
${\mu_j^{1/2}\eps_j^{1/2}\theta_j L_j^{1/2}}\ln^{1/2}(3+\frac{\eps_j}{\mu_j^3L_j\theta_j^2})\sim 
\frac{1}{L_j^{3/2}\ln^{1/4} L_j}\gg \eps_j L_j$.
Further,  $\theta_j^2L_j\sim L_j/\ln L_j$, $\mu\le 1$, and 
the three terms 
$\ln(3+L_j)$, $\ln(3+\frac{\eps_j}{\mu_j^2\theta_j^2})$
and $\ln(3+\frac{\theta_j}{\mu_j})$ {behave as} $ \ln L_j$, eliminating the first five regimes {(constant, affine, linear interpolation, single truncated branching and corner laminate)}.
}
\end{itemize}
 \item $R={\mu^{1/2}\eps^{1/2}\theta L^{1/2}\ln^{1/2}(3+\frac{\eps}{\mu^3\theta^2L})}+\eps L$ {(two-scale branching)}: This regime is (when relevant) attained by a two-scale branching construction sketched in Figure \ref{fig-br2-disc} (right). 
 \begin{itemize}
\item[(a)] $S=(\mu\theta^2 \eps L \ln(3+\frac{\eps}{\mu^3\theta^2L}))^{1/2}$: 
 We take ${L_j\to} \infty$, $\theta_j=\frac1{L_j}$, $\mu_j=\frac1{L_j^{5/2}}$, $\eps_j=\mu_j^3\theta_j^2{L_j}\ln L_j=\frac{1}{L_j^{8+\frac12}}\ln L_j$.
 Then $\mu_j\theta_j^2=\frac1{L_j^{4+\frac12}}$, $\eps_j L_j=\frac{1}{L_j^{7+\frac12}}\ln L_j$, $\eps_j^{2/3}\theta_j^{2/3}L_j^{1/3}=
 \frac{1}{L_j^{6}}\ln^{2/3}L_j$,
 ${\mu_j^{1/2}\eps_j^{1/2}\theta_j L_j^{1/2}}=\frac{1}{L_j^{6}}\ln^{1/2}L_j$, $\frac{\eps}{\mu_j^3\theta_j^2L_j}=\ln L_j$, so that
 $S(\mu_j,\eps_j,\theta_j,L_j)\sim \frac{1}{L_j^{6}}\ln^{1/2} L_j \ln^{1/2}\ln L_j$ and  $\eps_j L_j/S(\mu_j,\eps_j,\theta_j,L_j)\to0$.
 The {laminate is} eliminated by
 $(\mu_j\theta_j^2\eps L_j)^{1/2}\ln^{1/2}(3+\frac1{\theta_j^2})\sim \frac{1}{L_j^{6}}\ln L_j\gg S(\mu_j,\eps_j,\theta_j,L_j)$, those with $\mu_j\theta_j^2$
 by $\mu_j\theta_j^2/S(\mu_j,\eps_j,\theta_j,L_j)\to\infty$, and 
 {$\theta_j^2L_j=\frac{1}{L_j}$}.
  \item[(b)] $S=\eps L$:
{We take $L_j\to \infty$, $\theta_j=\frac1{\ln^{1/5}L_j}$, $\mu_j=\frac{1}{L_j\ln L_j}$, 
$\eps_j=\frac{1}{L_j^2\ln^{3/5}L_j}$. Then
$\mu_j\theta_j^2=\frac{1}{L_j\ln^{7/5}L_j}$, $\eps_jL_j=\frac{1}{L_j\ln^{3/5}L_j}=\mu_j\theta_j^2\ln^{4/5}L_j$, 
$\frac{\eps_j}{\mu_j^3\theta_j^2L_j}=\ln^{14/5}L_j$, which implies that 
 $S_{TSB}:={\mu_j^{1/2}\eps_j^{1/2}\theta_j L_j^{1/2}}\ln^{1/2}(3+\frac{\eps}{\mu^3\theta^2L}) \sim (\mu_j\theta_j^2\eps_jL_j)^{1/2}\ln^{1/2} \ln L_j
\sim{\frac{\ln^{1/2} \ln L_j}{L_j\ln L_j}\sim} \mu_j\theta_j^2{(}\ln^{2/5}L_j {)}{(}\ln^{1/2}\ln L_j{)}\ll \eps_j L_j$. 
To conclude, we need to check that for all regimes $R$ entering $\mathcal I$ we have
$S_{TSB}/R\to0$. 
This is obvious for $\theta_j^2 L_j$, for 
$\mu_j\theta_j^2\ln L_j$, and for all regimes that contain an $\eps_jL_j$ scaling.  Since $\mu_j\le 1$, the regime with $\ln(3+L_j/\mu_j)$ is also irrelevant.
Since $\eps_j^{1/2}\theta_j^{3/2}=\eps_jL_j$, this is also true for the regimes that contain the $\eps_j^{1/2}\theta_j^{3/2}$ scaling{, and finally $\frac{\theta_j}{\mu_j}=L_j\ln^{4/5}L_j$ shows that this also holds true for the corner laminate.} 
}
\end{itemize}
\end{enumerate}

\subsubsection{Rough overview over some parameter ranges}\label{sec:regimesrough}
The proof of the lower bound in Section \ref{sec:lowerbound} is split into several parts that address different parameter ranges. We shall briefly motivate and sketch heuristically why different behaviours are expected in the considered ranges, and how this is reflected in our scaling law.\\[.3cm]
(i) We first consider the range in which $\eps$ is not so small, in the sense that $\eps\geq\min\{\theta^2,\mu\theta^2\}$. This is the range considered in Subsection \ref{sec:lbunif}. Roughly speaking, interfacial energy is expensive, and one expects rather uniform structures. Note that in the scaling regimes, the "uniform" constructions of constant functions (austenite, $L\theta^2$) and affine functions (majority variant of martensite, $\mu\theta^2\ln(3+L)$) scale differently in the size of the nucleus $L$. Hence, comparing these two regimes leads for large $\mu$ to a competition between $\mu$ and $L$. 
 \begin{itemize}
\item[(a)] If $\mu< 1$ then elastic strain in the austenite part is more favourable than  elastic strain in the martensite part. Further, $\eps \geq\min\{\mu\theta^2,\theta^2\}=\mu\theta^2$ implies that also interfacial energy is expensive compared to elastic energy in the austenite part. Therefore, one would expect that low energy configurations behave roughly like affine functions (corresponding to the majority variant of martensite) inside the nucleus. This is reflected in our scaling law: Since $\mu\theta^2\ln(3+L)\leq\mu\theta^2\ln(3+\frac{L}{\mu})\leq c\mu\theta^2(3+\frac{L}{\mu})\leq c(\mu\theta^2+L\theta^2)\leq cL\theta^2, \quad \frac{\eps L}{\mu\theta^2}\geq L,\text{\ and\ }\eps L\geq \mu\theta^2 L\geq c\mu\theta^2\ln(3+L)$,
we obtain $\mathcal{I}(\mu,\eps,\theta,L)\sim \mu\theta^2\ln(3+L)$, which corresponds to the affine test function. 
\item[(b)] If $\mu\geq 1$, then the behaviour is different, and the above mentioned competition between $L$ and $\mu$ becomes relevant. Note that $\mu\geq 1$ and $\eps\geq\min\{\mu\theta^2,\theta^2\}=\theta^2$ implies that $\eps L\geq\theta^2 L$, and hence all the branching and laminate regimes with a scaling $\eps L$ are not relevant. To see that the regimes with three scalings are not relevant is more complicated: We always have $\frac{\eps L}{\mu\theta^2}\geq\frac{L}{\mu}$. Hence, if $\mu\theta^2\ln(3+\frac{L}{\mu})\gtrsim \eps \theta$, we have $\mu\theta^2\ln(3+{\frac{\eps L}{\mu\theta^2}})\geq \mu\theta^2\ln(3+\frac{L}{\mu})\gtrsim \mu\theta^2\ln(3+\frac{L}{\mu})+\eps \theta$, and we are done. Otherwise, $\eps\theta\gtrsim \mu\theta^2\ln(3+\frac{L}{\mu})$ implies $\eps\gtrsim\mu\theta^2$ and hence $\frac{\eps L}{\mu\theta^2}\gtrsim L$, which yields $\mu\theta^2\ln(3+\frac{\eps L}{\mu\theta^2})\gtrsim \mu\theta^2\ln(3+L)$. Summarizing, we obtain
\[\mathcal{I}(\mu,\eps,\theta,L)\sim\min\left\{\theta^2L,\mu\theta^2\ln(3+L),\mu\theta^2\ln(3+\frac{L}{\mu})+\eps \theta\right\}. \]
The examples given in Subsection \ref{sec:regnec} show that all scalings are relevant in this parameter range.
\end{itemize}
(ii) The parameter range $\eps\leq\min\{\theta^2,\mu\theta^2\}$ is more delicate since here many contributions compete. Note that in this range, the scaling ${\theta^2 L}$ is not relevant since $4{\theta^2 L}\geq \eps^{2/3}{\theta^{2/3}}L^{1/3}+\eps L$ (recall that $L\geq 1/2$). Also the regime $\mu\theta^2\ln(3+\frac{L}{\mu})+\eps \theta$ does not occur: {Indeed, 
if} $\mu\le 1$ then  $\mu\theta^2\ln(3+L)\leq\mu\theta^2\ln(3+\frac{L}{\mu})$.
If $\mu>1$,
then $\eps^{1/2}\theta^{3/2}\le \theta^2\le \mu\theta^2$,
{$\frac{\eps L}{\mu\theta^2}\le \frac{L}{\mu}$} and $\frac{\eps}{\mu^2\theta^2}\le \frac{1}{\mu^2}\le 1$. Summarizing, $ 4\mu\theta^2\ln(3+\frac{L}{\mu})\geq\min\{\mu\theta^2\ln(3+L),\mu\theta^2\ln(3+\frac{\eps L}{\mu\theta^2})+\mu\theta^2\ln(3+\frac{\eps}{\mu^2\theta^2})+\eps^{1/2}\theta^{3/2}\}$, and hence $\mu\theta^2\ln(3+\frac{L}{\mu})+\eps\theta$ does not occur.\\
In this parameter range, the main difficulty lies in the  logarithmic corrections. Roughly speaking, complex patterns and rather uniform structures can occur, and the overall behaviour is mainly determined by the comparison of $\eps L$ and several scalings with $\mu\theta^2$. The latter, however, contain logarithmic corrections that make the comparison rather involved and lead to mixtures of different constructions. There are mainly two qualitatively different reasons for the logarithmic terms: Some of them arise (rather locally) for laminated structures in the vicinity of the left and right boundaries of the nucleus (see Lemma \ref{lem:lamoutside}). Others are due to the fact that in long nuclei affine structure deep inside the nucleus lead to non-periodic boundary conditions at the top and bottom boundaries of the nucleus and hence to elastic strain in the austenite (see Lemma \ref{lem:lemout}). 
To indicate the different phenomena, we consider several subcases, corresponding to the competition between  $\eps^{2/3}\theta^{2/3}L^{1/3}$ and $\eps L$, and the size of $\mu$. 
\begin{itemize}
\item[(a)] Assume $\eps \geq\frac{\theta^2}{L^2}$. For these rather large values of $\eps$, one expects that the relevant structures are rather uniform with few horizontal interfaces passing through the whole nucleus. This behaviour is reflected in our scaling law as follows: We have $\eps L\geq \eps^{2/3}\theta^{2/3}L^{1/3}$ which means that the branching regime behaves as $\eps L$ and that the laminate and two-scale branching regimes are not relevant. 
{
Since $\big(\frac{\eps}{\mu^2\theta^2}\big)^{1/2}\leq \frac{\eps L}{\mu\theta^2}$, the {single truncated branching} regime reduces to 
$\mu\theta^2\ln(3+\frac{\eps L}{\mu\theta^2})+\eps^{1/2} \theta^{3/2}$.
The {corner laminate} regime can then be removed. Indeed, if $\eps\le\mu^2\theta$
then $\eps^{1/2}\theta^{3/2}\le \mu\theta^2$. If $\mu^2\theta\le \eps$, then $L\le\frac\theta\mu\frac{\eps L}{\mu\theta^2}$ implies
$\ln(3+L)\le \ln(3+\frac\theta\mu)+\ln(3+\frac\theta\mu\frac{\eps L}{\mu\theta^2})${, for details see the proof of Theorem \ref{th:upperbound} (vi)d). Therefore,}}
\[\mathcal{I}(\mu,\eps,\theta,L)\sim\min\left\{\mu\theta^2\ln(3+L),\ \mu\theta^2\ln(3+\frac{\eps L}{\mu\theta^2})+\eps^{1/2} \theta^{3/2},\ \eps L\right\}. \]
{The corresponding lower bound is} the statement of Proposition \ref{propinterp} with the additional assumption $\eps L^2\geq\theta^2$. Using $\eps L\geq \eps^{1/2}\theta^{3/2}$, $\frac{\eps L}{\mu\theta^2}\leq L$ and $\mu\theta^2\ln(3+\frac{\eps L}{\mu\theta^2})\leq \mu\theta^2+\eps L$, one can see that all the scalings are relevant.
\item[(b)] Assume $\eps \leq\min\{\theta^2,\mu\theta^2,\theta^2/L^2\}$. Note that the condition $\eps \leq\frac{\theta^2}{L^2}$ implies in particular $\eps L\leq\theta^2/L$, i.e., roughly speaking, in the martensite part, a single laminate is cheaper than having a constant function or interpolating from a constant function at the left and right boundaries to an affine function deep inside the nucleus. The first two conditions indicate that interfacial energy is cheap compared to elastic energy in both, the austenite and the martensite part. However, there are various competitions between the interfacial energy, the elastic energies in the austenite and martensite parts, and the size of the nucleus. We shall outline the main points in these competitions by considering the cases $\mu\geq 1$ (i.e., elastic energy in the austenite is more expensive than in the martensite part), $\frac{1}{L}\leq\mu\leq 1$ (i.e., elastic energy in the austenite part is less expensive than in the martensite part but the nucleus is rather large), and the case $\mu\leq \frac{1}{L}$ (i.e., elastic energy in the austenite part is less expensive than in the martensite part and the nucleus is not so large).
\begin{itemize}
\item Assume $\mu\geq 1$. Then interfacial energy is rather cheap, the size of the nucleus is small in terms of $\eps$, and elastic energy in the austenite part is rather expensive. Therefore, one expects that optimal configurations form complex microstructures inside the nucleus, with little strain the austenite part. This is reflected in our scaling law as follows: We have $\eps L\leq \eps^{2/3}\theta^{2/3}L^{1/3}\leq ({\frac{\theta^2}{L^2}})^{2/3}\theta^{2/3}L^{1/3}\lesssim \mu\theta^2$ (since $L\geq 1/2$ and $\mu\geq 1$), which shows that all regimes with a scaling $\mu\theta^2\ln(3+X)$ are irrelevant. Since also $\eps^{2/3}\theta^{2/3}L^{1/3}\lesssim \mu^{1/2}\eps^{1/2}\theta L^{1/2}$, also the laminate and two-scale branching regimes with a scaling $\mu^{1/2}\eps^{1/2}\theta L^{1/2}\ln^{1/2}(3+Y)$ are not relevant, and therefore 
{we are in the branching regime of Fig.~\ref{fig-br2-disc}(left),}
$$\mathcal{I}(\mu,\eps,\theta, L)\sim \eps^{2/3}\theta^{2/3}L^{1/3}.$$
\item Assume $\frac{1}{L}\leq\mu\leq 1$. The situation is similar to the case above. However, if the size $L$ of the nucleus is large (in terms of $\mu$), then one expects a competition between the formation of complex patterns inside the nucleus and elastic energy in the austenite part in the vicinity of the left and right boundaries: This is reflected in our scaling law as follows: Again, $\eps L\leq\eps^{2/3}\theta^{2/3}L^{1/3}\leq (\frac{\theta^2}{{L^2}})^{2/3}\theta^{2/3}L^{1/3}\lesssim \mu\theta^2$ implies that all regimes with a scaling $\mu\theta^2\ln(3+X)$ are not relevant. Furthermore, $\frac{\eps}{\mu^3\theta^2 L}\leq\frac{1}{\theta^2}$ since $\mu^3 L\geq\frac{1}{L^2}\geq\frac{\theta^2}{L^2}\geq \eps$, which means that two-scale branching is more favourable than laminates. Therefore, in this parameter range 
$$\mathcal{I}(\mu,\eps,\theta, L)\sim\min\left\{\eps^{2/3}\theta^{2/3}L^{1/3},\mu^{1/2}\eps^{1/2}\theta L^{1/2}\ln^{1/2}(3+\frac{\eps}{\mu^3\theta^2 L})+\eps L\right\}.$$ Using that $\eps^{2/3}\theta^{2/3}L^{1/3}\geq \mu^{1/2}\eps^{1/2}\theta L^{1/2}\ln^{1/2}(3+\frac{\eps}{\mu^3\theta^2 L})$ is equivalent to $y\ln^{1/2}(3+y)\geq 1$ for $y:=\frac{\eps}{\mu^3 L\theta^2}$, i.e., $y\geq c$, one easily checks that all three scalings are relevant.
\item Assume finally $\mu\leq \frac{1}{L}$. This is the richest and most complex parameter range. \\[.2cm]
If $\eps$ is very small in the sense that additionally $\eps\leq\frac{\mu^{3/2}\theta^2}{L^{1/2}}\leq\frac{\mu}{L}\theta^2$ then one expects the formation of complex patterns inside the nucleus. This is reflected in our scaling law as follows: On the one hand $\eps L\leq \eps^{2/3}\theta^{2/3}L^{1/3}\leq\mu\theta^2$, which shows that branching behaves as $\eps^{2/3}\theta^{2/3}L^{1/3}$, and that all the regimes with a term $\mu\theta^2\ln(3+X)$ are not relevant. On the other hand, $\eps L\leq \mu^{1/2}\eps^{1/2}\theta L^{1/2}$. Summarizing, only branching, two-scale branching and laminates {(see Fig. \ref{fig-br2-disc})} are relevant, and $$\mathcal{I}(\mu,\eps,\theta, L)\sim\min\{\eps^{2/3}\theta^{2/3}L^{1/3},\mu^{1/2}\eps^{1/2}\theta L^{1/2}\ln^{1/2}(3+\frac{\eps}{\mu^3\theta^2 L}),\mu^{1/2}\eps^{1/2}\theta L^{1/2}\ln^{1/2}(3+\frac{1}{\theta^2})\}.$$ As above, one easily checks that all the scalings are relevant. {For the case of large $\theta\geq m_1$, the logarithms disappear since $\ln(3+\frac{1}{\theta^2})\leq c$, and the corresponding lower bound is proven in Lemma \ref{lem:thetalargebranching}.}\\[.2cm]
Let us finally address the remaining range $\frac{\mu^{3/2}\theta^2}{L^{1/2}}\leq\eps\leq\min\{\mu\theta^2,\frac{\theta^2}{L^2}\}$ in which the overall behaviour is essentially determined by the logarithmic terms. Setting  $y:=\frac{\eps}{L\mu^3\theta^2}\geq 1$, we have $y/\ln^{3}(3+y)\geq c$ and hence $\eps^{2/3}\theta^{2/3}L^{1/3}\geq c\mu^{1/2}\eps^{1/2}\theta L^{1/2}\ln^{1/2}(3+\frac{\eps}{\mu^3\theta^2 L})$, which shows that branching is not relevant. We also note that $\frac{\eps L}{\mu\theta^2}\leq\frac{\eps}{\mu^2\theta^2}$. In this case, 
\begin{eqnarray*}
\mathcal{I}(\theta,\eps,L,\mu)= &\min\Big\{
 \mu \theta^2 \ln (3+L),
 \mu\theta^2\ln (3+\frac{\eps}{\mu^2\theta^2})
 + \varepsilon^{1/2}\theta^{3/2},
 {\mu \theta^2 \ln (3+ \frac { \varepsilon {L}}{\mu \theta^{2}}) 
 +\mu\theta^2\ln (3+\frac{\theta}{\mu})}
 ,\\
& \mu^{1/2}\varepsilon^{1/2}\theta L^{1/2}  \ln^{1/2} (3+ \frac 1 {\theta^{2}})+\varepsilon L, 
 \mu^{1/2}\varepsilon^{1/2}\theta L^{1/2} \ln^{1/2} (3+ \frac{\varepsilon}{ \mu^{3}\theta^{2}{L}})+\varepsilon L
\Big\}.
\end{eqnarray*}
Here many competitions between the "more local" (lower line) and "global"(upper line)  logarithms take place, and the different contributions are treated separately in the proof of the lower bound. Precisely, in Lemma \ref{lemmabdryln}, the "global" logarithms are captured which always are in competition with a single laminate ($\eps L$). Combined with the energy required for an interpolation from a {constant} to an affine function (see Lemma \ref{lem:tildev}), this leads to the lower bound in Proposition \ref{propinterp}. The case of a single laminate requires additional care since {there the} incompatibility at the left and right boundaries lead to a competition between complex microstructures {inside the nucleus} and elastic strain in the austenite part. We point out that the situation here is (even in a scalar-valued setting) more complicated and the scaling behaviour is more complex than in the well-studied case of a vertical austenite/martensite interface with periodic  boundary conditions at the top and bottom {boundaries}. This is in particular reflected in Proposition \ref{lem:lbbranching} by the additional regime $\mu\theta^2\ln(3+\frac{\theta}{\mu})$.

\end{itemize}

\end{itemize}

\clearpage

\section{Lower bound}\label{sec:lowerbound}
The proof of the lower bound will be divided in three main parts, addressing various regimes in which qualitatively different behavior is expected from the constructions in Subsection \ref{sec:ub}. We shall briefly outline the structure of the proof:\\

In Subsection \ref{sec:lbunif}, we deal with the case that $\eps$ is not very small. Specifically, we assume that one of $\theta^2$ and $\mu\theta^2$ is below $\eps$. Roughly speaking, in this regime, one expects rather uniform structures inside the nucleus. The lower bound in this regime is given in Proposition \ref{lem:logsammelnthetasmall}. The key competition is between the bulk energy in the martensite and the bulk energy in the austenite, and is made quantitative in Lemma  \ref{lem:gammawuerfel}.  \\
 
In Subsection \ref{sec:ubcomplexstrip}, we treat the case of small $\eps$, in which we expect microstructure. This is the most interesting and richest regime, in which a variety of one- and two-scale branching patterns appear. The smallness of $\eps$ corresponds to two conditions: {firstly,} it should be such that there is at least a single interface over the entire length of the sample, as made quantitative by comparing $\eps L$ with $\mu\theta^2$ (up to a logarithmic factor, see below for the precise condition). Secondly, it must be such that the cost of a branching pattern is 
not dominated by the cost of a single straight interface, {in the sense} that $ \eps L\le \eps^{2/3}\theta^{2/3}L^{1/3}$, {which is equivalent to}
$\eps L^2\le \theta^2$.
 Roughly speaking, in this regime, one expects complex patterns inside the whole martensitic nucleus, and contributions from the austenite part only close to the left and right boundaries of $\Omega_{2L}$. The lower bound is derived in Proposition \ref{lem:lbbranching}, which builds upon a series of {Lemmata} for specific parts of the estimate. \\

 In Subsection \ref{sec:lb3}, we address the remaining part of the small-$\eps$ {range} which is not covered in Subsection \ref{sec:ubcomplexstrip}, corresponding to the cases that $\mu$ is small ({in the sense that} $\mu\theta^2 \lesssim \eps {L}$) or that $L$ is large ({in the sense that}
$\theta^2<\eps L^2$). In this case, one expects that there are parts inside the nucleus in which the displacement is affine or a single laminate. The relevant  lower bound is obtained in Proposition \ref{lem:logsammeln}.  \\ 
 
 Finally, in Subsection \ref{sec:lbconclusion}, we put together the above results and conclude the proof of the lower bound.
\\
 
{We start by making} a few general observations and definitions that will be used all over the argument.
{The condition} $\nabla u\in BV{(\Omega_{2L};\R^{2\times 2})}$ implies that $u$ has a representative {which is continuous on $\overline{\Omega_{2L}}$} (see, for example, \cite[Lemma 9]{co16}). We work with this representative, and mainly work on slices in direction $\xi=(1/4,1)$.
One important quantity is the set $\mathcal C$ of slices which are almost affine with slope $\theta$ or $\theta-1$ {(recall \eqref{eq:defuxi})}
        \begin{equation}\label{eqdefC}
\mathcal{C}:= \Big{\{}x_1\in (0,L-\xi_1)\,:\, 
\min\big\{ \|u_{x_1}^\xi(s)-u^\xi_{x_1}(0)-s\theta\|_{L^\infty((0,1))}, 
\|u_{x_1}^\xi(s)-u^\xi_{x_1}(0)-s(\theta-1)\|_{L^\infty((0,1))}{\big\}} < \frac1{16}\theta\Big\}.
        \end{equation}
        These are slices which have almost no energy in the martensitic nucleus. The boundary values on the top and bottom of the slice, however, differ by approximately $\theta$ (or $1-\theta$). Therefore, these slices generate a large energy in the austenitic matrix.
        
Correspondingly, we shall consider the set $\mathcal P$ of slices where the boundary values are close,
\begin{equation}\label{eqdefPth}
\mathcal{P} := \left\{x_1\in (0,L-\xi_1)\,:\, 
|u((x_1,0)+\xi)-u(x_1,0)| \leq 2^{-7} \theta   \right\}.
\end{equation}
These slices generate very small energy in the austenitic matrix, but cannot be low energy inside the martensitic nucleous. They can be realized either by microstructure (with energy density at least $\eps$) or by having a deformation which does not match the eigendeformation of the martensite (with energy density at least $\theta^2$). The sets {$\mathcal{C}$ and $\mathcal{P}$} are clearly disjoint. This competition and these energy contributions will be made {precise} in Lemma \ref{lem:gammawuerfel}  below.

{Over the entire lower bound we shall often focus on ``typical'' slices, and relate the one-dimensional integrals over slices to the energy via Fubini's theorem. For example, recalling the definition (\ref{eqdeffl2delta}), one has
$\int_{(0,L)\times(0,1)} |f|^2 \dxy\ge \int_0^{L-\xi_1} \|f\|_{L^2(\Delta^\xi_{x_1})}^2 \dx
\ge \alpha^2 \calL^1\left(\{x_1\in(0, L-\xi_1): \|f\|_{L^2(\Delta^\xi_{x_1})}\ge\alpha\}\right)$ for any $\alpha>0$.
}

\subsection{A lower bound in the parameter range $\theta^2\le \eps$ {or} $\mu\theta^2\le \eps$}\label{sec:lbunif}
In this {case} the structure inside the martensite is coarse, and the optimal bound is obtained considering a  path that contains the segment $(x_1,0)$ {to} $(x_1+\xi_1,1)$, and then goes back in the austenite phase, staying at a distance of order $x_1$ from the  martensite (see Figure \ref{figlemma34} and Lemma \ref{lem:gammawuerfel}).

\begin{prpstn} \label{lem:logsammelnthetasmall}
There exists $c>0$ such that for all $u\in\mathcal{X}$, $\mu>0$, $\theta\in (0,1/2]$, $\eps>0$, and {$L\in[1/2,\infty)$} 
with
\begin{equation*}
 \min\{\theta^2,\mu\theta^2\}\le \eps
\end{equation*}

there holds 
\begin{equation*}
 I(u)\ge c \min\Big\{\mu\theta^2\ln (3+L), \theta^2L,  \mu\theta^2\ln (3+\frac L\mu)+\eps\theta\Big\}.
\end{equation*}
\end{prpstn}

The key estimate, that will be useful also later in the proof, is the following.
We recall the definition of the sets $\mathcal P$ and $\mathcal C$ in 
(\ref{eqdefC}) and (\ref{eqdefPth}), 
{the definition of $\Delta^\xi_{x_1}$
in (\ref{eq:sliceDelta}) and of the $L^2(\Delta^\xi_{x_1})$ norm in (\ref{eqdeffl2delta}).} The geometry is illustrated in Figure \ref{figlemma34}.

\begin{lmm}[Bulk energy]\label{lem:gammawuerfel}
Suppose that $\theta\in(0,1/2]$, $L\in[1/2,\infty)$, $u\in{\mathcal{X}}$. The following holds:
\begin{enumerate}
\item\label{lem:gammawuerfelout} For $i\in\{1,2\}$
and for almost every $x_1\in (0,L-\xi_1)$
one has
\[\mu\int_{S_{x_1}}|\nabla u_i|^2\dcalH^1 \geq \frac18 \mu \frac{1}{1+ x_1}
 |u_i(x_1+\xi_1,1)-u_i(x_1,0)|^2,\]
 where $S_{x_1}$ is the polygonal {arc in $\R^2\setminus\Omega_{2L}$} joining the points
\begin{eqnarray*}
&&(x_1+\xi_1,1), \,(x_1+\xi_1,1+x_1),\,(-x_1,1+x_1), \,
(-x_1,-x_1), \,(x_1,-x_1), \, (x_1,0).
\end{eqnarray*}

\item\label{lem:gammawuerfelin}
{If $x_1\in[0, L-\xi_1]$ obeys} $\|\min \{|e(u)-\theta e_1 \odot e_2|,\,|e(u)+(1-\theta) e_1 \odot e_2|\}\|_{L^2(\Delta^\xi_{x_1})}^2 \leq \tilde{C}_1\theta^2$ and $|\partial_s\partial_s u_{x_1}^\xi|((0,1)) \leq \tilde{C}_1$, where 
$\tilde C_1:=2^{-13}$, {then} $x_1\in\mathcal C$.

\item\label{lem:gammawuerfelC}
For any $x_1\in\mathcal C$ we have
$|u_{x_1}^\xi(1)-u_{x_1}^\xi(0)|\ge \frac34\theta$.
In particular, the sets $\mathcal P$ and $\mathcal C$ 
are disjoint. 

\item\label{lem:gammawuerfelest}
One has
\begin{equation*}
I(u)\ge c \min\{\eps,\theta^2\} \calL^1([0,L-\xi_1]\setminus \mathcal C),
\end{equation*}
and 
\begin{equation*}
I(u)\ge  c\mu\theta^2\ln\frac{L+1-\xi_1}{\calL^1(P)+1}+c\min\{\eps,\theta^2\} \calL^1(\mathcal P).
\end{equation*}
\end{enumerate}
\end{lmm}

 \begin{figure}
\begin{center}
  \includegraphics[width=7cm]{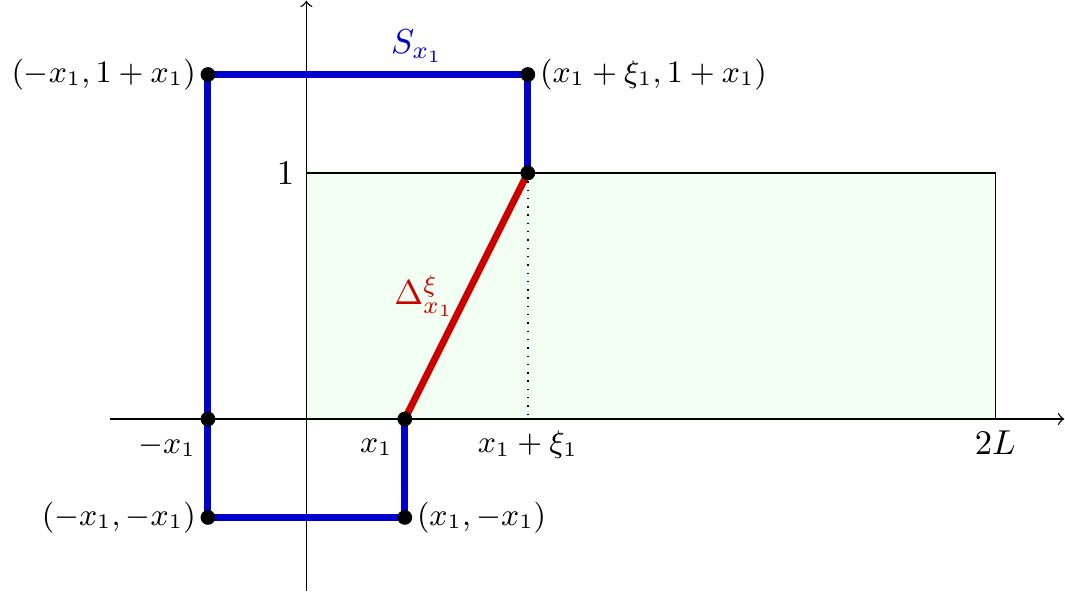} 
\end{center}
  \caption{
Sketch of the geometry in Lemma \ref{lem:gammawuerfel}. The set $S_{x_1}$ is marked blue,
the segment
 $\Delta^\xi_{x_1}$
red.
The path used here is analogous to the one in the upper bound construction, see 
Fig.~\ref{fig-regimes1} and
Lemma \ref{lem:lemout}.
 }
  \label{figlemma34}
 \end{figure}

\begin{proof}
(\ref{lem:gammawuerfelout}): The assertion follows by a direct computation, using that 
\[ \calH^1(S_{x_1})=8x_1+1+\xi_1\leq 8(1+x_1).\]
{Indeed, for} almost every $x_1 \in (0,L-\xi_1)$ we have that $u_i \in W^{1,2}(S_{x_1})$. 
It then follows that 
\begin{eqnarray*}
\int_{S_{x_1}}|\nabla u_i|^2\text{ d}\mathcal{H}^1\geq(\mathcal{H}^1(S_{x_1}))^{-1} \Big(\int_{S_{x_1}}|\nabla u_i|\text{ d}\mathcal{H}^1\Big)^2 \geq \frac{|u_i(x_1+\xi_1,1)-u_i(x_1,0)|^2}{8(1+x_1)}.
\end{eqnarray*}

(\ref{lem:gammawuerfelin}): Fix such an $x_1$ and define $v(s):= u^\xi_{x_1}(s)$. 
Since $|\partial_s \partial_s v|\le \frac18$, there is $b\in\R$ such that $|v'(s)-b|\le \frac18$ for almost all $s\in(0,1)$.
We observe that {by H\"older's inequality and \eqref{eq:estduxi}}
\begin{equation*}
 \min\{|b-\theta|, |b+(1-\theta)|\} \le 
 \int_0^1 \min\left\{|v'-\theta|, |v'+(1-\theta)|\right\}+|v'-b| \ds \le   {5}\tilde{C}_1^{1/2}\theta+\frac18 \le \frac14.
\end{equation*}
Therefore there is $\sigma\in\{0,1\}$ such that $|b+\sigma-\theta|\le \frac14$, and correspondingly $|v'(s)+\sigma-\theta|\le \frac12$ for almost all $s${. Since $|\theta-(1-\theta)|=1$ it follows that $|v'(s)+\sigma-\theta|=\min\{|v'-\theta|, |v'+1-\theta|\}$ for almost every $s$, and thus} 
\begin{equation*}
\int_0^1 |v'+\sigma-\theta|\ds=
\int_0^1 \min\{|v'-\theta|, |v'+1-\theta|\} \ds\le{5} \tilde C_1^{1/2}\theta.
\end{equation*}
Integrating we obtain
\begin{equation}
|v(s)+(\sigma-\theta)s-v(0)|\le {5}\tilde  C_1^{1/2}\theta \text{ for all $s\in(0,1)$.}
\end{equation}
Since  $\tilde C_1=2^{-13}$ this implies $x_1\in\mathcal C$. 

(\ref{lem:gammawuerfelC}): For $x_1\in\mathcal C$ we have 
$|u^\xi_{x_1}(1)-u^\xi_{x_1}(0)|\ge \min_{\sigma\in\{0,1\}} |\sigma-\theta| -2 \frac1{16}\theta
\ge {\frac34}\theta$.
The second assertion follows from 
${\frac34}\theta\le |u^\xi_{x_1}(1)-u^\xi_{x_1}(0)|
\le 4|\xi|\, |u((x_1,0)+\xi)-u(x_1,0)|$ and $|\xi|\le 2$.

{(\ref{lem:gammawuerfelest}) 
The first assertion follows from (\ref{lem:gammawuerfelin}) with Fubini's theorem. Indeed, if 
$x_1\in(0,L-\xi_1)\setminus\mathcal C$ {then} 
$\|\min \{|e(u)-\theta e_1 \odot e_2|,\,|e(u)+(1-\theta) e_1 \odot e_2|\}\|_{L^2(\Delta^\xi_{x_1})}^2 > \tilde{C}_1\theta^2$ or $\eps |\partial_s\partial_s u_{x_1}^\xi|((0,1)) > \eps \tilde{C}_1$.
Integrating over all such $x_1$ we obtain $I(u)\ge c \min\{\eps,\theta^2\} \calL^1([0,L-\xi_1]\setminus \mathcal C)$.

To prove the {other} one, for almost every $x_1\in (0, L-\xi_1)\setminus\mathcal P$
we obtain by (\ref{lem:gammawuerfelout}) that
\begin{equation*}
\mu\int_{S_{x_1}}|\nabla u|^2\text{ d}\mathcal{H}^1 \geq c \frac{\mu\theta^2}{1+ x_1}
\end{equation*}
so that, using Fubini and monotonicity of $1/(1+x_1)$,
\begin{eqnarray*}
I^\ext(u)\geq c\mu\theta^2\int_{(0, L-\xi_1)\setminus\mathcal P}\frac{1}{x_1+1}\text{ d}x_1\geq 
c\mu\theta^2\int_{\mathcal L^1(\mathcal P)}^{L-\xi_1} \frac{1}{x_1+1}\text{ d}x_1= c\mu\theta^2\ln\frac{L+1-\xi_1}{\mathcal L^1(\mathcal P)+1}.
\end{eqnarray*}
Finally, since $\mathcal P$ and $\mathcal C$ are disjoint, we have $\mathcal P\subset[0,L-\xi_1]\setminus\mathcal C$ and
$I(u)\ge c \min\{\eps,\theta^2\} \calL^1([0,L-\xi_1]\setminus \mathcal C)
\ge c \min\{\eps,\theta^2\} \calL^1(\mathcal P)$.
}
\end{proof}

\begin{proof}[Proof of Proposition \ref{lem:logsammelnthetasmall}]
{
By Lemma \ref{lem:gammawuerfel}(\ref{lem:gammawuerfelest}) we have, with $p:=\calL^1(\mathcal P)$,
\begin{equation}\label{eqboundpp34}
I(u)\ge c\mu\theta^2\ln\frac{L+1-\xi_1}{p+1} + cp \min\{\eps,\theta^2\}.  
\end{equation} 
We distinguish two cases. }

\begin{enumerate}
 \item { Assume $\mu\le1$. Then the assumption on $\eps$ implies $\mu\theta^2\le\eps$, so that $\mu\theta^2\le \min\{\eps,\theta^2\}$ and (\ref{eqboundpp34}) gives
 \begin{equation*}
I(u)\ge c \min_{p'\in[0,L-\xi_1]}\Big( \mu\theta^2\ln \frac{L+\frac34}{p'+1} + p'\mu\theta^2\Big)
=c\mu\theta^2\ln(L+\frac34) \ge c\mu\theta^2\ln (3+L),
 \end{equation*}
which concludes the proof. We used here that the expression to be minimized is nondecreasing in $p'$, hence the minimum is attained at $p'=0$.  
 }
 
 \item {
Assume $1<\mu$. Then the assumption on $\eps$ implies   $\theta^2\le\eps$, and
 (\ref{eqboundpp34}) gives
\begin{equation*}
I(u)\ge c \Big( \mu\theta^2\ln \frac{L+\frac34}{p+1} + p\theta^2\Big).
 \end{equation*}
If $p\ge\frac15 L$, then $I(u)\ge c \theta^2L$ and we are done. 
If $p=0$, then $I(u)\ge c \mu\theta^2 \ln(3+L)$ as above and we are done.
Assume now $0<p<\frac15 L$. We observe that $p\le \frac15 L$ and $\frac12\le L$ imply 
$\frac{10}9(p+1)\le L+\frac34$ and therefore
$I(u)\ge c\mu\theta^2\ln\frac{10}9\ge c \mu\theta^2$. Further, 
\begin{equation*}
I(u)\ge c \min_{p'\in[0,L-\xi_1]}\Big( \mu\theta^2\ln \frac{L+\frac34}{p'+1} + p'\theta^2\Big)
\ge c \min\left\{ \mu\theta^2\ln (3+L),\mu\theta^2\ln \Big(\frac {L+\frac34}\mu\Big),\theta^2L\right\}
 \end{equation*}
 (we used here that the only zero of the derivative is at $p'+1=\mu$, hence the minimum over $[0,\infty)$ is at $p'=\mu-1$). Since we had already proven $I(u)\ge c\mu\theta^2$, distinguishing the two cases $L/\mu\ge 3$ and $L/\mu\le 3$ gives
\begin{equation}\label{eqmutheta3lldsf}
I(u)\ge  c\min\left\{ \mu\theta^2\ln (3+L),\mu\theta^2\ln (3+\frac {L}\mu),\theta^2L\right\}.
 \end{equation}
At this point it only remains to obtain the $\eps\theta$ term. 
We are working under the assumption that $p>0$.
For almost any $x_p\in\mathcal P$ we have 
  \begin{equation*}
 \int_{(0,1)} (u_{x_p}^\xi)' (s) \ds = |u_{x_p}^\xi(1)-u_{x_p}^\xi(0)|\le \frac1{4}\theta.
\end{equation*}
If $\mathcal C=\emptyset$ then by 
 Lemma \ref{lem:gammawuerfel}(\ref{lem:gammawuerfelest})
 we have $I(u)\ge c\min\{\eps,\theta^2\}( L-\xi_1)\ge c \theta^2 L$, {and we are done}. Otherwise,  for any $x_c\in\mathcal C$ we have
\begin{equation*}
\frac34\theta\le |u_{x_c}^\xi(1)-u_{x_c}^\xi(0)|\le \int_{(0,1)} (u_{x_c}^\xi)' (s) \ds .
\end{equation*}
Therefore
\begin{equation*}
\frac12\theta\le \int_{(0,1)} | (u_{x_c}^\xi)' - (u_{x_p}^\xi)' | \ds\le
 4|D^2u|((0,L)\times(0,1))
\end{equation*}
and therefore $I(u)\ge c\eps\theta$. Recalling (\ref{eqmutheta3lldsf}) we have
\begin{equation*}
I(u)\ge c \min\left\{ \mu\theta^2\ln (3+L),\mu\theta^2\ln (3+\frac L\mu) +\eps\theta, \theta^2L\right\}
 \end{equation*}
which concludes the proof.
}
\end{enumerate}
\end{proof}

 \subsection{A lower bound in the {parameter range} $\eps L^2\le\theta^2$ and  {$\eps L\lesssim\mu \theta^{2}$}}\label{sec:ubcomplexstrip}
We turn to the case in which $\eps L^2\le \theta^2$, and formulate a lower bound on the energy restricted to  {the set}
\begin{equation}\label{eq:tildeOmegal}
{\tilde{\Omega}_\ell:=(-\infty,\ell)\times\R.  }
\end{equation}
We give these estimates for general $\ell\leq 2L$ since this does not require extra work. To prove the main theorem, however, we will only need the case $\ell=2L$.
\begin{prpstn}\label{lem:lbbranching}
 There exists  $c>0$ such that for all $u\in \mathcal{X}$,  $\theta \in (0, 1/2]$, 
{ $\eps>0$,  $\ell>0$, $\mu>0$ }
which obey
  \begin{equation*}
1\le \ell\le 2L\,,\qquad 
  \eps\ell^2\le\theta^2\,,\qquad
  \text{and}\qquad  \eps\ell\le \mu\theta^2 \min\left\{
   \ln(3+\frac{1}{\theta^2}),
  \ln(3+\frac{\eps}{\mu^3\theta^2\ell})
 \right\}
 \end{equation*}
there holds
  \begin{equation*}
  \begin{split}
  I_{\tilde{\Omega}_\ell}(u) \geq c\min 
  \Big{\{}&
{ \mu \theta^2 \ln(3+\ell)}, 
  {\mu\theta^2\ln(3+\frac\theta\mu)},
  {\varepsilon^{2/3}  \theta^{2/3}\ell^{1/3}}, 
  \\
  &
\mu^{1/2}\varepsilon^{1/2} { \theta}\ell^{1/2} (\ln (3+ \frac 1 {\theta^2}))^{1/2} , \mu^{1/2}\varepsilon^{1/2}{\theta} \ell^{1/2}  (\ln (3+ \frac \varepsilon {\mu^3\theta^2{\ell}}))^{1/2}
  \Big{\}}.
\end{split} \end{equation*}
 \end{prpstn}

The key estimate for the most difficult case, in which $\theta$ is small, is proven in Lemma \ref{lem:thetasmallbranching} below.
 A proof of similar statements has been provided by two of the authors in \cite{conti-zwicknagl:16} in the context of simplified scalar-valued models for austenite/martensite  interfaces and crystal plasticity. Our proof follows the general {strategy} of {the} proof and builds on techniques from there with two main differences: {First,} the vectorial structure requires more refined arguments; and second, the isotropic elastic modulus allows for more flexibility which is treated differently than in the corresponding model for dislocation microstructures.\\
The  vector-valued setting including a symmetrized gradient requires more careful slicing techniques than in \cite{conti-zwicknagl:16}. 
We follow a $BD$-type {approach} and consider diagonal slices instead of nearly vertical slices. 
An intuitive choice of the direction of slices would be  $(1,1)$. 
However, this would lead to problems in the case $L=1/2$. {Indeed, in this case $\Omega_{2L}=(0,1)^2$ contains only a single segment parallel to $(1,1)$ which connects the bottom and the top boundaries, and it would not be possible to choose a {`typical'} slice (in the sense of Fubini's theorem).} {Note that the energy controls only the symmetric part of the gradient of $u$, and hence we rely on $BD$-type slicing results which hold along diagonal slices but not on vertical ones.} 
For {these reasons}, we fixed the direction of the slices as 
\begin{equation*}
 \xi=(\frac14,1).
\end{equation*}

We first treat the simpler case in which $\theta$ is bounded away from zero. In this case we use a different proof which uses ideas that were introduced in \cite{chan-conti:14-1} in the geometrically nonlinear {setting} and were refined in the linear setting in \cite{diermeier:13,melching:15}. 
All these works treat only the case  $\theta=\frac 12$, but the argument can  easily {be} extended to the case $\theta \in [m_1, \frac 12 ]$. We start with this argument, which is simpler and does not require much preparation.

\begin{lmm}[The case of large $\theta$] \label{lem:thetalargebranching}
 Let $m_1 \in(0,1/2]${.} There exists $c>0$ depending only on $m_1$ such that for all $u\in \mathcal{X}$, 
  $\eps>0$,  $\ell>0$, $\mu>0$, $\theta>0$ 
which obey
  \begin{equation*}
1\le \ell\le 2L\,,\qquad 
  \eps\ell^2\le 1\,,\qquad
  \eps\ell\le \mu \,,\qquad 
  \text{and}\qquad  
  m_1\le \theta\le \frac12 
  \end{equation*}
 we have
 \[
  I_{\tilde{\Omega}_\ell}(u) \geq c\min 
  \Big{\{}
   \varepsilon^{2/3}{\ell^{1/3}}, {\mu^{1/2}\varepsilon^{1/2} \ell^{1/2} }
  \Big{\}}.
 \]
\end{lmm}
\begin{figure}
 \begin{center}
  \includegraphics[width=8cm]{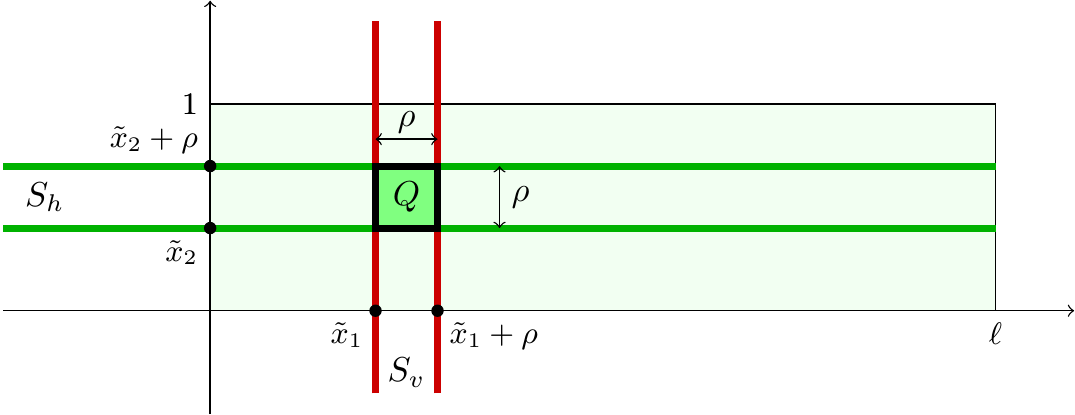}
 \end{center}
\caption{Sketch of the geometry in the proof of Lemma \ref{lem:thetalargebranching}. The horizontal and vertical stripes $S_h$ and $S_v$, as well as their intersection $Q$,  are emphasized.
\label{figlemmamm1}}
\end{figure}

\begin{proof}
Let $\sigma:\Omega_\ell\rightarrow \{\theta,-1+\theta\}$ be the function that indicates which variant is locally {attained,} i.e., such that $\min\{|e(u)-\theta e_1 \odot e_2|,|e(u)+(1-\theta)e_1\odot e_2|\} = |e(u)-\sigma e_1 \odot e_2|$ almost everywhere.
We fix $\rho \in (0,1]$, 
and choose $\tilde x_1\in [0,\ell-\rho]$, $\tilde x_2\in [0,1-\rho]$ such that the 
 infinite horizontal strip $S_h:=(-\infty,\ell)\times(\tilde{x}_2,\tilde{x}_2+\rho)$, and {the }infinite vertical strip $S_v:=(\tilde{x}_1,\tilde{x}_1+\rho)\times \R$  obey, 
 recalling the definition of $\tilde{\Omega}_{\ell}$ in \eqref{eq:tildeOmegal},
\begin{eqnarray}\label{eq:choiceQ}
I_Q(u)\leq c\ell^{-1}\rho^2 I_{\tilde{\Omega}_\ell}(u), \quad I_{S_v}(u) \leq c \ell^{-1}\rho I_{\tilde{\Omega}_\ell}(u),  \text{\ and\ } I_{S_h}(u) \leq c\rho I_{\tilde{\Omega}_\ell}(u),
\end{eqnarray}
{ where} $Q:=S_h\cap S_v\subset \Omega_\ell$ (see Figure \ref{figlemmamm1}).
By Poincar\'e's inequality we have that 
\begin{eqnarray}\label{eq:Pbigtheta}
 \|\partial_2 u_1-\langle \partial_2 u_1\rangle_Q \|_{L^1(Q)}\leq \rho |D^2u_1|(Q) \quad\text{  and } \quad\|\partial_1 u_2-\langle \partial_1 u_2\rangle_Q \|_{L^1(Q)}\leq \rho |D^2u_2|(Q)
\end{eqnarray}
where $\langle f\rangle_Q$ denotes the average of $f$ over $Q$.
Combined with H\"older's inequality, we obtain
\begin{eqnarray*}
 \|\sigma-\langle \partial_2u_1\rangle_Q -\langle \partial_1u_2\rangle_Q\|_{L^1(Q)} &\leq&  \|\sigma-\left( \partial_2u_1+\partial_1u_2\right)\|_{L^1(Q)}+ \|\partial_2u_1-\langle \partial_2u_1\rangle_Q \|_{L^1(Q)}\\
 &&+ \|\partial_1u_2-\langle \partial_1u_2\rangle_Q \|_{L^1(Q)}\\
&\leq& c\left(\rho I^{1/2}_Q(u)+\rho \varepsilon^{-1}I_Q(u)\right) \leq c\left(\rho^2\ell^{-1/2} I_{\tilde{\Omega}_\ell}^{1/2}(u)+\rho^3 \varepsilon^{-1}\ell^{-1}I_{\tilde{\Omega}_\ell}(u)\right).
\end{eqnarray*}
 If $\calL^2(\{\sigma=\theta\}\cap Q)\geq \frac 12\rho^2$, then
\[\frac{1}{2}\rho^2\left|\theta- \langle \partial_2u_1\rangle_Q-\langle \partial_1u_2\rangle_Q\right|\leq \left\|\sigma-\langle \partial_2u_1\rangle_Q -\langle \partial_1u_2\rangle_Q\right\|_{L^1(Q)}. \]
Otherwise $\calL^2(\{\sigma=\theta-1\}\cap Q)\geq \frac 12\rho^2$ and hence
\[\frac{1}{2}\rho^2|(\theta-1)- \langle \partial_2u_1\rangle_Q-\langle \partial_1u_2\rangle_Q|\leq \|\sigma-\langle \partial_2u_1\rangle_Q -\langle \partial_1u_2\rangle_Q\|_{L^1(Q)}. \]
Hence if $|\langle \partial_2u_1\rangle_Q|\leq m_1/4$ and $|\langle \partial_1u_2\rangle_Q|\leq m_1/4$, then {since $m_1\leq \theta$, $|\theta-\langle\partial_2u_1\rangle_Q-\langle\partial_1u_2\rangle_Q|\geq\theta/2$ and $|\theta-1-\langle\partial_2u_1\rangle_Q-\langle\partial_1u_2\rangle_Q|\geq\theta/2$, we deduce}
\[
 m_1 \rho^2 \leq c\left(\rho^2 \ell^{-1/2} I_{\tilde{\Omega}_\ell}^{1/2}(u)+\rho^3 \varepsilon^{-1}\ell^{-1}I_{\tilde{\Omega}_\ell}(u)\right),
\]
which implies that
\begin{align}\label{eq:lbfirstest}
 I_{\tilde{\Omega}_\ell}(u)\geq c\min\{\ell,\varepsilon \ell\rho^{-1}\}.
\end{align}
This estimate will be considered as one part of  (\ref{eq:lbfinaltest}) below. \\
It remains to consider the other cases.
If $|\langle \partial_2u_1\rangle_Q| > m_1/4$, we set  $a:= \langle \partial_2u_1\rangle_Q$ and $b:= \langle u_1 -{a} x_2\rangle_Q$. Then by Poincar\'e's and H\"older's inequalities {(see also \eqref{eq:Pbigtheta})}
\begin{align*}
 \|u_1 - a x_2 -b\|_{L^1(Q)}\leq \rho\|\nabla(u_1-\langle\partial_2u_1\rangle_Qx_2)\|_{L^1(Q)}
 &\leq \rho \|\partial_1u_1\|_{L^1(Q)} + \rho \|\partial_2u_1 -\langle \partial_2u_1 \rangle_Q\|_{L^1(Q)}\\
 &\leq \rho^2 I^{1/2}_Q(u) +\rho^2 \varepsilon^{-1} I_Q(u).
\end{align*}
By Fubini's theorem, there is $x_1^\ast\in(\tilde{x}_1,\tilde{x}_1+\rho)$ such that 
\[\int_{\tilde{x}_2}^{\tilde{x}_2+\rho}|u_1(x_1^\ast,x_2)-ax_2-b|\text{ d}x_2\leq   \rho I^{1/2}_Q(u) +\rho \varepsilon^{-1} I_Q(u)\]
and $u_1(x_1^*,\cdot)$ is the trace of $u_1$ on $\{x_1=x_1^*\}$.
We use the fundamental theorem to transfer this information to a corresponding slice on the boundary. 
Precisely, we estimate
\begin{align*}
&{
\int_{\tilde{x}_2}^{\tilde{x}_2+\rho}|u_1(0,x_2)-ax_2-b|\text{ d}x_2 }\\
&\leq 
  \int_{\tilde{x}_2}^{\tilde{x}_2+\rho}|u_1(x_1^\ast,x_2)-ax_2-b|\text{ d}x_2+\int_{\tilde{x}_2}^{\tilde{x}_2+\rho}\int_0^{x_1^\ast}|\partial_1u_1|(x_1,x_2)\text{ d}x_1 \text{ d}x_2
  \\
   &\leq c\left( \rho I^{1/2}_Q(u) +\rho \varepsilon^{-1} I_Q(u) + \ell^{1/2} \rho^{1/2} I^{1/2}_{S_h}(u)\right)\\
   &\leq c\left(  \rho^2 \ell^{-1/2} I_{\tilde{\Omega}_\ell}^{1/2}(u) + \rho^3\varepsilon^{-1}\ell^{-1} I_{\tilde{\Omega}_\ell}(u) + \rho \ell^{1/2} I_{\tilde{\Omega}_\ell}^{1/2}(u) \right) .
\end{align*}
Since $|a|\geq m_1/4$, we get by Lemma \ref{lem:interpolSC} applied to $u_1(0,\tilde x_2+\cdot)$ that 
\begin{align}\nonumber
m_1 \rho^2 
\leq c\left( \rho^2 \ell^{-1/2} I_{\tilde{\Omega}_\ell}^{1/2}(u) + \rho^3\varepsilon^{-1}\ell^{-1} I_{\tilde{\Omega}_\ell}(u) + \rho \ell^{1/2} I_{\tilde{\Omega}_\ell}^{1/2}(u)  +{m_1^{-1}}\mu^{-1}  I_{S_h}(u)\right)\\
\leq c\left( \rho^2 \ell^{-1/2} I_{\tilde{\Omega}_\ell}^{1/2}(u) + \rho^3\varepsilon^{-1}\ell^{-1} I_{\tilde{\Omega}_\ell}(u) + \rho \ell^{1/2} I_{\tilde{\Omega}_\ell}^{1/2}(u)  +\mu^{-1} \rho I_{\tilde{\Omega}_\ell}(u)\right), \label{eq:lbsecondtest}
\end{align}
where in the last step we subsumed $m_1$ in the constant $c$.
Finally, if $|\langle\partial_1u_2\rangle_Q|> m_1/4$, we proceed analogously, interchanging the indices $1$ and $2$, and obtain the estimate  
\begin{eqnarray*}
m_1\rho^2&\leq& c\left(\int_{\tilde{x}_1}^{\tilde{x}_1+\rho} |u_2(x_1,x_2^\ast)-\tilde{a}x_1-\tilde{b}|\text{ d}x_1+\int_{\tilde{x}_1}^{\tilde{x}_1+\rho}\int_0^{x_2^\ast}|\partial_2u_2|\text{ d}x + \mu^{-1}[u(\cdot ,0)]_{H^{1/2}(\tilde {x}_1,\tilde{x}_1+\rho)}^2\right)\\
&\leq& c\left(\rho^2\ell^{-1/2}I_{\tilde{\Omega}_\ell}^{1/2}(u)+\rho^3\eps^{-1}\ell^{-1}I_{\tilde{\Omega}_\ell}(u)+\rho I_{\tilde{\Omega}_\ell}^{1/2}(u) +\mu^{-1} \ell^{-1} \rho I_{\tilde{\Omega}_\ell}(u)\right)
\end{eqnarray*}
with $\tilde{a}:= \langle \partial_1u_2\rangle_Q$ and $\tilde{b}:= \langle u_2 -
{\tilde a} x_1\rangle_Q$ and an appropriately chosen $x_2^\ast\in (\tilde{x}_2,\tilde{x}_2+\rho)$.

Putting this together with \eqref{eq:lbfirstest} and \eqref{eq:lbsecondtest}, we see that for any $\rho\in (0,1)$
\begin{align}\label{eq:lbfinaltest}
 I_{\tilde{\Omega}_\ell}(u) \geq c\min \{   \ell, \varepsilon {\ell }\rho^{-1}, \ell^{-1}\rho^2 ,\mu \rho  \}
= c\min \{    \varepsilon {\ell }\rho^{-1}, \ell^{-1}\rho^2 ,\mu \rho  \},
\end{align}
where we used that $\rho\le\ell$ implies $\ell^{-1}\rho^2\le\ell$.
\\
If $\mu \geq \varepsilon^{1/3} \ell^{-1/3}$ we choose $\rho= \eps^{1/3} \ell^{2/3}$ and conclude $I_{\tilde{\Omega}_\ell}(u) \geq c \varepsilon^{2/3}\ell^{1/3}  $ ($\rho\le 1$ since $\eps\ell^2\le1$).\\
If $\mu < \varepsilon^{1/3} \ell^{-1/3}$ we choose $\rho= \mu^{-1/2}\eps^{1/2} \ell^{1/2}$ and conclude $I_{\tilde{\Omega}_\ell}(u) \geq c \mu^{1/2} \varepsilon^{1/2}\ell^{1/2} $ since $\mu^{1/2} \varepsilon^{1/2}\ell^{1/2}\leq  \ell^{-1}\rho^2=\mu^{-1}\eps$ by the assumption on $\mu$
({note that} $\rho\leq 1$ since $\mu\leq \eps\ell$).
This concludes the proof of Lemma \ref{lem:thetalargebranching}.
\end{proof}

If $\theta$ is small then a more complex procedure is needed, in order to capture the various logarithmic divergences in the energy. Before presenting the main estimate in Lemma \ref{lem:thetasmallbranching}  we need a number of preliminary results, which characterize the behavior of low-energy functions close to the corners in $(0,0)$ and $(0,1)$ and on ``good'' slices. 

 We {begin} by proving estimate for the behavior of $u_1$ close to the corners $(0,0)$ and $(0,1)$, which captures the logarithmic divergence of the matrix energy, see (\ref{equ1diffc}).  
 
\begin{lmm}\label{lemmaystin}
 {For any $c_*\in(0,1]$ there is 
 $C=C(c_*)$ such that whenever ${1\leq \ell\le 2L}$, $m\in(0,\frac14]$ and
 \begin{equation*}
 0<\mu\le \theta\le m
 \end{equation*}
one of the following holds: either
 \begin{equation*}
  \frac1C I(u)\ge \min\left\{\mu\theta^2 \ln(3+\frac\theta\mu),\mu\theta^2 \ln\frac1m,\mu\theta^2\ln(3+\ell)
  \right\}
 \end{equation*}
or 
\begin{equation}\label{equ1diffc}
\int_m^{1} \frac{|(u_1(x_1,1)-u_1(x_1,0)) |}{x_1}\dx\le
c_*\theta\ln\frac{1}{m}.
\end{equation}
}
\end{lmm}
 
 \begin{figure}
 \begin{center}
  \includegraphics[width=8cm]{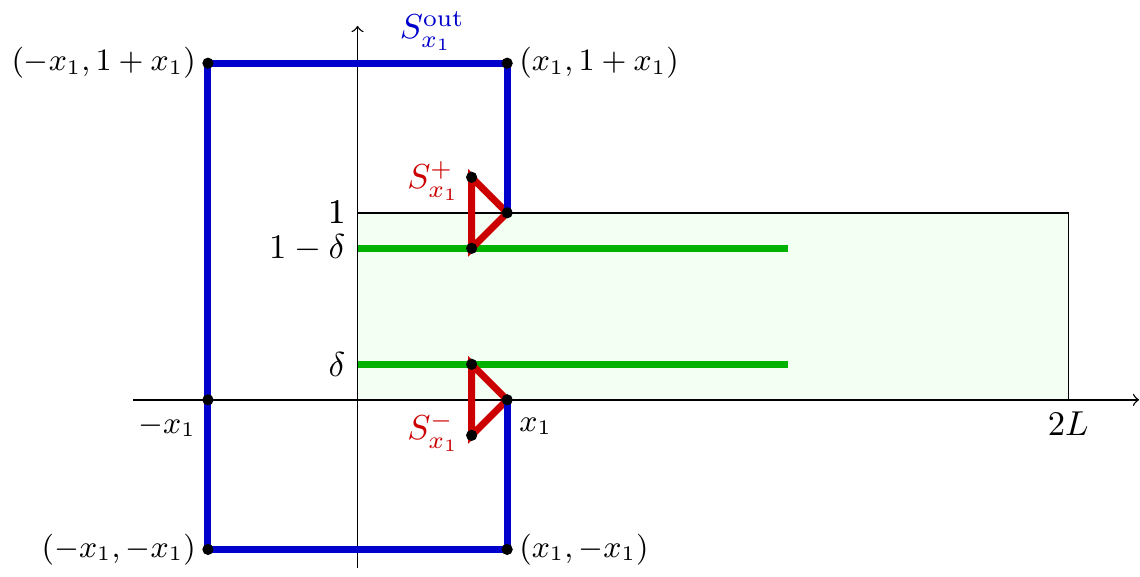}
 \end{center}
\caption{Sketch of the geometry in the proof of Lemma \ref{lemmaystin}. 
The two horizontal lines entering (\ref{equ1diffb}), at $x_2=\delta$ and $x_2=1-\delta$, are shown in green. The 
two polygonal lines $S_{x_1}^+$, $S_{x_1}^-$, 
used 
in (\ref{eqpolygonalminus}) and (\ref{eqpolygonalminus2}) are shown in red. The exterior polygonal line 
$S_{x_1}^\text{out}$ used to obtain (\ref{eqqvcp}) is shown in blue.
\label{figlh}
}
\end{figure}

\begin{proof}
We show below the following: 
either (\ref{equ1diffc}) holds or 
\begin{equation}\label{eqa5iquezq}
  \frac1c I(u)\ge \min\left\{\frac{\theta^3}q,\mu\theta^2 \ln(3+q),\mu\theta^2 \ln\frac1m
  \right\}
  \text{ for any $q\in [1,\ell]$.}
 \end{equation}
We first prove that this implies the assertion.
We let $q_*:=\frac{ 2\theta}{\mu \ln(3+\theta/\mu)}$, observe that $\mu\le\theta$ implies $q_*\ge 1$, and 
 distinguish two cases.  If $q_*\le \ell$ we set $q=q_*$. Then $\theta^3/q_\ast=\frac12\mu\theta^2\ln(3+\frac{\theta}{\mu})$. Further, observe that for any $t>0$ one has
$3+\frac{ 2t}{\ln(3+t)}\ge (3+t)^{1/2}$, which implies $\ln(3+\frac{ 2t}{\ln (3+t)})\ge \frac12\ln (3+t)$, 
and obtain
$ \mu\theta^2\ln(3+q_*)\ge \frac 12\mu\theta^2\ln(3+\frac\theta\mu)$, which concludes the proof.
If instead $q_*\ge \ell$ we set $q=\ell$ and observe that in this case
$\frac{\theta^3}{q}\ge \frac{\theta^3}{q_*}=\frac12\mu\theta^2\ln(3+\frac\theta\mu)$, 
which also concludes the proof.
Therefore it suffices to show that one of (\ref{equ1diffc})  and (\ref{eqa5iquezq}) holds.

\newcommand\ystar{\delta}
Fix $q\in [1,\ell]$.
{We can assume that
$I(u)\le 2^{-9} c_*^3 \frac{\theta^3}{q}$ (if not,  (\ref{eqa5iquezq})  holds and the proof is finished). Since $\|\partial_1 u_1\|_{L^2(\Omega_\ell)}^2\le I(u)$,
 we can choose $\ystar\in (0, \frac1{16}c_*\theta)$ such that
 $u_1(\cdot,\delta)$, $u_1(\cdot, 1-\delta)\in W^{1,2}((0,\ell))$ with
\begin{equation*}
 \int_0^{\ell} |\partial_1 u_1|^2(x_1,\ystar)+|\partial_1 u_1|^2(x_1,1-\ystar)\dx \le \frac1{32} c_*^2\frac{\theta^2}q.
\end{equation*}
This implies, setting
$\bar u_1^-:=u_1(0,\ystar)$ and 
$\bar u_1^+:=u_1(0,1-\ystar)$, 
\begin{equation}\label{equ1diffb}
 \left| u_1({\tilde{x_1}},\ystar)-\bar u_1^-\right|
 +\left| u_1({\tilde{x_1}},1-\ystar)-\bar u_1^+\right|
 \le \frac14c_* \theta \text{ for all } {\tilde{x_1}}\in [0,q].
\end{equation}
}

To shorten notation we define 
\begin{equation*}
 f(x):=
 \begin{cases}
    \min\{|e(u)(x)-\theta e_1\odot e_2|, |e(u)(x)-(\theta-1) e_1\odot e_2|\}, & \text{ if } x\in(0,\ell)\times(0,1),\\
         |\nabla u|(x), & \text{ otherwise.}
 \end{cases}
\end{equation*}
For $x_1\in (\ystar,q)$ we define $z:[0,1]\to\R$ by $z(s):=(u_1-u_2)(x_1-s,s)$. For almost every $x_1$ we can estimate, similar to  (\ref{eqestduxiv1})-(\ref{eq:estduxi}),
\begin{equation}\label{eqzprime}
 \begin{split}
 |z'(s)|=&|\partial_1 u_2+\partial_2 u_1-\partial_1 u_1-\partial_2 u_2 |(x_1-s,s)\\
 \le & \min_{\sigma\in\{\theta,\theta-1\}} |\sigma| + 
 |(e_{12}(u)+e_{21}(u)-\sigma)-(e_{11}(u)+e_{22}(u))|(x_1-s,s)\\
 \le & 1 + 2\min\left\{|e(u)(x_1-s,s)-\theta e_1\odot e_2|, |e(u)(x_1-s,s)-(\theta-1) e_1\odot e_2|\right\}
  \le 1+ 2f(x_1-s,s).
 \end{split}
\end{equation}
With multiple triangular inequalities,
\begin{equation}\label{eqpolygonalminus}
 \begin{split}
  |u_1(x_1,0)-u_1^-|\le& 
  |u_1(x_1,0)-u_1(x_1-\ystar,\ystar)|+|u_1(x_1-\ystar,\ystar)-u_1^-| \\
  \le& \
  |z(0)-z(\ystar)| + |u_2(x_1,0)-u_2(x_1-\ystar,\ystar)|+|u_1(x_1-\ystar,\ystar)-u_1^-|\\
  \le& 
  |z(0)-z(\ystar)|+ |u_2(x_1,0)-u_2(x_1-\ystar,-\ystar)|
  + |u_2(x_1-\ystar,-\ystar)-u_2(x_1-\ystar,\ystar)|\\ &+|u_1(x_1-\ystar,\ystar)-u_1^-|.
 \end{split}
\end{equation}
By the fundamental theorem of calculus, using (\ref{eqzprime}), $|\partial_2 u_2|(x)\le f(x)$ everywhere
and  $|\nabla u_2|(x)\le f(x)$ for $x_2\le 0$, 
\begin{equation}\label{eqpolygonalminus2}
 \begin{split}
  |u_1(x_1,0)-u_1^-|\le&
  |u_1(x_1-\ystar,\ystar)-u_1^-| +\delta+2 \int_{S_{x_1}^-} f \dcalH^1,
 \end{split}
\end{equation}
where  $S_{x_1}^-$ is the polygonal joining
$ (x_1,0), (x_1-\ystar,\ystar), (x_1-\ystar, -\ystar), (x_1,0)$
(see Figure \ref{figlh}). Repeating the computation on the other side with $\tilde z(s):=(u_1+u_2)(x_1-s,1-s)$ leads to
\begin{equation}
 \begin{split}
  |u_1(x_1,1)-u_1^+|\le&
  |u_1(x_1-\ystar,1-\ystar)-u_1^+| +\delta+2 \int_{S_{x_1}^+} f \dcalH^1,
 \end{split}
\end{equation}
where  $S_{x_1}^+$ is the polygonal joining
$ (x_1,1), (x_1-\ystar,1-\ystar), (x_1-\ystar, 1+\ystar), (x_1,1)$.
Adding the two, and recalling (\ref{equ1diffb}) {for $\tilde{x_1}=x_1-\delta$} yields, since $\calH^1(S_{x_1}^+)=(2+2\sqrt2)\delta\le 5\delta$,
and $2\delta\le \frac18 c_*\theta$,
\begin{equation*}
 \begin{split}
  |u_1(x_1,0)-u_1^-|+ |u_1(x_1,1)-u_1^+|\le&
  \frac14 c_*\theta +2\delta+ 2\int_{S_{x_1}^-\cup S_{x_1}^+} f \dcalH^1
\le  \frac38 c_*\theta + 8\ystar^{1/2} \left(\int_{S_{x_1}^-\cup S_{x_1}^+} f^2\dcalH^1\right)^{1/2}
 \end{split}
\end{equation*}
for almost every $x_1\in (\ystar,q)$.
We divide by $x_1$ and integrate over $x_1\in (a,b)$, for some $a,b$ with $\ystar\le a<b\le q$,
\begin{equation*}
 \begin{split}
\int_{a}^b  \frac{|u_1(x_1,0)-u_1^-|+ |u_1(x_1,1)-u_1^+|}{x_1} \dx \le&
  \frac38 c_*\theta \ln\frac{b}{a}+8\ystar^{1/2} \int_{a}^b\frac{1}{x_1} \left(\int_{S_{x_1}^-\cup S_{x_1}^+} f^2 \dcalH^1\right)^{1/2}\dx\\
  \le&
  \frac38 c_*\theta \ln\frac{b}{a} + 8\ystar^{1/2}\left(\int_a^b \frac{1}{x_1^2} \dx \right)^{1/2} \left(\int_a^b\int_{S_{x_1}^-\cup S_{x_1}^+}
f^2 \dcalH^1 \dx\right)^{1/2}.
 \end{split}
\end{equation*}
The first integral is controlled by $1/a$. For the second one we use Fubini's theorem,
\begin{equation*}
 \int_a^b\int_{S_{x_1}^-\cup S_{x_1}^+} f^2 \dcalH^1\dx\le 2 
 \int_{(0,b)\times(-\ystar,1+\ystar)}  f^2 \dxy \le 
  \max\{1,\mu^{-1}\} 
 2 I(u)= \frac{2 I(u)}{\mu},
\end{equation*}
since by assumption $\mu\le \theta\le 1$.
Therefore
\begin{equation}\label{eqabint}
\int_{a}^b  \frac{|u_1(x_1,0)-u_1^-|+ |u_1(x_1,1)-u_1^+|}{x_1} \dx \le
  \frac38 c_*\theta \ln\frac{b}{a} + \frac{12\ystar^{1/2}I(u)^{1/2}}{a^{1/2} \mu^{1/2} } 
  \text{ whenever $\ystar\le a<b\le q$.} 
\end{equation}

We first use (\ref{eqabint}) with $a=m$, $b=1$. This gives, recalling $\ystar\le \theta\le m$, 
\begin{equation*}
 \begin{split}
\int_{m}^1  \frac{|u_1(x_1,0)-u_1^-|+ |u_1(x_1,1)-u_1^+|}{x_1} \dx \le&
 \frac38 c_*\theta \ln\frac{1}{m} + \frac{12}{ \mu^{1/2} } I(u)^{1/2}.
 \end{split}
\end{equation*}
If the second term is larger than $\frac18 c_*\theta \ln\frac{1}{m}$
then $I(u)\ge c \mu\theta^2\ln \frac{1}{m}$, (\ref{eqa5iquezq}) holds and we are done. Otherwise
the right-hand side is not larger than $ \frac12 c_*\theta \ln\frac{1}{m}$, so that
\begin{equation*}
 \begin{split}
\int_{m}^1  \frac{|u_1(x_1,0)-u_1(x_1,1)|}{x_1} \dx \le&
|u_1^--u_1^+|\ln\frac1m +
  \frac12 c_*\theta \ln\frac{1}{m} .
 \end{split}
\end{equation*}
If 
\begin{equation}\label{equ1pm}
|u_1^--u_1^+| \le
  \frac12 c_*\theta
\end{equation}
then (\ref{equ1diffc}) holds and we are done. 

It remains to consider the case that 
(\ref{equ1pm}) does not hold.
For $x_1\in (0,\ell)$ we let $S_{x_1}^\text{out}:=(\partial ((-x_1,x_1)\times(-x_1,1+x_1))) \setminus (0,L)\times(0,1)$ (see Figure \ref{figlh}). Then
\begin{equation*}
  |u_1(x_1,0)-u_1(x_1,1)|\le \int_{S_{x_1}^\text{out}} |\nabla u_1|\dcalH^1
\le \int_{S_{x_1}^\text{out}} f \dcalH^1
\le \left(\calH^1(S_{x_1}^\text{out})\right)^{1/2} \left(\int_{S_{x_1}^\text{out}} f^2 \dcalH^1\right)^{1/2}.
\end{equation*}
Therefore
\begin{equation}\label{eqqvcp}
\begin{split}
\int_{1/8}^q\frac{  |u_1(x_1,0)-u_1(x_1,1)|}{x_1}\dx \le &\int_{1/8}^q\frac{ \left(\calH^1(S_{x_1}^\text{out})\right)^{1/2}}{x_1}  \left(\int_{S_{x_1}^\text{out}} f^2 \dcalH^1\right)^{1/2}\dx\\
\le& \left( \int_{1/8}^q\frac{\calH^1(S_{x_1}^\text{out})}{x_1^2} \dx\right)^{1/2}  \left(\int_0^q\int_{S_{x_1}^\text{out}} f^2 \dcalH^1\dx\right)^{1/2}\\
\le& c \mu^{-1/2}  \ln^{1/2}(8q) I^{1/2}(u),
\end{split}
\end{equation}
where in the last step we used that $\calH^1(S_{x_1}^\text{out})=8x_1+1\le 16x_1$ for $x_1\ge \frac18$.
We use  (\ref{eqabint}) with $a=\frac18$, $b=q$, and obtain
\begin{equation*}
 \begin{split}
\int_{1/8}^q  \frac{|u_1(x_1,0)-u_1^-|+ |u_1(x_1,1)-u_1^+|}{x_1} \dx \le
\frac38 c_*\theta \ln(8q) + \frac{c\ystar^{1/2}}{\mu^{1/2} } I^{1/2}(u) 
\le
\frac38 c_*\theta \ln(8q) + 
\frac{c  \ln^{1/2}(8q) I^{1/2}(u) }{\mu^{1/2}}
    \end{split}
\end{equation*}
where we used $\ystar\le 1/8$ and $q\ge 1$.
Combining with 
(\ref{eqqvcp}) and $\delta\le \frac18$ yields
\begin{equation*}
|u_1^--u_1^+|\ln (8q) 
=\int_{1/8}^q  \frac{|u_1^--u_1^+|}{x_1} \dx 
\le 
  \frac38 c_*\theta \ln(8q) +
  \frac{c  \ln^{1/2}(8q) I^{1/2}(u) }{\mu^{1/2}}. 
  \end{equation*}
Since (\ref{equ1pm}) does not hold, we have
$ I(u)\ge c  \mu\theta^2\ln(8q)\ge 
 c\mu\theta^2\ln(3+q)$
since $q\ge 1$. Therefore (\ref{eqa5iquezq}) holds and the proof is concluded also in this case.
\end{proof}

The next Lemma proves that $u_2$ is, up to a small exceptional set, very well controlled by the energy. In particular, it is significantly smaller than $\theta$, in different measures. In this Section (Lemma  \ref{lem:thetasmallbranching}) we shall use
(\ref{lem:costsu2inA46l32}). Since the proofs are naturally connected, to avoid repetition we present here also the proof of two estimates that will be used in the next Section, specifically, (\ref{lem:costsu2inA3}) in Lemma \ref{lemmainterpolationestim}
and (\ref{lem:costsu2inA46bdryln}) in  Lemma \ref{lemmabdryln}.

\newcommand\Gstar{{\mathcal G}}
\begin{lmm}[Local estimates for $u_2$]\label{lem:costsu2inA}
For all $\bar C>0$ there {exist} $C=C(\bar C)>0$
such that for any $\delta\in(0,1]$ and any $u \in W^{1,2}_{\loc}(\R^2,\R^2)$ 
there is a set $\Gstar{\subseteq[0,L-\xi_1}]$ such that
\begin{equation*}
 I(u)\ge C\min\{\mu,1\} \theta^2\calL^1(\Gstar)
\end{equation*}
and
\begin{enumerate}
\item\label{lem:costsu2inA46l32}{
for any $m\in(0, 1/4]$ and any $x_1\in  [0,L-\xi_1]\setminus \Gstar$, one has 
\begin{equation*}
\int_m^{1} \frac{|u_2(x_1-s\xi_1,1)-u_2(x_1+(1+s)\xi_1,1)|}{s} \ds\le 3\bar C\theta \ln \frac1m,
\end{equation*}
}
\item\label{lem:costsu2inA46bdryln}{
for any $m\in(0, 1/4]$ and any $x_1\in  [0,L-\xi_1]\setminus \Gstar$, one has 
\begin{equation*}
\int_m^{1} \frac{|u_2(x_1+s\xi_1,s)-u_2(x_1+(1-s)\xi_1,1-s)|}{s} \ds\le 6\bar C\theta \ln\frac1m,
\end{equation*}
}
\item\label{lem:costsu2inA3}
for any $x_1\in [0,L-\xi_1]\setminus \Gstar$  and any $\delta\in(0,1]$ one has 
\begin{equation*}
{\frac1\delta\int_{(0,\delta)} |u_2(x_1+{(1-s)\xi_1},1-s)-u_2(x_1+s\xi_1,s)|\ds \le  5\bar C\theta}.
\end{equation*}
\end{enumerate}
\end{lmm} 

\newcommand\Gvert{\mathcal G_\text{vert}}
\newcommand\Ghor{\mathcal G_\text{hor}}
\newcommand\Ghoravob{\mathcal G_\text{hor}^\text{av,+}}
\newcommand\Ghoravun{\mathcal G_\text{hor}^\text{av,-}}
\newcommand\Gvertav{\mathcal G_\text{vert}^\text{av}}

\begin{figure}
\begin{center}
  \includegraphics[width=12cm]{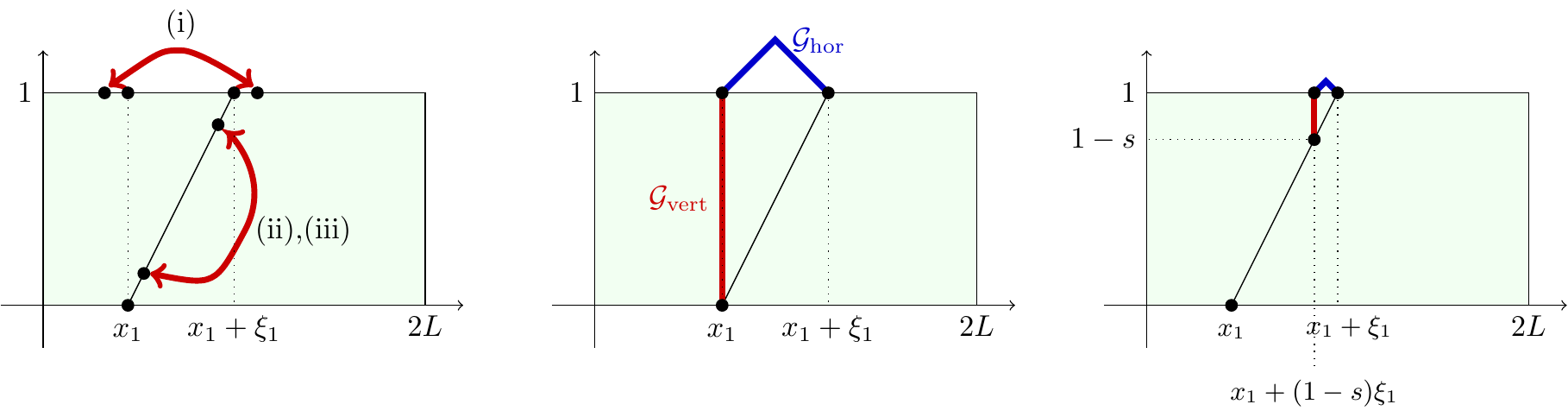} 
\end{center}
  \caption{
Sketch of the geometry in Lemma \ref{lem:costsu2inA}.
Left panel: the boundary estimate {\eqref{lem:costsu2inA46l32}} compares values across the top boundary, at distance $\xi_1(1+2s)$. The other two estimates compare points along the diagonal 
$(x_1,0)+\R\xi$. 
Middle panel: the sets $\Gvert$ and $\Ghor$ compare points which are separated vertically or horizontally, and are estimated integrating along the sketched curves,{see (\ref{eqdefgvert}) and (\ref{eqdefghor}).}
Right panel: the sets $\Ghoravob$ and $\Gvertav$ are used for similar estimated, but integrated over $s$.}
  \label{figlemma310}
 \end{figure}

\begin{proof} 
{
{\bf Step 1. Construction of $\Gstar$.}

We construct $\Gstar$ as the union of different pieces, which are all defined and estimated similarly.
}

The first one contains points with large vertical differences (see Fig. \ref{figlemma310} (middle)),
\begin{equation}\label{eqdefgvert}
\Gvert:=\left\{x_1\in[0,L-\xi_1]:  \bar C\theta\le |u_2(x_1,0)-u_2(x_1,1)|\right\}.
\end{equation}
By the fundamental theorem of calculus, for almost every $x_1$ we have
\begin{equation*}
 |u_2(x_1,0)-u_2(x_1,1)|\le \int_{(0,1)} |\partial_2 u_2|(x_1,s)\ds
\end{equation*}
which gives, using first Hölder's inequality and then Fubini's theorem,
\begin{equation*}
\bar C^2\theta^2\calL^1(\Gvert)\le \int_{(0,L-\xi_1)} |u_2(x_1,0)-u_2(x_1,1)|^2\dx \le \int_{\Omega_L} |\partial_2 u_2|^2\dxy \le I(u).
\end{equation*}
The second one contains points with large horizontal differences along the top boundary,
\begin{equation}\label{eqdefghor}
 \Ghor:=\left\{x_1\in[0,L-\xi_1]:\bar C\theta\le  |u_2(x_1,1)-u_2(x_1+\xi_1,1)| \right\}.
\end{equation}
Let $S_{x_1}$ be the polygonal joining the points {(see Fig. \ref{figlemma310} (middle))}
\begin{equation*}
 (x_1,1), (x_1+\frac12\xi_1,1+\frac12\xi_1), (x_1+\xi_1,1).
\end{equation*}
By the fundamental theorem of calculus, for almost every $x_1$ we have
\begin{equation*}
 |u_2(x_1,1)-u_2(x_1+\xi_1,1)|\le \int_{S_{x_1}} |\nabla  u_2|\dcalH^1
\le\Big(\calH^1(S_{x_1})\int_{S_{x_1}} |\nabla  u_2|^2\dcalH^1\Big)^{1/2}
\end{equation*}
so that, squaring, integrating over $\Ghor$,  and using $\calH^1(S_{x_1})=\sqrt2\xi_1=\frac{1}{2\sqrt2}$, 
\begin{equation*}
\begin{split}
\bar C^2\theta^2\calL^1(\Ghor)\le &
\int_{(0,L-\xi_1)} |u_2(x_1,1)-u_2(x_1+\xi_1,1)|^2\dx 
\le\frac1{2\sqrt2} \int_{\R}  \int_{S_{x_1}} |\nabla  u_2|^2\dcalH^1\dx\\
\le& \frac12 \int_\R \int_{(0,\frac12\xi_1)}  \big[|\nabla u_2|^2(x_1+t,1+t) 
+|\nabla u_2|^2(x_1+\xi_1-t,1+t) \big]
\dt\dx
\\\le&
 \int_{\R\times(1,2)} |\nabla u_2|^2\dxy \le \mu^{-1} I(u). 
\end{split}
\end{equation*}
The next one controls vertical fluctuations, it will be used to estimate $u_2(x_1+s\xi_1,0)-u_2(x_1+s\xi_1,s)$, and the same term on the other side. {Precisely, we set}
\begin{equation}\label{eqdefgvertav}
\Gvertav:=\left\{x_1\in[0,L-\xi_1]:  \bar C^2\theta^2\le  
 \int_{(0,1)} \int_{(0,1)}  \big( |\partial_2 u_2|^2(x_1+s\xi_1,t) + |\partial_2 u_2|^2(x_1+(1-s)\xi_1,t) \big) \ds\dt\right\}.
\end{equation}
Integrating over all $x_1\in\Gvertav$ and swapping the order of integration gives
\begin{equation*}
\bar C^2\theta^2\calL^1(\Gvertav)\le 
2 \int_{(0,1)} \ds \int_{(0,L)}\int_{(0,1)} |\partial_2 u_2|^2(x_1,t)  {\dt \dx}
\le
2\int_{\Omega_L} |\partial_2 u_2|^2\dxy \le 2I(u).
\end{equation*}

{
And finally we consider an averaged version of $\Ghor$,
\begin{equation}\label{eqdefghoravob}
\Ghoravob:=\left\{x_1\in\R: \bar C^2\theta^2\le  \int_{(-1,1)} \frac{|u_2(x_1+s\xi_1,1)-u_2(x_1,1)|^2}{|s|}
\ds  \right\}.
\end{equation}
Let $S^+_{x_1,s}$ be the polygonal line which joins the points 
\begin{equation*}
 (x_1,1),\,\, 
 (x_1+\frac12s\xi_1,1+\frac12 |s|\xi_1),\,\,
 (x_1+s\xi_1,1)
 \end{equation*}
(see Figure \ref{figlemma310}).}
{By the fundamental theorem of calculus for Sobolev functions,
Hölder's inequality and $\calH^1(S^+_{x_1,s})=  |s| \xi_1\sqrt2\le |s|$ we get
\begin{equation*}
|u_2(x_1+s\xi_1,1)-u_2(x_1,1) |
 \le \int_{S^+_{x_1,s}} |\nabla u_2|{\dcalH^1}
 \le |s|^{1/2}\left(\int_{S^+_{x_1,s}} |\nabla u_2|^2 {\dcalH^1}
\right)^{1/2} .
 \end{equation*}
Squaring, dividing by $|s|$ and integrating over $s\in(-1,1)$ gives
\begin{equation*}
\int_{(-1,1)}\frac{
|u_2(x_1+s\xi_1,1)-u_2(x_1,1) |^2}{|s|}
\ds 
 \le \int_{(-1,1)}\int_{S^+_{x_1,s}} |\nabla u_2|^2 {\dcalH^1}\ds.
 \end{equation*}
}
 {
We integrate over all $x_1\in\Ghoravob$ and obtain, using Fubini's theorem as above,
\begin{equation*}
\begin{split}
\bar C^2\theta^2 \calL^1(\Ghoravob) \le&
\int_{\R} \int_{(-1,1)} \frac{|u_2(x_1+s\xi_1,1)-u_2(x_1,1) |^2}{|s|}
\ds \dx \\
\le &\int_{\R} \int_{(-1,1)} \int_{S^+_{x_1,s}} |\nabla u_2|^2 {\dcalH^1}\ds\dx\\
\le& 2\sqrt2\int_{(-1,1)} \int_{(0,\frac12 |s| \xi_1)} \int_{\R} |\nabla u_2|^2(x_1,1+t) \dx\dt\ds \\
\le& 4\sqrt2 \int_{\R\times (1,2)} |\nabla u_2|^2\dxy\le  6 \mu^{-1} I(u).
\end{split}
 \end{equation*}
 }

The {analogue estimate} holds for
\begin{equation}\label{eqdefghoravun}
\Ghoravun:=\{x_1\in\R: \bar C^2\theta^2\le  \int_{(-1,1)} \frac{|u_2(x_1+s\xi_1,0)-u_2(x_1,0)|^2}{|s|}
\ds  \}.
\end{equation}

{We finally define
\begin{equation*}
 \Gstar :=\Gvert\cup \Ghor \cup\Gvertav  \cup
 \Ghoravob\cup  \Ghoravun\cup (\xi_1+\Ghoravob).
\end{equation*}
The previous estimates imply $I(u) \ge C_4 \min\{\theta^2,\mu\theta^2\} \calL^1(\Gstar)$, with 
$C_4:=\frac1{36} \bar C^{2}$.
}

 {\bf Step 2. Proof of (\ref{lem:costsu2inA46l32}) (estimate on the boundary).}
 For any $f\in L^2((0,1))$ one has {for $m\in(0,1)$}
 \begin{equation}\label{eqdlog1mdf}
  \int_{(m,1)} \frac{|f|}{|s|} \ds \le 
   \Big(\int_{(m,1)}\frac1{|s|}\ds\Big)^{1/2}
  \Big(\int_{(m,1)}\frac{|f|^2}{|s|}\ds\Big)^{1/2}
   \le \ln^{1/2}\frac1m \Big(\int_{(0,1)}\frac{|f|^2}{|s|}\ds\Big)^{1/2},
 \end{equation}
 and {similarly} on $(-1,-m)$. 
Therefore, for $x_1\in\R\setminus \Ghoravob$ (recall (\ref{eqdefghoravob})) we have
\begin{equation}\label{eq1m1mg0}
\begin{split}
\int_{(m,1)} \frac{|u_2(x_1+s\xi_1,1)-u_2(x_1,1)|}{|s|}\ds\le &
\bar C\theta \ln^{1/2} \frac1m ,
\end{split}
\end{equation}
{and }analogously
\begin{equation}\label{eq1m1mg1}
\int_{(m,1)} \frac{|u_2(x_1-s\xi_1,1)-u_2(x_1,1)|}{|s|}\ds\le 
   \bar C\theta \ln^{1/2} \frac1m \,.
\end{equation}
For later reference we notice that a similar computation shows that
for $x_1\in\R\setminus \Ghoravun$ (recall (\ref{eqdefghoravun})) we have
\begin{equation}\label{eq1m1mg0min}
\begin{split}
\int_{(m,1)} \frac{|u_2(x_1+s\xi_1,0)-u_2(x_1,0)|}{|s|}\ds\le &
\bar C\theta \ln^{1/2} \frac1m .
\end{split}
\end{equation}

For $x_1\in[0,L-\xi_1]\setminus \Gstar$ we have
\begin{equation*}
\begin{split}
 \int_m^{1} \frac{|u_2(x_1-s\xi_1,1)-u_2(x_1+(1+s)\xi_1,1)|}{s} \ds\le&
 \int_m^{1} \frac{|u_2(x_1-s\xi_1,1)-u_2(x_1,1)|}{s} \ds\\
&+ |u_2(x_1,1)-u_2(x_1+\xi_1,1) |
\int_m^{1} \frac1s \ds\\
&+\int_m^{1} \frac{|u_2(x_1+\xi_1,1)-u_2(x_1+(1+s)\xi_1,1)|}{s} \ds
\\
\le &3\bar C \theta \ln\frac1m,
\end{split}
\end{equation*}
where we used 
$x_1\not\in  \Ghoravob$ and (\ref{eq1m1mg1}) to estimate the first term,
$x_1\not\in  \Ghor$ to estimate the second one, 
and 
$x_1+\xi_1\not\in  \Ghoravob$ and (\ref{eq1m1mg0}) to estimate the third one, and then $\ln^{1/2}\frac1m\le\ln\frac1m$ to simplify the estimate.

This concludes the proof of (\ref{lem:costsu2inA46l32}).

{\bf Step 3. Proof of (\ref{lem:costsu2inA46bdryln}) (estimate on the diagonal).}

Let $x_1\in[0,L-\xi_1]\setminus \Gstar$.
For any $s\in(0,1)$ we have, by the fundamental theorem of calculus and Hölder's inequality,
\begin{equation*}
\begin{split}
\frac{ |u_2(x_1+s\xi_1,s)-u_2(x_1+s\xi_1,0)|^2}{s}\le&
\frac1s\left( \int_{(0,s)} |\partial_2 u_2|(x_1+s\xi_1,t)
 \dt \right)^2 \le \int_{(0,1)} |\partial_2 u_2|^2(x_1+s\xi_1,t)\dt.
\end{split}
\end{equation*}
Integrating over $s\in(m,1)$ and using $x_1\not\in\Gvertav$ (recall (\ref{eqdefgvertav})), 
\begin{equation*}
\begin{split}
\int_0^1\frac{ |u_2(x_1+s\xi_1,s)-u_2(x_1+s\xi_1,0)|^2}{s}\ds\le \bar C^2\theta^2,
\end{split}
\end{equation*}
and combining with  $x_1\not\in\Ghoravun$ (recall (\ref{eqdefghoravun}),
\begin{equation}\label{step4unten}
\begin{split}
\int_0^1\frac{ |u_2(x_1+s\xi_1,s)-u_2(x_1,0)|^2}{s}\ds\le 4\bar C^2\theta^2.
\end{split}
\end{equation}
The same estimate on the other side gives, using  $x_1+\xi_1\not\in\Ghoravob$ (recall (\ref{eqdefghoravob})), 
\begin{equation}\label{step4oben}
\begin{split}
\int_0^1\frac{ |u_2(x_1+(1-s)\xi_1,1-s)-u_2(x_1+\xi_1,1)|^2}{s}\ds\le 4\bar C^2\theta^2,
\end{split}
\end{equation}

By the triangular inequality,
\begin{equation*}
\begin{split}
 \int_m^{1} \frac{|u_2(x_1+s\xi_1,s)-u_2(x_1+(1-s)\xi_1,1-s)|}{s} \ds\le&
 \int_m^{1} \frac{|u_2(x_1+s\xi_1,s)-u_2(x_1,0)|}{s} \ds\\
&+ |u_2(x_1,0)-u_2(x_1+\xi_1,1) |
\int_m^{1} \frac1s \ds\\
&+\int_m^{1} \frac{|u_2(x_1+\xi_1,1)-u_2(x_1+(1-s)\xi_1,1-s)|}{s} \ds.
\end{split}
\end{equation*}
By (\ref{step4unten}) and (\ref{eqdlog1mdf}), the first term is estimated by $2\bar C\theta\ln^{1/2}\frac1m\le 2\bar C\theta\ln\frac1m$. 
The same holds for the last one, by  (\ref{step4oben}) and (\ref{eqdlog1mdf}).
For the middle one we use
 $x_1\not\in\Gvert$
and $x_1\not\in \Ghor$, which give
$ |u_2(x_1,0)-u_2(x_1+\xi_1,1) |\le 2\bar C\theta$. 
Adding these three estimates leads to
\begin{equation*}
\begin{split}
 \int_m^{1} \frac{|u_2(x_1+s\xi_1,s)-u_2(x_1+(1-s)\xi_1,1-s)|}{s} \ds
\le &6\bar C \theta \ln\frac1m,
\end{split}
\end{equation*}
which concludes the proof.

{\bf Step 4. Proof of (\ref{lem:costsu2inA3}) (estimate close to the boundary).}

For any $f\in L^2((0,1))$ and any $\delta\in(0,1)$ one has
 \begin{equation}\label{eqfdeltal1delta}
  \frac1\delta\int_{(0,\delta)} |f| \ds \le 
  \frac1\delta \Big(\int_{(0,\delta)}\frac{|f|^2}{|s|}\ds\Big)^{1/2}
   \Big(\int_{(0,\delta)}{|s|}\ds\Big)^{1/2}
   \le\frac1{\sqrt2}\Big(\int_{(0,1)}\frac{|f|^2}{|s|}\ds\Big)^{1/2},
 \end{equation}
 and {analogously} on $(-\delta,0)$.
Let $x_1\in[0,L-\xi_1]\setminus\Gstar$.
 Using (\ref{eqfdeltal1delta}) and (\ref{step4oben}) we obtain
 \begin{equation*}
\frac1\delta \int_{(0,\delta)} |u_2(x_1+(1-s)\xi_1,1-s)-u_2(x_1+\xi_1,1)|\ds\le \sqrt 2\bar C\theta.
\end{equation*}
Analogously, with (\ref{eqfdeltal1delta}) and (\ref{step4unten}) we obtain
 \begin{equation*}
\frac1\delta \int_{(0,\delta)} |u_2(x_1+s\xi_1,s)-u_2(x_1,0)|\ds\le \sqrt 2\bar C\theta.
\end{equation*}
As above,
 $x_1\not\in\Gvert$
and $x_1\not\in \Ghor$ give
$ |u_2(x_1,0)-u_2(x_1+\xi_1,1) |\le 2\bar C\theta$, so that {by triangle inequality}
 \begin{equation*}
\frac1\delta \int_{(0,\delta)} |u_2(x_1+(1-s)\xi_1,1-s)-u_2(x_1+s\xi_1,s)|\ds\le 5\bar C\theta
\end{equation*}
which concludes the proof of (\ref{lem:costsu2inA3}).

\end{proof}

At this point we are ready to present the main result of this Section, which basically gives the proof of the lower bound in the cases with fine microstructure and small $\theta$. 
{Following  \cite{conti-zwicknagl:16} we introduce two new parameters, $\lambda,m>0$, which will be chosen below (see the proof of Proposition \ref{lem:lbbranching}) in different ways depending on the regime. This permits to unify different parts of the proof of the lower bound.} {Roughly speaking, the parameters $\lambda$ and $m$ correspond to the length scales of the martensitic laminate deep inside the nucleus and on the vertical austenite/martensite interface.}

\begin{lmm}[The case of small $\theta$]\label{lem:thetasmallbranching}
 There exists $m_0\in (0,1/{4}]$ and $c>0$ such that for all $u\in \mathcal{X}$, 
 {
 $\theta>0$, $\mu>0$, $\varepsilon>0$, $\lambda>0$, $\ell>0$, and  $m>0$ 
 which obey}
  \begin{equation*}
  {
 \theta\le m\le m_0\,,\qquad
  1\le \ell\le 2L\,,\qquad 
\eps\le  \theta^2\lambda\,,\qquad
 \text{and}\qquad
  \lambda\le 1
  }
 \end{equation*}
 one has
\[
  I_{\tilde{\Omega}_\ell}(u) \geq c\min 
  \Big{\{}
\frac{  \eps \ell }{\lambda},
\mu  m^2,\,  
\mu  \theta^2 {\lambda}\ln \frac 1m ,\,  
  \theta^2 {\ell^{-1}} {\lambda^2} m\left( \ln \frac1 m \right)^2,
  {\mu\theta^2\ln(3+\frac\theta\mu)},
  {
  \mu\theta^2\ln(3+\ell)}
  \Big{\}}.
 \]
\end{lmm}
This proof follows the {strategy} of \cite[Section 5.2]{conti-zwicknagl:16}, with important modifications to treat both the vectorial nature of this problem and the additional logarithmic terms which appear here, due to the different boundary conditions and the fact that we do not have a hard constraint on the order parameter.

\begin{proof}[Proof of Lemma \ref{lem:thetasmallbranching}]

\textbf{Step 1: Preparation.}\\
{We start by choosing a ``good'' slice, parametrized as usual by $x_1$. 
The slice is chosen so as to have boundary values {at the upper and lower boundary of $\Omega_{2L}$} which are close together, in a sense similar to the one used in the definition of 
 the set $\mathcal P$ defined in (\ref{eqdefPth}). However, in order to capture the different logarithmic factor, we need to use a larger set, in which the difference between the boundary values is controlled in the scale of $m\ge\theta$. 
 Specifically,
we consider the set
\begin{equation*}
\mathcal{P}_* := \left\{x_1\in (0,\ell-\xi_1)\,:\, 
|u((x_1,0)+\xi)-u(x_1,0)| \leq \frac1{10} m   \right\}.
\end{equation*}}
For almost every $x_1\in (0, \ell-\xi_1)\setminus\mathcal P_*$
we obtain
by Lemma \ref{lem:gammawuerfel}(\ref{lem:gammawuerfelout})  that
\begin{equation*}
\mu\int_{S_{x_1}}|\nabla u|^2\text{ d}\mathcal{H}^1 \geq c \frac{\mu m^2}{1+ x_1}
\end{equation*}
so that, using Fubini and monotonicity of $1/(1+x_1)$ as in the proof of Lemma \ref{lem:gammawuerfel}(\ref{lem:gammawuerfelest}),
\begin{eqnarray*}
I^\ext_{\tilde{\Omega}_\ell}(u)\geq c\mu m^2\int_{(0, \ell-\xi_1)\setminus\mathcal P_*}\frac{1}{x_1+1}\text{ d}x_1\geq c\mu m^2\int_{p}^{\ell-\xi_1} \frac{1}{x_1+1}\text{ d}x_1= c\mu m^2\ln\frac{\ell+1-\xi_1}{p+1},
\end{eqnarray*}
where $p:=\calL^1(\mathcal P_*)$.
If $p\le \frac12\ell$, then, recalling $1\le\ell$, we see that
$p+1\le \frac12\ell+1\le \frac12\ell+\frac5{14}\ell+\frac9{14}=\frac67(\ell+\frac34)$. Therefore
in this case $\ln\frac{\ell+1-\xi_1}{p+1}\ge \ln\frac76>0$, so that
$I(u)\ge c\mu m^2$ and we are done. Therefore we can assume $\calL^1(\mathcal P_*)> \frac12\ell$ in the following.

{We set $\hat C:=2^{-10}$ and consider the set where the surface energy or the elastic energy are large along slices in the $\xi$ direction.
For $x_1 \in (0, \ell-\xi_1)$ we recall that 
in \eqref{eq:defuxi} we defined $u_{x_1}^\xi:[0,1]\to\R^2$ by $u_{x_1}^\xi
:=(4u_1+u_2)((x_1,0) +s\xi)$. We set
\begin{align*}
\mathcal{R} :=& \left\{x_1\in (0,\ell-\xi_1)\,:\, |\partial_s\partial_su^\xi_{x_1}|((0,1)) \ge  \frac14\hat C\lambda^{-1} \, {\text{ or }}
\left\|\min \left\{|\partial_s u^\xi_{x_1}-\theta |,\,|\partial_s u^\xi_{x_1}+(1-\theta) |\right\}\right\|_{L^2((0,1))}^2 \ge  \hat C \eps \lambda^{-1} 
\right\}.\end{align*}}
{
By  \eqref{eq:estduxi} and Fubini's theorem,
\begin{equation*}
 I_{\tilde\Omega_\ell}(u)\ge c \calL^1(\mathcal R)
 \frac\eps\lambda.
\end{equation*}
}

{
If $\calL^1(\mathcal R)\ge \frac18\ell$, then we have
$I_{\tilde\Omega_\ell}(u)\ge  
c \eps\ell \lambda^{-1}$
and the proof is concluded.}

Let $\Gstar$ be as in  Lemma \ref{lem:costsu2inA}, with $\bar C=2^{-7}$. 
If  $\calL^1(\Gstar)\ge\frac18\ell$
then 
{$I_{\tilde{\Omega}_\ell}(u)\ge c\min\{\theta^2\ell,\mu\theta^2\ell\}\geq \\c\min\{\theta^2\ell^{-1}\lambda^2m\ln^2\frac{1}{m},\,\mu\theta^2\ln(3+\ell)\}$} and we are done (recall that $\ell^{-1}\lambda^2m\ln^2\frac1m\le \ell$ and $\ln(3+\ell)\le c \ell$ by our assumptions). Therefore we can assume
$\mathcal L^1(\mathcal R\cup\Gstar)<\frac14\ell$.

We  choose $x_1^\ast \in\mathcal P_*\setminus\mathcal R\setminus \Gstar$ such that
$v:=u_{x_1^\ast}^\xi\in W^{1,2}((0,1))$ is the trace of $u$, which necessarily satisfies
\begin{equation}\label{eq:estv}
 \begin{split}
|v(1)-v(0)|\le \frac12 m, \qquad
 |\partial_s\partial_s v |((0,1)) <\frac14\hat C\lambda^{-1}\,, \quad
 \left\|\min \left\{| v'-\theta|, | v'+(1-\theta)|\right\}\right\|^2_{L^2((0,1))} <\hat C\eps\lambda^{-1} 
 \end{split}
\end{equation}
{and, since $x_1^*\not\in \Gstar$, by  Lemma \ref{lem:costsu2inA}(\ref{lem:costsu2inA46l32})
\begin{equation}\label{eq:estv2}
 \int_{(m,1)} \frac{1}{s}\left|u_2(x_1^\ast-s\xi_1,1)-u_2(x_1^\ast+(s+1)\xi_1,1)\right| \ds\le  \frac1{32}\theta\ln\frac1m.
\end{equation}
}

For $t\in\R$, we define $\omega_t:=\{s\in (0,1)\,:\, v'(s) \leq \theta-t\}$ and $\chi_t:=\chi_{\omega_t}$. We observe that
\begin{equation}\label{eq:chi12}
 |v'(s)-\theta+\chi_{1/2}(s)|=\min\left\{|v'(s)-\theta|,|v'(s)+(1-\theta)|\right\}.
\end{equation}
By the coarea formula we have
\[\int_{\frac{1}{2}}^{\frac{3}{4}}\calH^0\left(\partial\omega_t\cap(0,1)\right)\dt
\leq\int_{\R}\calH^0\left(\partial\omega_t\cap(0,1)\right)\dt
=\left|\partial_s\partial_s v\right|\left(\left(0,1\right)\right), \]
and hence, by \eqref{eq:estv},
there is $t^\ast \in (1/2,3/4)$ such that $\omega:=\omega_{t^\ast}$ consists of at most $\hat C \lambda^{-1}$ many intervals.

{
We compute
\begin{equation*}
 v(1)-v(0)=\int_{(0,1)} v'\ds =\theta-\calL^1(\omega_{1/2})+\int_{(0,1)} (v'-\theta+\chi_{{1/2}}) \ds 
\end{equation*}
which, by Hölder's inequality, (\ref{eq:estv}), $\eps\le \theta^2\lambda$, {\eqref{eq:chi12}}, and the choice of $\hat C$ gives
\begin{equation*}
| v(1)-v(0)-\theta+\calL^1(\omega_{1/2})|\le \|v'-\theta+\chi_{{1/2}}\|_{L^1((0,1))} \le {\hat{C}^{1/2}\eps^{1/2}\lambda^{-1/2}\leq}2^{-5}\theta.
\end{equation*}
}

{
Using \eqref{eq:estv}, $\omega\subset\omega_{1/2}$ and $\eps\le\theta^2\lambda$ as above, {(note that $v'(s)-\theta\in[-3/4,-1/2]$ implies $v'(s)+1-\theta\in[1/4,1/2]$)}
\begin{equation}\label{eqomegaomega12}
 \begin{split}
0\le 
 \calL^1(\omega_{1/2})-\calL^1(\omega) &\le \calL^1(\{s: v'(s)-\theta \in [-3/4, -1/2]\})
 \le 16 \int_0^1 \min\{|v'-\theta|^2,|v'+(1-\theta)|^2\} \ds \\
 &\le 16\hat C \eps \lambda^{-1}\le 16 \hat C \theta^2\le  2^{-6}\theta,
 \end{split}
\end{equation}
so that
\begin{equation}\label{eqv1v0}
| v(1)-v(0)-\theta+\calL^1(\omega)|\le {|v(1)-v(0)-\theta+\mathcal{L}^1(\omega_{1/2})|+|\mathcal{L}^1(\omega_{1/2})-\mathcal{L}^1(\omega_{1})|\leq}2^{-4}\theta.
\end{equation}
}
{We conclude that $\omega$  consists of at most $\hat C \lambda^{-1}$ many intervals and obeys {(recall \eqref{eq:estv} and $\theta\leq m$)}
\begin{equation}
 \calL^1(\omega) \le |v(1)-v(0)|+\theta+2^{-4}\theta\le 2 m.
\end{equation}}

\begin{figure}
 \begin{center}
  \includegraphics[width=\textwidth]{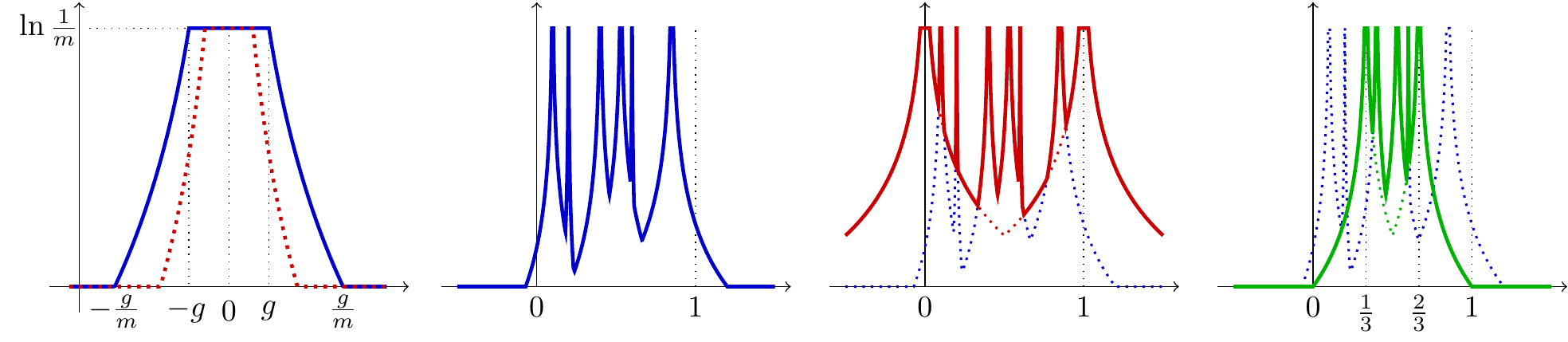}
 \end{center}
 \caption{Sketch of the construction of the test function $\psi$ in the proof of Lemma \ref{lem:thetasmallbranching}. From left to right: the first panel shows two of the functions $\psi_i$ {with different $g_i$},
 the second one the function $\hat\psi$ defined in Step 2, the third one the function $\psi$ defined in Step 3,  which equals $\ln\frac1m$ at the boundary points, and the fourth one the function $\psi$ defined in Step 4, which equals $\ln\frac1m$ at $\frac13$ and $\frac23$ and vanishes at the boundary points. In the last two panels $\hat \psi$, $\max\{\psi_B(t),\psi_B(1-t)\}$ and $\max\{\psi_T(t-\frac13), \psi_T(t-\frac23)\}$ are shown dotted for comparison. The vertical axis is compressed for clarity.}
 \label{figpsi}
\end{figure}

\noindent
\textbf{Step 2: A test function for the logarithmic scaling in the interior.}\\
We denote the connected components of $\omega$ by $(y_i-r_i, y_i+r_i)$ and define {$g_i:=\min\{r_i+\lambda m,m\}$.} {Notice that $2r_i\le\calL^1(\omega)\le 2m$ implies $r_i\le g_i$ for all $i$.}
Recall that the number $n$ of these components is at most $\hat C \lambda^{-1}$.
We then have 
\begin{eqnarray}\label{eq:gi}
\lambda m\leq g_i\leq  m \text{\qquad for all $i$},
\end{eqnarray}
 and \begin{equation*}
\sum_{i=1}^n 2g_i\le{\calL^1}(\omega)+2n\lambda m \leq 2m+2 \hat C m\leq {3}m.      
     \end{equation*}

We consider the test function $\hat\psi:\R\to\R$ defined by (see Fig.~\ref{figpsi})
\begin{equation*}\label{eq:defpsi}
  \hat\psi(t):=\max_{1, \dots n} \psi_i(t-y_i)\,,\hskip1cm
  \psi_i(t):=\left[ \ln \frac{1}{m} - \left(\ln \frac{|t|}{g_i}\right)_+ \right]_+
  {=
  \begin{cases}
  \ln\frac1m, & \text{ if } |t|\le g_i,\\
  \ln\frac{g_i}{m|t|}, & \text{ if } g_i<|t|\le \frac{g_i}{m},\\
   0, & \text{ if } |t|>\frac{g_i}{m}\,,
  \end{cases}}
\end{equation*}
{where $a_+:=\max\{a,0\}$.} 
{Since $m_0\leq 1/{4}$, we have $\ln 1/m\ge {\ln {4}>0}$ for any $m\le m_0$.} 
{Recalling $y_i\in \omega\subset(0,1)$ and $g_i\le m$ we see that $\supp\hat\psi\subset(-1,2)$.}

{We compute  
\begin{equation*}
\|\psi_i\|_{L^1(\R)}\le 2\int_0^{g_i/m} \ln\frac{g_i}{m t} \dt = \frac{2g_i}{m} \int_0^1 \ln \frac1t \dt = \frac{2g_i}{m},
\end{equation*}
}
{and analogously $\|\psi_i\|_{L^2(\R)}^2\le 4 g_i/ m$, which imply, recalling that $\sum_i 2g_i\le 3m$,}
\begin{equation*}
  \|\hat\psi\|_{L^1(\R)} \le \sum_{i=1}^n \|\psi_i\|_{L^1(\R)}\leq\frac{1}{m}\sum_{i=1}^n 2 g_i\leq  3
{\qquad\text{and}\qquad  
\|\hat\psi\|_{L^2(\R)}^2 \le 
\int_\R \max_i  |\psi_i|^2 (t-y_i)\dt \le
\sum_{i=1}^n \|\psi_i\|_{L^2(\R)}^2\leq 6 .}
  \end{equation*}  

  {To estimate the $L^2$ norm of $\hat\psi'$ we first compute
\begin{equation*}
|\psi'_i|(t)=\frac{1}{|t|}\chi_{\{g_i\le |t|\le g_i/m\}}
\end{equation*}  
and observe that $|\hat\psi'|^2(t)\le \max_i |\psi_i'|^2(t-y_i)\le \sum_i |\psi_i'|^2(t-y_i)$.
This implies}
\begin{equation*}
\|\hat\psi'\|^2_{L^2(\R)}\le \sum_{i=1}^n\int_{\R}|\psi_i'(t)|^2\dt
={2}\sum_{i=1}^n\int_{g_i}^{g_i/m} \frac1{t^2}\dt
\le 2\sum_{i=1}^n\frac{1}{g_i}
\leq \frac{2n}{\lambda m}
\le  \frac{1}{{\lambda^2} m}.
\end{equation*}

We then estimate the $H^{1/2}$ norm of $\hat\psi$. Specifically, we define an extension and estimate its homogeneous $H^1$ norm. Following 
\cite[Lemma 5.2]{conti-zwicknagl:16} we let
$\Psi_i(x):=\psi_i(|x|)$ be the radially symmetric extension of $\psi_i$ to $\R^2$ and compute
\begin{eqnarray*}
\int_{\R^2}|\nabla \Psi_i|^2\dxy=2\pi\int_{g_i}^{g_i/m} r\frac{1}{r^2}\,\dr=2\pi\ln\frac{1}{m}.
\end{eqnarray*}
We define the function $\hat\Psi(x_1,x_2):=\max_i \Psi_i(x_1,x_2-y_i)$, which obeys $\hat\Psi(0,t)=\hat\psi(t)$
for $t\in \R$ and $|\nabla \hat\Psi|(x)\le \max_i |\nabla \Psi_i|(x_1, x_2-y_i)$ for almost every $x\in\R^2$. This implies (recall \eqref{eq:gi} and $n\leq \hat{C}\lambda^{-1}\leq (2\pi)^{-1}\lambda^{-1}$)
\begin{equation*}
{\|\nabla \hat\Psi\|^2_{L^2(\R^2)}
=\int_{\R^2} |\nabla \hat\Psi|^2\,\dxy}\le
\sum_{i=1}^n\int_{{\R^2}}|\nabla \Psi_i(x_1,x_2-y_i)|^2\,\dxy
=2\pi n\ln\frac{1}{m}\leq\frac{1}{\lambda}\ln\frac{1}{m}.
\end{equation*}

\textbf{Step 3: Boundary correction for $\mu\le\theta$}.\\
In this situation we take the largest value, $\ln\frac1m$. 
{Specifically, we set $\psi(t):=\max \{\hat\psi(t),\psi_B(t), \psi_B(1-t)\}$, where
\begin{equation*}
  \psi_B(t):=\left[ \ln \frac{1}{m} - \left(\ln \frac{|t|}{m}\right)_+ \right]_+
  {=
  \begin{cases}
  \ln\frac1m, & \text{ if } |t|\le m,\\
  \ln\frac{1}{|t|}, & \text{ if } m<|t|\le 1,\\
   0, & \text{ if } |t|>1\,.
  \end{cases}}
\end{equation*}
We remark that $\psi_B$ has the same form as the functions $\psi_i$, with the only difference that the width of the central region is not $g_i\in[\lambda m,m]$ but exactly $m$. This is important to ensure symmetry of the boundary conditions.

One computes
$\|\psi_B\|_{L^1((0,1))} \le 1$, $\|\psi_B\|_{L^2((0,1))}^2\le 2$ and
$\|\psi_B'\|^2_{L^2((0,1))}\le \frac1m$.
The previous estimates for $\hat\psi$ lead then to
\begin{equation}\label{eqestpsiwithbdry}
 \|\psi\|_{L^1((0,1))}\le 5\,,\hskip4mm
 \|\psi\|^2_{L^2((0,1))}\le 10\,,\hskip4mm
 \|\psi'\|_{L^2(\R)}\le \frac{3}{\lambda m^{1/2}}\,,\hskip4mm
 \|\nabla \Psi\|_{L^2(\R^2)} \le c\frac{1}{\lambda^{1/2}}\ln^{1/2}\frac1m
\end{equation}
where $\Psi(x):=\max\{\hat\Psi (x), \psi_B(|x|), \psi_B(|x-e_2|)\}$ obeys $\supp\Psi\subseteq[-1,1]\times[-1,2]$,
$\Psi(0,t)=\psi(t)$, $\Psi(t,0)=\Psi(t,1)=\psi_B(t)$. Here it is important that $g_i\le m$, so that, in taking the maximum, $\psi_B$ is always the larger one on the top and bottom boundaries, $x_2\in\{0,1\}$.
}

{Since we are working in the case $\mu\le\theta$, from Lemma \ref{lemmaystin} with $c_*:=\frac1{8}$ we obtain the following: either
 \begin{equation*}
  \frac1c I(u)\ge \min\{\mu\theta^2 \ln(3+\frac\theta\mu),\mu\theta^2 \ln\frac1m, \mu\theta^2\ln(3+\ell)
  \}
 \end{equation*}
and (since $\lambda\le 1$) we are done, or 
\begin{equation*}
\int_m^{1} \frac{|(u_1(s,1)-u_1(s,0)) |}{s}\ds\le
 \frac1{8}\theta\ln\frac{1}{m}.
\end{equation*}
Since $\psi_B'=0$ on $(0,m)$ and $\psi_B'(s)=-1/s$ on $(m,1)$, 
this implies
\begin{equation}\label{claimstep3}
 \int_{(0,1)} \psi_B'(s) \big[u_1(s, 1)-u_1(s,0)\big]\ds \le \frac1{8} \theta\ln \frac1m .
\end{equation}
}

We now turn to $u_2$ and recall that (\ref{eq:estv2}) {implies} 
\begin{equation}\label{claimstep3b}
 \int_{(0,1)} \psi'_B(s) \big[u_2(x_1^\ast-s\xi_1,1)-u_2(x_1^\ast+(s+1)\xi_1,1)\big] \ds\le  \frac1{32}\theta\ln\frac1m.
\end{equation}

\textbf{Step 4: Energy estimate for $\mu\le\theta$.\\}

We compute, recalling that $\omega\subset\omega_{1/2}$ and that $\psi=\ln\frac1m$ on $\omega\cup\{0,1\}$,
\begin{equation*}
 \begin{split}
  \calL^1(\omega)\ln\frac1m =& \int_{(0,1)} \chi_{\omega}\psi \ds \le \int_{(0,1)} \chi_{\omega_{1/2}}\psi \ds \\
  =&
  \int_{(0,1)} (\chi_{\omega_{1/2}}-\theta+v')\psi\ds+
  \int_{(0,1)} (\theta-v') \psi \ds
  \\
  \le&
  \|v'+\chi_{\omega_{1/2}}-\theta\|_{L^2((0,1))} \|\psi\|_{L^2((0,1))} + 
  \theta \|\psi\|_{L^1((0,1))}+
  (v(0)-v(1))\ln\frac1m +\int_{(0,1)} v \psi'\ds,
 \end{split}
\end{equation*}
{where we integrated by parts in the last term. }
We recall that (\ref{eqv1v0})
gives
\begin{equation*}
( v(0)-v(1))\ln\frac1m\le
( \calL^1(\omega)-\theta+2^{-4}\theta)\ln\frac1m
\end{equation*}
and that $ \|v'+\chi_{\omega_{1/2}}-\theta\|_{L^2((0,1))}\le 2^{-5}\theta$ by (\ref{eq:estv}) {using $\eps\lambda^{-1}\leq\theta^2$}. 

{
Inserting in the previous expression gives
\begin{equation*}
 \begin{split}
 \theta\ln\frac1m 
  \le& 2^{-5}\theta\|\psi\|_{L^2((0,1))} 
  +\theta \|\psi\|_{L^1((0,1))}
  +2^{-4}\theta \ln\frac1m
  +\int_{(0,1)} v \psi'\ds.
 \end{split}
\end{equation*}
Since the estimates in (\ref{eqestpsiwithbdry}) give $\|\psi\|_{L^2((0,1))} \le 4$ and $\|\psi\|_{L^1((0,1))} \le 5$, choosing $m_0$ such that $5\le 2^{-4}\ln\frac1m$ for $m\in(0,m_0]$ we get
\begin{equation*}
 \begin{split}
\frac12 \theta\ln\frac1m 
  \le& 
  \int_{(0,1)} v \psi'\ds.
 \end{split}
\end{equation*}
}

{Using the definition of $v$, this gives
\begin{equation} \label{eq:compN}
{\frac12 \theta\ln \frac1{ m}}
\leq \int_{(0,1)} u_1((x_1^\ast,0)+s\xi) \psi'(s)\text{ d}s+ 4\int_{(0,1)} u_2((x_1^\ast,0)+s\xi) \psi'(s)\text{ d}s .  
\end{equation}
}

{At this point we distinguish two cases, depending on which of the two terms in \eqref{eq:compN} is larger. In the first case,} 
by the fundamental theorem {and the trace theorem} we have
\begin{equation*}
{\frac1{4} \theta\ln \frac1{ m}}\le
\int_{(0,1)}( u_1(x_1^\ast,0)+s\xi) \psi'(s)\text{ d}s
 =  \int_{(0,1)} \int_0^{x_1^\ast+s \xi_1} \partial_1 u_1(x_1,s\xi_2) \text{ d}x_1\, \psi'(s) \text{ d}s + \int_{(0,1)} u_1(0,s) \psi'(s) \text{ d}s.
 \end{equation*}
The first integral can be estimated by $\|\psi'\|_{L^2((0,1))}\ell^{1/2} I^{1/2}(u)$.
Recalling (\ref{claimstep3}),
\begin{equation*}
{\frac1{8} \theta\ln \frac1{ m}}\le
\|\psi'\|_{L^2((0,1))} \ell^{1/2} I^{1/2}(u)
 + \int_{(0,1)} u_1(0,s) \psi'(s) \ds
 + \int_{(0,1)} (u_1(s,1) -u_1(s,0)) \psi_B'(s) \ds.
 \end{equation*}
For brevity in we write here $I(u)$ for  $I_{\tilde{\Omega}_\ell}(u)$. 
 {We define $F_1:=(-1,1)\times(-1,2)\setminus[0,1]^2$ and observe that the last two integrals can be written as a boundary integral 
of $u_1$ times the tangential derivative $\partial_\tau\Psi$, and that $\Psi$ vanishes on the rest of the boundary of $F_1$. Therefore
\begin{align*}
{\frac1{8} \theta\ln \frac1{ m}}
 \le \int_{\partial F_1} \partial_\tau \psi u_1 \dcalH^1+\|\psi'\|_{L^2((0,1))}\ell^{1/2} I^{1/2}(u)
\end{align*}
With Lemma \ref{lemmah12app} and the estimates for $\psi$  in (\ref{eqestpsiwithbdry}) this gives
\begin{align*}
{\frac1{8} \theta\ln \frac1{ m}}
 \le \| \nabla\Psi\|_{L^2(F_1)} \|\nabla u_1\|_{L^2(F_1)}
 +\|\psi'\|_{L^2((0,1))}\ell^{1/2} I^{1/2}(u)
 \le c\frac{1}{\lambda^{1/2}}\ln^{1/2}\frac1m \mu^{-1/2} I^{1/2}(u) 
+c\frac{1}{\lambda m^{1/2}} \ell^{1/2} I^{1/2}(u),
 \end{align*}
which  gives $I(u)\ge c\min\{\mu\theta^2 \lambda \ln\frac1m, \frac{\theta^2\lambda^2 m}{\ell} \ln^2\frac1m \}$ and concludes the proof in this case.
 }

We now turn to the second case, in which the second term in \eqref{eq:compN} is the largest, and write correspondingly using the fundamental lemma of calculus
\begin{equation*}\begin{split}
{\frac1{16} \theta\ln \frac1{ m}}\le&
\int_{(0,1)} u_2(x_1^\ast+s\xi_1,s\xi_2) \psi'(s)\ds\\
 =& - \int_{(0,1)} \int_{( s\xi_2,1)} \partial_2 u_2(x_1^\ast+s\xi_1,x_2) \, \psi'(s) \dy\ds + \int_{(0,1)} u_2(x_1^\ast+s\xi_1,1) \psi'(s) \text{ d}s,
 \end{split}
\end{equation*}
 with the first integral being estimated by $2\|\psi'\|_{L^2((0,1))} I^{1/2}(u)$.

We recall that (\ref{claimstep3b}) states, after  changing variables separately in the two terms,
\begin{equation*}
- \int_{(-1,0)} u_2(x_1^\ast+s\xi_1,1)  \psi'_B(s) \ds
{-}\int_{(1,2)}u_2(x_1^\ast+s\xi_1,1) \psi'_B(s-1)  \ds\le  \frac1{32}\theta\ln\frac1m,
 \end{equation*}
 sum and obtain
\begin{equation*}
{\frac1{32} \theta\ln \frac1{ m}}\le
2 \|\psi'\|_{L^2((0,1))} I^{1/2}(u)+ \int_{\R} u_2(x_1^\ast+s\xi_1,{1}) \psi'(s) \text{ d}s .
 \end{equation*}
As above, using the estimates for $\psi$  in (\ref{eqestpsiwithbdry}) and Lemma 
\ref{lemmah12app} with 
the extension to $F_2:=(-2,\ell+2)\times {(1,2)}$, {and using $\tilde\Psi(x_1^\ast+s\xi_1,x_2):=\Psi(x_2-1,s)$,} leads to
\begin{equation*}
{\frac1{32} \theta\ln \frac1{ m}}
 \le 
 2\|\psi'\|_{L^2((0,1))} I^{1/2}(u)
 {+2\| \nabla \tilde\Psi\|_{L^2(F_2)}} \|\nabla u_2\|_{L^2(F_2)}
 \le
c\frac{1}{\lambda m^{1/2}} I^{1/2}(u)
+ c{\lambda^{-1/2}}\ln^{1/2}\frac1m \mu^{-1/2} I^{1/2}(u), 
\end{equation*}
which  gives $I(u)\ge c\min\{\mu\theta^2 \lambda\ln\frac1m, {\theta^2\lambda^2 m} \ln^2\frac1m \}$ and, since $1\le\ell$, concludes the proof also in this case.

\textbf{Step 5: Boundary correction for $\theta<\mu$}.\\
In this case we truncate, so that the new function vanishes at $s=0$ and $s=1$. We set
\begin{equation*}
  \psi_T(t):=\left[ \ln \frac{1}{m} - \left(\ln \frac{3|t|}{m}\right)_+ \right]_+
  {=
  \begin{cases}
  \ln\frac1m, & \text{ if } |t|\le \frac13 m,\\
  \ln\frac{1}{3|t|}, & \text{ if } \frac13 m<|t|\le \frac13 ,\\
   0 ,& \text{ if } |t|>\frac13 \,.
  \end{cases}}
\end{equation*}
We remark that $\psi_T$ has the same form as the functions $\psi_i$, with the only difference that the width of the central region is not $g_i\in[\lambda m,m]$ but exactly $\frac13 m$, so that $\supp \psi_T=[-\frac13,\frac13]$. This is important to ensure symmetry of the boundary conditions.
{We then define 
\begin{equation*}
\psi(t):=
\begin{cases}
\max \{\hat\psi(t),\psi_T(t-\frac13), \psi_T(t-\frac23)\}, & \text{ if } t\in(\frac13,\frac23),\\
\psi_T(t-\frac13),&\text{ if } t\le  \frac13,\\
\psi_T(t-\frac23),&\text{ if } t\ge \frac23 ,
\end{cases}
\end{equation*}
and observe that $\psi=0$ on $\R\setminus (0,1)$.
Correspondingly,
\begin{equation*}
\Psi(x):=
\begin{cases}
\max \{\hat\Psi(x),\psi_T(| (\frac13 x_1,x_2-\frac13)|) 
,\psi_T(| (\frac13 x_1,x_2-\frac23)|), & \text{ if } t\in(\frac13,\frac23),\\
\psi_T(| (\frac13 x_1,x_2-\frac13)|) 
,&\text{ if } t\le  \frac13, \\
\psi_T(| (\frac13 x_1,x_2-\frac23)|),&\text{ if }  t\ge \frac23 .
\end{cases}
\end{equation*}
It is apparent that $\Psi(0,t)=\psi(t)$, and that $\Psi=0$ outside $(-1,1)\times(0,1)$.
For $x_2=\frac13$ we observe that 
$\hat\Psi((x_1,\frac13))\le \max_i \psi_i(x_1)\le \min\{\ln\frac1m, (\ln \frac{\max g_i}{m|x_1|})_+\}\le \min\{\ln\frac 1m, (\ln\frac1{|x_1|})_+\}=\psi_T(\frac13|x_1|)$. Therefore $\Psi$ is continuous across the boundaries $x_2\in\{\frac13,\frac23\}$.
}
Repeating the same estimates as above we obtain
\begin{equation}\label{eqestpsiwithbdryd}
 \|\psi\|_{L^1((0,1))}\le 5\,,\hskip4mm
 \|\psi\|^2_{L^2((0,1))}\le 10\,,\hskip4mm
 \|\psi'\|_{L^2(\R)}\le \frac{c}{\lambda m^{1/2}}\,,\hskip4mm
 \|\nabla \Psi\|_{L^2(\R^2)} \le c{\frac{1}{\lambda^{1/2}}}\ln^{1/2}\frac1m.
\end{equation}

At this point we need to check that 
restricting to the central one-third of $(0,1)$ {we} did not {loose} most of the minority phase. Specifically, we claim that {we may assume that} 
\begin{equation}\label{eqvol1323}
 \calL^1(\omega\cap (\frac13,\frac23))\ge \frac1{8}\theta.
\end{equation}
To prove (\ref{eqvol1323}), we first
 show that
\begin{equation}\label{eqalphavalpha}
 \min_{\alpha\in\R}\|v-\alpha \|_{L^1((\frac13,\frac23))}
\le   c (\mu^{-1/2}+ \ell^{1/2}) I^{1/2}(u).
\end{equation}
Let $\alpha_1:=\int_{(-1,0)\times(0,1)} u_1(0,s)\ds$. By  Poincar\'{e}'s inequality and the trace theorem 
in $W^{1,2}$ we have
\begin{equation*}
 \| u_1(0,\cdot)-\alpha_1\|_{L^1((\frac13,\frac23))}\le 
 \| u_1(0,\cdot)-\alpha_1\|_{L^1((0,1))}\le c\|\nabla u_1\|_{L^2((-1,0)\times(0,1))} \le c\mu^{-1/2} I^{1/2}(u)
\end{equation*}
and
\begin{equation*}
\begin{split}
 \int_{1/3}^{2/3} |u_1((x_1^\ast,0)+\xi s)-\alpha_1| \ds
 \le &
\int_{1/3}^{2/3} |u_1(0,s)-\alpha_1| \ds
+\int_{1/3}^{2/3}\int_{(0,x_1^\ast+\xi_1)} |\partial_1 u_1|(t,s) \dt\ds\\
\le& c\mu^{-1/2} I^{1/2}(u) + c \ell^{1/2} I^{1/2}(u).
\end{split}
 \end{equation*}
Analogously, with $\alpha_2:=\int_{(x_1^\ast, x_1^\ast+1)\times(-1,0)} u_2(0,s)\ds$,
\begin{equation*}
\int_{1/3}^{2/3} |u_2((x_1^\ast,0)+\xi s)-\alpha_2| \ds
  \le c\mu^{-1/2} I^{1/2}(u) + c  I^{1/2}(u).
 \end{equation*}
Recalling that $v(s)=(u_1+4u_2)((x_1^\ast,0)+\xi s)$ concludes the proof of (\ref{eqalphavalpha}) {since $\ell\geq 1$}.

We now prove  (\ref{eqvol1323}). If it does not hold, then
\begin{equation*}
\begin{split}
 \int_{1/3}^{2/3} |v'-\theta| \ds \le&
 \int_{1/3}^{2/3} |v'-\theta+\chi_{1/2}| \ds +\calL^1(\omega_{1/2}\cap(\frac13,\frac23)) \\
 \le&
 \int_{(0,1)} |v'-\theta+\chi_{1/2}| \ds +\calL^1(\omega_{1/2}\setminus\omega) 
 +\calL^1(\omega\cap(\frac13,\frac23))\\
 \le&
\hat C^{1/2} \theta+2^{-{6}}\theta + \frac18\theta\le \frac14\theta
 \end{split}
\end{equation*}
where we used \eqref{eq:chi12}, \eqref{eq:estv}, $\eps\lambda^{-1}\leq\theta^2$ and  (\ref{eqomegaomega12}).
This implies $v(s)-v(s')\ge \theta(s-s')-\frac14\theta$ for all $s,s'\in(\frac13,\frac23)$, and therefore
\begin{equation*}
 \min_{\alpha\in\R}\|v-\alpha \|_{L^1((\frac13,\frac23))}\ge 
 \frac12 \int_{0}^{1/6}|v(\frac12+s)-v(\frac12-s)|\ds
\ge \frac12\theta  \int_{1/8}^{1/6} (2s -\frac14) ds =c\theta.  
\end{equation*}
Recalling (\ref{eqalphavalpha}), this implies $I(u)\ge c \min\{\mu\theta^2, \ell^{-1}\theta^2\}$ which, since $\lambda^2m\ln^2\frac1m\le 1$  and $\theta\le \mu$, concludes the proof. Therefore we can assume  that (\ref{eqvol1323}) holds.

\textbf{Step 6: Energy estimate for $\theta<\mu$.\\}
{
The computation is similar to Step 4, but with significant differences in the treatment of the boundary terms. }
Recalling   (\ref{eqvol1323}), that  $\psi=\ln\frac1m$ 
and $-v' \psi\geq t^\ast \ln \frac{1}{m}$ on $\omega\cap (\frac13,\frac23)$, with $t^\ast\ge \frac12$, and $\psi\geq0$,  we have 
\begin{align*}
{\frac1{16} \theta\ln \frac1{ m}}&\le
t^\ast \mathcal{L}^1(\omega\cap (\frac13,\frac23))  \ln \frac1{ m}
\le -\int_{(0,1)} v'(s) \psi(s)\text{ d}s+\int_{\{v'\geq 0\}\cap(0,1)}v'(s) \psi(s) \text{ d}s.
\end{align*}
{First we observe that
\begin{equation*}
\begin{split}
 \int_{\{v'\geq 0\}\cap(0,1)}v'(s) \psi(s) \text{ d}s
& \le \theta \|\psi\|_{L^1((0,1))} + 
 \| \min\{|v'-\theta|, |v'+(1-\theta)|\}\|_{L^2((0,1))} \|\psi\|_{L^2((0,1))}
 \\&\le 5\theta + \hat C^{1/2}\theta 10^{1/2}\le\frac1{32}\theta\ln\frac1m ,
\end{split}
 \end{equation*}
 where in the first step we used (\ref{eqestpsiwithbdryd}) and in the second we assumed that $m_0$ is chosen {such} that $5+\hat C^{1/2}10^{1/2}\le\frac1{32}\ln\frac1m $.
}
Inserting in the previous expression and integrating by parts we get
\begin{equation*}
 \begin{split}
\frac1{32} \theta\ln\frac1m 
  \le&
  \int_{(0,1)} v \psi'\ds.
 \end{split}
\end{equation*}
The rest of the proof is very close to the one of Step 4, with some simplifications in the treatment of the exterior field. For the convenience of the reader we repeat the computation here. Recalling the definition of $v$,
\begin{equation} \label{eq:compNdue}
{\frac1{32} \theta\ln \frac1{ m}}
\leq  \int_{(0,1)} u_1((x_1^\ast,0)+s\xi) \psi'(s)\text{ d}s+ 4\int_{(0,1)} u_2((x_1^\ast,0)+s\xi) \psi'(s)\text{ d}s .  
\end{equation}
If the first term in \eqref{eq:compNdue} is  larger than the second, 
by the fundamental theorem {and the trace theorem}
\begin{equation*}
{\frac1{64} \theta\ln \frac1{ m}}\le
\int_{(0,1)}u_1(x_1^\ast,0)+s\xi) \psi'(s)\text{ d}s
 =  \int_{(0,1)} \int_0^{x_1^\ast+s \xi_1} \partial_1 u_1(x_1,s\xi_2) \text{ d}x_1\, \psi'(s) \text{ d}s + \int_{(0,1)} u_1(0,s) \psi'(s) \text{ d}s,
 \end{equation*}
 {with the first integral being estimated by $\|\psi'\|_{L^2((0,1))} \ell^{1/2}I^{1/2}(u)$.
}

{Letting $F_3:=(-1,0)\times(0,1)$ and recalling that $\Psi(0,t)=\psi(t)$  and $\Psi(-1,t)=\Psi(-t,0)=\Psi(-t,1)=0$ for $t\in(0,1)$, 
we see that the last two integrals can be written as a boundary integral 
of $u_1$ times the tangential derivative $\partial_\tau\Psi$, and that $\Psi$ vanishes on the rest of the boundary of $F_3$. Therefore
\begin{align*}
{\frac1{64} \theta\ln \frac1{ m}}
 \le \int_{\partial F_3} u_1\, \partial_\tau \psi  \dcalH^1+\|\psi'\|_{L^2((0,1))} \ell^{1/2}I^{1/2}(u).
\end{align*}
With Lemma \ref{lemmah12app} and the estimates for $\psi$  in (\ref{eqestpsiwithbdryd}) this gives
\begin{align*}
{\frac1{64} \theta\ln \frac1{ m}}
 \le \| \nabla\Psi\|_{L^2(F_3)} \|\nabla u_1\|_{L^2(F_3)}
 +\|\psi'\|_{L^2((0,1))} \ell^{1/2} I^{1/2}(u)
 \le c\frac{1}{\lambda^{1/2}}\ln^{1/2}\frac1m \mu^{-1/2} I^{1/2}(u) 
+c\frac{1}{\lambda m^{1/2}}\ell^{1/2} I^{1/2}(u),
 \end{align*}
which  gives $I(u)\ge c\min\{\mu\theta^2\lambda \ln\frac1m, \frac{\theta^2\lambda^2 m}{\ell} \ln^2\frac1m \}$ and concludes the proof in this case.
 }

If instead the second term in \eqref{eq:compNdue} is the {larger one}, we write 
\begin{equation*}\begin{split}
{\frac1{256} \theta\ln \frac1{ m}}\le&
\int_{(0,1)} u_2(x_1^\ast+s\xi_1,s\xi_2) \psi'(s)\ds\\
 =& - \int_{(0,1)} \int_{( s\xi_2,1)} \partial_2 u_2(x_1^\ast+s\xi_1,x_2) \, \psi'(s) \dy\ds + \int_{(0,1)} u_2(x_1^\ast+s\xi_1,1) \psi'(s) \text{ d}s,
 \end{split}
\end{equation*}
 with the first integral being estimated by $2\|\psi'\|_{L^2((0,1))} I^{1/2}(u)$.
{As above, using the estimates for $\psi$  in (\ref{eqestpsiwithbdryd}) and Lemma 
\ref{lemmah12app} with 
the extension to $F_4:=(-2,\ell+2)\times (-1,0)$ leads to
\begin{equation*}
{\frac1{256} \theta\ln \frac1{ m}}
 \le 2\| \nabla \Psi\|_{L^2(F_4)} \|\nabla u_2\|_{L^2(F_4)}
 +2\|\psi'\|_{L^2((0,1))} I^{1/2}(u)
 \le c\ln^{1/2}\frac1m \mu^{-1/2} \frac{1}{\lambda^{1/2}} I^{1/2}(u) 
+c\frac{1}{\lambda m^{1/2}} I^{1/2}(u), 
\end{equation*}
which  gives $I(u)\ge c\min\{\mu\theta^2 \lambda\ln\frac1m, {\theta^2\lambda^2 m} \ln^2\frac1m \}$ and, since $1\le\ell$, concludes the proof also in this case.}  
\end{proof}

 We finally turn to the proof of Proposition \ref{lem:lbbranching}, in which the different ingredients proven in this Section are put together.\\
\begin{proof}[Proof of Proposition \ref{lem:lbbranching}] 
Let $m_0\in(0,\frac14]$ be given as in Lemma \ref{lem:thetasmallbranching}, and define $m_1\in(0,m_0)$ as the unique solution to $m_1\ln\frac1{m_1}=m_0$.
If {$\theta  \geq m_1$}
the statement follows directly from Lemma \ref{lem:thetalargebranching}. 
Otherwise we use Lemma \ref{lem:thetasmallbranching} with the following choices of the parameters $\lambda$ and $m$:
 \begin{itemize}
  \item [i)] Consider first the case $\mu \geq \varepsilon^{1/3} {\theta^{-2/3}}\ell^{-1/3}{m_0^{2/3}}$. 
  Choose $m:=m_0$ and $\lambda:=\eps^{1/3}{\theta^{-2/3}}\ell^{2/3}$.
 { Then $\eps\ell^2\le \theta^2$ implies on the one hand $\lambda\le 1$, 
  and on the other hand $\eps\le \theta^2\ell^{-2}\le \theta^2\ell$, which gives $\eps\le \theta^{2}\lambda$.}
Thus, $\lambda$ is admissible. 
In this case, $\mu\theta^2\lambda \ln \frac1m\le c\mu\theta^2$, hence the second and the last two terms in the minimum can be ignored.
  Therefore, Lemma \ref{lem:thetasmallbranching} yields 
  that 
  \[I_{\tilde{\Omega}_\ell}(u)\geq c\min\{\eps^{2/3}{\theta^{2/3}}\ell^{1/3},\,
  \mu\eps^{1/3}{\theta^{4/3}}\ell^{2/3}\}\ge c \eps^{2/3}{\theta^{2/3}}\ell^{1/3}\]
{where we used $\mu\theta^2\lambda\ge
\eps^{2/3}\theta^{2/3}\ell^{1/3}{m_0^{2/3}}$ by the assumption on $\mu$ and the definition of $\lambda$.}
  \item[ii)] Suppose now that $\mu < \varepsilon^{1/3}  \theta^{-2/3}{\ell^{-1/3}}{m_0^{2/3}}$.  
  We set $m := \max \{ \mu^{3/2}\varepsilon^{-1/2} \theta \ell^{1/2},\theta\ln\frac1\theta\}$ and $\lambda:=\max\{\eps\theta^{-2},{{\frac12}
  \mu^{-1/2}\eps^{1/2}\theta^{-1}\ell^{1/2}}\ln^{-1/2}{\frac{1}{m}}\}$. 
  {By the assumption on $\mu$ and $0<\theta<m_1$ we obtain $m\in[\theta,m_0]$.}
  Further, $\lambda\ge\eps\theta^{-2}$ by definition. 
  {It remains to show $\lambda\le1$.
Since $\eps\ell^2\le\theta^2$ implies $\eps\theta^{-2}\le 1$, it suffices to prove that
  \begin{equation}\label{eqminmmln}
  \min\{
  \ln(3+\frac{\eps}{\mu^3\theta^2\ell}),
  \ln(3+\frac1{\theta^2})\}
  \le 4 \ln \frac1m.
  \end{equation}
Indeed,  (\ref{eqminmmln}) and the assumption on $\eps\ell$ give
$   \eps\ell \le 4 \mu\theta^2\ln\frac1m$
which immediately implies $\lambda\le1$.

  It remains to prove the algebraic inequality  (\ref{eqminmmln}).
  If $m= \mu^{3/2}\varepsilon^{-1/2} \theta \ell^{1/2}$, then
(\ref{eqminmmln}) follows from the fact that for any $x\in(0,\frac14)$ we have $\ln(3+\frac1{x^2})\le 3\ln\frac1x$.
If instead $m=\theta\ln\frac1\theta$, we use analogously
\begin{equation*}
\ln(3+\frac1{x^2})\le 4\ln\frac1{x\ln \frac1x} \text{ for all } x\in(0,\frac14].
\end{equation*}
The last inequality is equivalent to  
$3+\frac1{x^2}\le\frac1{x^4\ln^4 \frac1x}$,
which is true, since $x\ln^2\frac1x\le 4e^{-2}\le\frac23$ for all $x\in (0,1)$. 
Therefore (\ref{eqminmmln}) holds.
  }
  
  \sloppypar
 We use Lemma \ref{lem:thetasmallbranching} and estimate the terms separately below. First, using  that $\lambda^{-1} = \min \{\varepsilon^{-1}\theta^2,
 {2}\mu^{1/2} \eps^{-1/2}\theta\ell^{-1/2}\ln^{1/2}\frac{1}{m} \}$
 and then that $\eps\le\theta^2\ell^{-2}\le \theta^2\ell$, we find
 {
  \[ \frac{\eps\ell}{\lambda}=\min\left\{\theta^2\ell,\ 2\mu^{1/2}\eps^{1/2}\theta\ell^{1/2}\ln^{1/2}\frac{1}{m}\right\}
  \ge\min\left\{\eps^{2/3}\theta^{2/3}\ell^{1/3},\ 2\mu^{1/2}\eps^{1/2}\theta\ell^{1/2}\ln^{1/2}\frac{1}{m}\right\}
  .\]}
{From $m\ge \theta\ln\frac1\theta$ we get $\mu m^2\ge \frac13\mu\theta^2\ln(3+\frac1{\theta^2})$, and recalling the assumption on $\eps \ell$ we get
\begin{equation*}
 \mu m^2\ge \frac13 \mu\theta^2\ln(3+\frac1{\theta^2})\ge \frac14(\eps \ell \mu\theta^2)^{1/2}\ln^{1/2}(3+\frac1{\theta^2}).
\end{equation*}
}  
  {Next, by definition of $\lambda$, we have $\lambda\geq {\frac12}\mu^{-1/2}\eps^{1/2}\theta^{-1}\ell^{1/2}\ln^{-1/2}\frac{1}{m}$ and hence
\[\mu\theta^2\lambda\ln \frac{1}{m}\geq \frac12\mu^{1/2}\eps^{1/2}\theta\ell^{1/2}\ln^{1/2} \frac{1}{m}.\]  
  Finally, using that $\lambda\geq {\frac12}\mu^{-1/2}\eps^{1/2}\theta^{-1}\ell^{1/2}\ln^{-1/2}\frac{1}{m}$, $m\geq\mu^{3/2}\varepsilon^{-1/2} \theta \ell^{1/2}$ and then $\ln^{1/2}\frac{1}{m}\leq \ln\frac{1}{m}$, we find
  \begin{eqnarray*}
 \theta^2 \ell^{-1}\lambda^2m\ln^2\frac{1}{m}\geq {\frac14}\theta^2\ell^{-1}\left(\mu^{-1}\eps\theta^{-2}\ell\ln^{-1}\frac{1}{m}\right)\left(\mu^{3/2}\eps^{-1/2}\theta\ell^{1/2}\right)\ln^2\frac{1}{m}\geq {\frac14}\mu^{1/2}\eps^{1/2}\theta\ell^{1/2}\ln^{1/2}\frac{1}{m}.
  \end{eqnarray*}
  }  
  Putting things together,
  and recalling (\ref{eqminmmln}) 
  we obtain
  \begin{equation*}\begin{split}
  I_{\tilde{\Omega}_\ell}(u)&\geq  c\min\Big\{\eps^{2/3}\theta^{2/3}\ell^{1/3},\\
&  \mu^{1/2}\eps^{1/2}\theta\ell^{1/2}\ln^{1/2}(3+\frac{1}{\theta^2}),
  \mu^{1/2}\eps^{1/2}\theta\ell^{1/2}\ln^{1/2}(3+\frac{\eps}{\mu^3\theta^2\ell}),
{\mu\theta^2\ln(3+\frac\theta\mu)},
{ \mu \theta^2 \ln(3+\ell)}
  {\Big\}}
\end{split}  \end{equation*}  
which concludes the proof also in this case.
  \end{itemize}
  \end{proof}

\subsection{A lower bound for small $\eps$, but $\eps L$ not small.}\label{sec:lb3}
We will now consider the remaining cases. We focus here on the situation in which $\eps$ is small, but $L$ is so large that a straight interface along the entire martensitic sample is not optimal.  Although this condition does not appear explicitly in the assumptions of 
Proposition \ref{propinterp}, the result will only be useful in this situation, as the term $\eps L$ appears as one of the options in the estimate.

The proof of the lower bound in this case has a structure similar to the one of Section \ref{sec:lbunif}.
We shall use, as above, the subdivision of the set of diagonal slices into various subsets. In particular, we shall show in Lemma
\ref{lemmainterpolationestim} that the interface between a $\mathcal P$ and a $\mathcal C$ slice is, energetically speaking, expensive. At variance with 
Section \ref{sec:lbunif}, the interpolation between an affine region and a region with periodic boundary values will no longer be penalized with a $D^2u$ term but with a $\partial_1u_1$ term. This requires estimates on $u$, and not only on its derivatives.

\begin{prpstn}[A lower bound in the case {$\eps\le \mu\theta^2$ and $\varepsilon \le \theta^2$}] \label{lem:logsammeln}\label{propinterp}
  There exists $c>0$ such that for  all $L\in[1/2,\infty)$, $\theta\in(0,1/2]$, $\eps>0$, $\mu>0$, and $u \in \mathcal{X}$ with
\begin{equation*}
\eps\le\mu\theta^2 \qquad\text{ and } \qquad \eps\le\theta^2
\end{equation*}
{we have}
\begin{equation*}
 I(u)\ge  c \min\{\eps L,\mu\theta^2 \ln(3+L),  \mu\theta^2\ln (3+\frac{\eps L}{\mu\theta^2}) +
   \eps^{1/2}\theta^{3/2} + 
   \mu\theta^2\ln (3+\frac{\eps}{\mu^2\theta^2}), 
   \mu\theta^2\ln (3+\frac{\eps L}{\mu\theta^2}) +
   \mu\theta^2\ln(3+\frac\theta\mu) \}.
\end{equation*}
\end{prpstn}

We first prove {some} lemmata used in the proof of Proposition \ref{lem:logsammeln}.
The first one concerns a local variant of the set $\mathcal C$, for which a sharper estimate on the volume is possible.

\newcommand\setRone{\mathcal R}

\begin{lmm}[Estimates near the boundary]\label{lem:linftyonsmallsets}
Assume that {$ \delta\in(0,\frac1{64}\theta]$}, $u \in W^{1,2}(\Omega_L,\R^2)$, and let
{\begin{equation*}
\setRone:= \Big{\{}x_1\in (0,L-\xi_1)\,:\, 
\frac1{32}\theta\le 
\max\{ \|u_{x_1}^\xi(s)-u^\xi_{x_1}(0)\|_{L^\infty((0,\delta ))}, 
\|u_{x_1}^\xi(1-s)-u^\xi_{x_1}(1)\|_{L^\infty((0,\delta))} \Big\}.
       \end{equation*}
Then
\begin{equation*}
 I(u)\ge c\theta \calL^1(\setRone).
\end{equation*}
}
\end{lmm}

\begin{proof}
For almost every $x_1\in \setRone$ we have $v:=u^\xi_{x_1}\in W^{1,2}((0,1))$.
For  $s\in (0,\delta)$ we estimate, using  (\ref{eq:estduxi}),
\begin{align*}
|v(1-s)-v(1)|&\le s^{1/2}  \|v'\|_{L^2(( 1-s,1))}\\
& \le s^{1/2}(  \|1\|_{L^2(( 1-s,1))}+ \|\min \{ |v' -\theta|, |v'+(1-\theta)|\}\|_{L^2((0,1))})\\
& \le  {\delta} + \delta ^{1/2} 5 \|\min \{ |e(u)-\theta e_1 \odot e_2|, |e(u)+(1-\theta) e_1 \odot e_2|\}\|_{L^2(\Delta^\xi_{x_1})} 
\end{align*}
{
and correspondingly for $|v(s)-v(0)|$. Therefore for almost any $x_1\in \setRone$
we have 
\begin{equation*}
 \frac1{32}\theta\le \delta+\delta ^{1/2} 5 \|\min \{ |e(u)-\theta e_1 \odot e_2|, |e(u)+(1-\theta) e_1 \odot e_2|\}\|_{L^2(\Delta^\xi_{x_1})}. 
\end{equation*}
For $\delta\le\frac1{64}\theta$  we deduce
\begin{equation*}
 c\theta\le  \|\min \{ |e(u)-\theta e_1 \odot e_2|, |e(u)+(1-\theta) e_1 \odot e_2|\}\|_{L^2(\Delta^\xi_{x_1})}^2 
\end{equation*}
and integrating {over $x_1\in\mathcal{R}$} {we obtain} the assertion.
}
\end{proof}

\begin{lmm}[Interpolation estimate]\label{lemmainterpolationestim}
Let $\mathcal P$ be defined as in (\ref{eqdefPth}), $p:=\calL^1(\mathcal P)$.
Assume
\begin{equation*}
\eps\le\mu\theta^2 \qquad\text{ and } \qquad \eps\le\theta^2.
\end{equation*}
Then
\begin{equation*}
I(u)\ge c\min\{\eps^{1/2}\theta^{3/2},  \mu\theta^2 p, \theta^2 p, \eps L\}. 
\end{equation*}
\end{lmm}
\begin{proof}
The proof is based on selecting a good slice in which $u$ is approximately affine, and another one in which the boundary values are close to each other, and estimating the energy in between. We shall work on a thin slice around the boundary, of width  $\delta:=\frac1{64}\theta$.

{\bf Step 1. Estimate on good $\mathcal C$-slices.}\\
We recall that $\mathcal C$ 
was defined in (\ref{eqdefC}) as the set of slices such that the deformation is close to one of the two martensite variants in the interior, and that by Lemma \ref{lem:gammawuerfel}(\ref{lem:gammawuerfelest}) it obeys {(recall that $\eps\leq\theta^2$)}
\begin{equation}\label{eqfinaestC}
 I(u)\ge c \eps \calL^1\left([0,L-\xi_1]\setminus \mathcal C\right).
\end{equation}
We let $\Gstar$ be the set from  Lemma \ref{lem:costsu2inA} with $\bar C:=2^{-7}$, which obeys
\begin{equation}\label{eqfinaestG}
 I(u)\ge c \min\{\theta^2 ,\mu\theta^2 \}\calL^1(\Gstar).
\end{equation}
By Lemma \ref{lem:costsu2inA}(\ref{lem:costsu2inA3}),
\begin{equation*}
{\frac1\delta\int_{(0,\delta)} |u_2(x_1+{(1-s)\xi_1},1-s)-u_2(x_1+s\xi_1,s)|\ds < \frac1{16}\theta}
\text{ for any } x_1\in [0,L-\xi_1]\setminus\Gstar.
\end{equation*}

{
We can assume $\calL^1(\mathcal C\setminus\Gstar)>0$. Indeed, if this were not the case, then (up to null sets)
$\calC\subset\Gstar$ and $([0,L-\xi_1]\setminus \calC) \cup \Gstar=[0,L-\xi_1]$, which implies
 $I(u)\ge cL \min\{\theta^2,\mu\theta^2,\eps\}=c \eps L$ and concludes the proof.} 

{We claim that 
\begin{equation}\label{claimxc}
\begin{split}
\frac1{2}\theta 
&\le\frac1\delta
\int_{(0,\delta)}  |u_1((x_c,0)+s\xi)-u_1((x_c,0)+(1-s)\xi)|\ds
\text{ for any ${x_c}\in\mathcal C\setminus\Gstar$.}
\end{split}
\end{equation}
}
{To see this,
 assume $x_1\in\mathcal C$ and for $\sigma\in\{0,1\}$ let $f_\sigma(s):=u_{x_1}^\xi(0)+s(\theta-\sigma)$.
Then
$|f_\sigma(1-s)-f_\sigma(s)|
=|1-2s|\,|\theta-\sigma|\ge |1-2s|\theta$, therefore for any $s\in(0,\delta)$ we have
\begin{equation*}
\begin{split}
\frac78\theta\le (1-2s)\theta\le&
\min_{\sigma\in\{0,1\}} |f_\sigma(s)-f_\sigma(1-s)|\\
\le &|u^\xi_{x_1}(s)-u^\xi_{x_1}(1-s)|
+ \min_{\sigma\in\{0,1\}}  (|u^\xi_{x_1}(s)-f_\sigma(s)|+
|u^\xi_{x_1}(1-s)-f_\sigma(1-s)|)\\
\le &|u^\xi_{x_1}(s)-u^\xi_{x_1}(1-s)|+\frac2{16}\theta.
\end{split}
\end{equation*}
Therefore, recalling that $u_{x_1}^\xi(s)=u_1((x_1,0)+s\xi)+4u_2((x_1,0)+s\xi)$,
\begin{equation*}
\begin{split}
\frac34\theta&\le
|u_{x_1}^\xi(s)-u_{x_1}^\xi(1-s)|\\
&\le
 |u_1((x_1,0)+s\xi)-u_1((x_1,0)+(1-s)\xi)|
+4|u_2((x_1,0)+s\xi)-u_2((x_1,0)+(1-s)\xi)| 
\end{split}
\end{equation*}
for any $x_1\in\mathcal C$.
Averaging over $s\in(0,\delta)$,  and using that $x_1\not\in\Gstar$,
\begin{equation*}
\begin{split}
\frac34\theta &\le\frac1\delta
\int_{(0,\delta)}\left|u_{x_1}^\xi(s)-u_{x_1}^\xi(1-s)\right|\ds\\
&\le\frac1\delta
\int_{(0,\delta)}  \left|u_1((x_1,0)+s\xi)-u_1((x_1,0)+(1-s)\xi)\right|\ds
+\frac14\theta
\end{split}
\end{equation*}
which proves (\ref{claimxc}).
}

{\bf Step 2. Estimate on good $\mathcal P$-slices.}\\
We 
consider
the set $\setRone$ defined in  Lemma \ref{lem:linftyonsmallsets},  which obeys
\begin{equation}\label{eqfinaestR}
 I(u)\ge c \theta \calL^1(\setRone).
\end{equation}
We can assume $\calL^1(\mathcal P\setminus\setRone\setminus\Gstar)>0$. Indeed, if this were not the case, then 
$\calL^1(\setRone\cup\Gstar)\ge p$, $I(u)\ge cp \min\{\theta^2,\mu\theta^2\}$, and the proof is concluded.\\
We claim  that
\begin{equation}\label{claimxp}
\begin{split}
\frac1\delta
\int_{(0,\delta)}  \left|u_1((x_p,0)+s\xi)-u_1((x_p,0)+(1-s)\xi)\right|\ds\le
\frac3{8}\theta 
\text{ for any $x_p\in\mathcal P\setminus\setRone\setminus\Gstar$.}
\end{split}
\end{equation}
Indeed, let $x_1\in \mathcal P$. Then 
$|u(x_1,0)-u(x_1+\xi_1,1)|\le 2^{-7}\theta$, therefore
$|u_{x_1}^\xi(0)-u_{x_1}^\xi(1)|\le 2^{-4}\theta$.
If $x_1\not\in\setRone$, then 
a triangular inequality shows that for any $s\in(0,\delta)$
\begin{equation*}
\begin{split}
|u_{x_1}^\xi(s)-u_{x_1}^\xi(1-s)|&\le
|u_{x_1}^\xi(0)-u_{x_1}^\xi(1)| +
|u_{x_1}^\xi(s)-u_{x_1}^\xi(0)| +
|u_{x_1}^\xi(1-s)-u_{x_1}^\xi(1)|\le\frac{1}{16}\theta+\frac{2}{32}\theta=\frac18\theta.
 \end{split}
\end{equation*}
A similar computation as above, using 
\begin{equation*}
\begin{split}
 |u_1((x_1,0)+s\xi)-u_1((x_1,0)+(1-s)\xi)|
\le |u_{x_1}^\xi(s)-u_{x_1}^\xi(1-s)|
+
4|u_2((x_1,0)+s\xi)-u_2((x_1,0)+(1-s)\xi)|
\end{split}
\end{equation*}
and $x_1\not\in\Gstar$, leads to
\begin{equation*}
\begin{split}
\frac1\delta
\int_{(0,\delta)}  |u_1((x_1,0)+s\xi)-u_1((x_1,0)+(1-s)\xi)|\ds \le
&\frac1\delta
\int_{(0,\delta)}|u_{x_1}^\xi(s)-u_{x_1}^\xi(1-s)|\ds
+\frac14\theta\le \frac38\theta.
\end{split}
\end{equation*}
This concludes the proof of (\ref{claimxp}).

{\bf Step 3. Interpolation.}\\
We choose $x_p\in\calP\setminus \setRone\setminus\Gstar$ and $x_c\in\mathcal C\setminus\Gstar$ and compute
\begin{equation*}
 \begin{split}
 | u_1((x_c,0)+s\xi)-u_1((x_c,0)+(1-s)\xi)|
\le&
   | u_1((x_c,0)+s\xi)-u_1((x_p,0)+s\xi)|\\
   &+
   | u_1((x_p,0)+s\xi)-u_1((x_p,0)+(1-s)\xi)|\\
   &+
   | u_1((x_p,0)+(1-s)\xi)-u_1((x_c,0)+(1-s)\xi|.
  \end{split}
\end{equation*}
Averaging over $s\in(0,\delta)$ and using  (\ref{claimxc}),   (\ref{claimxp}) and Hölder gives
\begin{equation*}
\begin{split}
\frac12\theta\le & \frac1\delta
\int_{(0,\delta)}  | u_1((x_c,0)+s\xi)-u_1((x_c,0)+(1-s)\xi)|\ds\\
\le &
\frac1\delta \int_{(0,\delta)}  | u_1((x_p,0)+s\xi)-u_1((x_p,0)+(1-s)\xi)|\ds+
\frac1\delta
\int_{(0,\delta)\cup(1-\delta,1)}  | u_1((x_p,0)+s\xi)-u_1((x_c,0)+s\xi)|\ds\\
\le & \frac38\theta + 
\frac{\sqrt2}{\delta^{1/2}} \Big(\int_{(0,\delta)\cup(1-\delta,1)}  | u_1((x_p,0)+s\xi)-u_1((x_c,0)+s\xi)|^2\ds\Big)^{1/2}.
\end{split}
\end{equation*}
Therefore
\begin{equation*}
\begin{split}
\frac1{128}\theta^2\delta\le 
\int_{(0,1)}  | u_1((x_p,0)+s\xi)-u_1((x_c,0)+s\xi)|^2\ds.
\end{split}
\end{equation*}
By the fundamental theorem of calculus and Hölder's inequality, for any $s\in (0,1)$ we have
\begin{equation*}
 | u_1((x_p,0)+s\xi)-u_1((x_c,0)+s\xi)|^2
 \le |x_p-x_c| \int_{(0,L)} |\partial_1 u_1|^2(t,s) \dt.
\end{equation*}
Integrating over $s$ gives
\begin{equation*}
\frac{\theta^3}{c}\le
 \int_0^1 |u_1((x_p,0)+s\xi)-u_1((x_c,0)+s\xi)|^2\ds\le 
  |x_c-x_p|\, \|\partial_1u_1\|^2_{L^2((x_c,x_p)\times(0,1))}\le|x_c-x_p|\, I(u).
\end{equation*}

{\bf Step 4. Conclusion of the proof.}\\
{
Let $\beta:=\calL^1(
\setRone \cup \Gstar
\cup(
[0,L-\xi_1]\setminus\mathcal C\setminus\mathcal P))$.
On the one hand,  (\ref{eqfinaestC}), (\ref{eqfinaestG}), (\ref{eqfinaestR})
give $I(u)\ge c \min\{\eps,\theta^2,\mu\theta^2\} \beta=c \eps \beta$.
On the other hand,
$\inf\{|x_c-x_p|: \text{ (\ref{claimxc}) and (\ref{claimxp}) hold}\}\le \beta$. Therefore
\begin{equation*}
 I(u)\ge c \min_{\beta'\in(0,L]}\big[ \frac{\theta^3}{\beta'} +\eps \beta' \big].
 \end{equation*}
The minimum is attained at $\beta'=\eps^{-1/2}\theta^{3/2}$ or at $\beta'=L$, and
 gives
\begin{equation*}
 I(u)\ge c\min\{\theta^{3/2}\eps^{1/2}, \eps L\}
 \end{equation*}
 which concludes the proof.
 }
  \end{proof}

\begin{lmm}[The boundary logarithm]\label{lemmabdryln}
There are $c>0$ and $m_2\in(0,\frac14]$ such that the following holds. 
If 
\begin{equation*}
\frac12\le L, \qquad
0<\theta\le m_2, \qquad
\eps\le \mu\theta^2 \qquad
\mu\le \theta,\quad\text{and}\quad
\quad \mu^2\theta^2\le\eps,
\end{equation*}
then for any $u\in{\mathcal{X}}$ one has
\begin{equation*}
\frac1c I(u)\ge 
  \min\left\{\eps L,
  \mu\theta^2 \ln(3+\frac1{\theta^2}),
  \mu\theta^2 \ln(3+L) ,
  \mu\theta^2\ln(3+\frac\theta\mu),
   {\mu\theta^2}\ln (3+\frac{\eps}{\mu^2\theta^2})
  \right\}.
\end{equation*}
 \end{lmm}

\begin{proof}
{\bf Step 1. Energy estimate.}\\
We show that there are $c>0$, $m_2\in(0,\frac14]$ such that for any $m\in[\theta,m_2]$
there is $q\in[0,L]$ such that
\begin{equation}\label{eqstep1energyest}
 \frac1c I(u)\ge  \min\Big\{ \eps L,
  \mu\theta^2\ln(3+\frac\theta\mu),
  \eps q +   \frac{\theta^2m}{q+1}\ln^2 \frac1m,
     \mu\theta^2 \ln\frac1m,
      \mu\theta^2 \ln(3+L) 
  \Big\}.
\end{equation}

{
We define, similar to Lemma \ref{lem:thetasmallbranching},  $\psi_C(t):=\max \{\psi_B(t), \psi_B(1-t)\}$, where
\begin{equation*}
  \psi_B(t):=\left[ \ln \frac{1}{m} - \left(\ln \frac{|t|}{m}\right)_+ \right]_+
  =
  \begin{cases}
  \ln\frac1m, & \text{ if } |t|\le m,\\
  \ln\frac{1}{|t|}, & \text{ if } m<|t|\le 1,\\
   0, & \text{ if } |t|>1\,,
  \end{cases}
\end{equation*}
and compute
$\|\psi_B\|_{L^1((0,1))} \le 1$,
$\|\psi_B'\|_{L^1((0,1))}\le \ln\frac1m$, and
$\|\psi_B'\|_{L^2((0,1))}\le m^{-1/2}$, which imply
\begin{equation}\label{eqestpsiclemmaln}
\|\psi_C\|_{L^1((0,1))} \le 2, \quad
\|\psi_C'\|_{L^1((0,1))}\le 2\ln\frac1m,\quad \text{ and }\quad
\|\psi_C'\|_{L^2((0,1))}\le \frac2{m^{1/2}}.
\end{equation}
}

We first claim that
\begin{equation}\label{eqthetalnxc}
\begin{split}
\frac78 \theta\ln\frac1m 
&\le 
 \Big|\int_0^1 u_{x_1}^\xi(s)\psi_C'(s) \ds \Big| 
\quad \text{ for any $x_1\in\mathcal C$},
 \end{split}
\end{equation}
where $\mathcal C$ was defined in (\ref{eqdefC}).
{To see this we compute, for any $x_1\in\mathcal C$ and $\sigma\in\R$,
\begin{equation*}
\begin{split}
 \sigma\ln\frac1m &= \int_0^1 \frac{d}{ds} (s\sigma \psi_C(s)) \ds 
 =\sigma \int_0^1\psi_C\ds + \int_0^1 (s\sigma)\psi_C'(s) \ds\\
 & =\sigma \int_0^1\psi_C\ds + \int_0^1 (s\sigma+u^\xi_{x_1}(0)-u^\xi_{x_1}(s))\psi_C'(s) \ds
 + \int_0^1 (u_{x_1}^\xi(s){-u^\xi_{x_1}(0)})\psi_C'(s) \ds. 
\end{split}
\end{equation*}
{Since $\psi_C(0)=\psi_C(1)$, the last term disappears, and}
\begin{equation*}
\begin{split}
| \sigma|\ln\frac1m & \le |\sigma|\, \|\psi_C\|_{L^1((0,1))}  + 
\|s\sigma{+u^\xi_{x_1}(0)}-u^\xi_{x_1}(s)\|_{L^\infty((0,1))} 
\|\psi_C'\|_{L^1((0,1))} 
+\Big| \int_0^1 u_{x_1}^\xi(s)\psi_C'(s) \ds\Big|.
\end{split}
\end{equation*}
At this point we recall (\ref{eqestpsiclemmaln}) and that $x_1\in\mathcal C$. Therefore there is a choice of $\sigma\in\{\theta,{\theta-1}\}$ such that
\begin{equation*}
\begin{split}
| \sigma|\ln\frac1m & \le 2|\sigma|+ \frac1{16}\theta
\ln\frac1m+\Big| \int_0^1 u_{x_1}^\xi(s)\psi_C'(s) \ds\Big|.
\end{split}
\end{equation*}
}

{
If $m_2$ is sufficiently small, $2\le 2^{-4}\ln\frac1m$. For both choices of $\sigma$ we have $\theta\le|\sigma|$. Therefore
\begin{equation*}
\begin{split}
\frac78 \theta\ln\frac1m & \le \Big| \int_0^1 u_{x_1}^\xi(s)\psi_C'(s) \ds\Big|,
\end{split}
\end{equation*}
which concludes the proof of (\ref{eqthetalnxc}).
}

{
Let $\Gstar$  be as in  Lemma \ref{lem:costsu2inA} with 
$\bar C:=2^{-8}$.
Since $\psi_C'(s)=-\psi_C'(1-s)=-1/s$ on $(m,1/2)$ and $\psi_C'(s)=\psi_C'(1-s)=0$ on $(0,m)$, 
using Lemma \ref{lem:costsu2inA}(\ref{lem:costsu2inA46bdryln}) for 
$x_1\not\in\Gstar$ we have
\begin{equation*}
\begin{split}
\left|\int_0^1 4u_2((x_1,0)+s\xi)\psi_C'(s) \ds\right|\le &
 \int_m^{1/2} \frac4s |u_2((x_1,0)+s\xi)-u_2((x_1,0)+(1-s)\xi)|\ds
 \le \frac18\theta\ln\frac1m. 
\end{split}
\end{equation*}
Recalling that $u_{x_1}^\xi(s)= (u_1+4u_2)((x_1,0)+s\xi)$, we see that 
\begin{equation}\label{eq34gttheatpsic1}
\frac34 \theta\ln\frac1m \le \left|\int_0^1 u_1((x_1,0)+s\xi) \psi_C'(s)\ds\right| \qquad\text{for any $x_1\in\mathcal C\setminus\Gstar$}.
\end{equation}
}

By Lemma \ref{lem:gammawuerfel}(\ref{lem:gammawuerfelest}) and $\eps \le\theta^2$ we have $I(u)\ge c\eps \calL^1([0,L-\xi_1]\setminus\mathcal C)$. 
By Lemma \ref{lem:costsu2inA} and $\eps\le \min\{\theta^2,\mu\theta^2\}$ we have $I(u)\ge c \eps \calL^1(\Gstar)$. 
Then 
\begin{equation*}
I(u)\ge c \eps q, \quad\text{  with }\quad q:=\calL^1(\Gstar\cup ([0,L-\xi_1]\setminus\mathcal C)).  
\end{equation*}
If $q\ge \frac12 L$ we are done. Otherwise we 
 pick $x_1\in[0,{q+\frac14}]\cap \mathcal C\setminus\Gstar$. We compute
\begin{equation*}
\begin{split}
\left| \int_0^1 (u_1((x_1,0)+s\xi)-u_1(0,s)) \psi_C'(s)\ds\right|
&\le \int_0^1\int_0^{x_1+\xi_1} |\partial_1 u_1|(t,s) |\psi_C'|(s) \dt\ds\\
& \le \|\psi_C'\|_{L^2((0,1))} (x_1+\frac14)^{1/2} I(u)^{1/2}
\le c m^{-1/2} (q+1)^{1/2} I(u)^{1/2}
\end{split}
\end{equation*}
where we used (\ref{eqestpsiclemmaln}). 
If the right-hand side is larger than $\frac14\theta\ln\frac1m$ then
$I(u)\ge c (q+1)^{-1} \theta^2 m\ln^2\frac1m$ and {(\ref{eqstep1energyest}) is proven.} Otherwise, with (\ref{eq34gttheatpsic1}) we obtain
\begin{equation*}
\frac12 \theta\ln\frac1m \le \Big|
 \int_0^1 u_1(0,s) \psi_C'(s)\ds\Big|.
\end{equation*}
{
We define $\Psi_C:\R^2\to\R$ by
\begin{equation*}
 \Psi_C(x):=\max\{\psi_B(|x-(0,1)|), \psi_B(|x|)\}.
\end{equation*}
One easily checks that $\Psi_C(0,t)=\psi_C(t)$ and
$\Psi_C(t,0)=\Psi_C(t,1)=\psi_B(t)$ for $t\in(0,1)$, $\Psi_C=0$ on the rest of 
the boundary of $F_1:=(-1,1)\times(-1,2)\setminus(0,1)^2$. Further, $\|\nabla\Psi_C\|_{L^2(\R^2)}^2\le c\ln\frac1m$.

By Lemma \ref{lemmaystin} with $c_*=2^{-4}$ and $\ell=2L$, either $I(u)\ge c\min\{\mu\theta^2\ln\frac1m, \mu\theta^2\ln(3+\frac\theta\mu),\mu\theta^2\ln(3+L)\}$ and (\ref{eqstep1energyest}) holds, so that  we are done, or
\begin{equation*}
\left| \int_0^1 \psi_B'(s)(u_1(s,1)-u_1(s,0)) \ds\right| \le \frac1{16}\theta\ln\frac1m.
\end{equation*}
We 
 conclude that (recalling Lemma  \ref{lemmah12app})
\begin{equation*}
\frac18 \theta\ln\frac1m \le \int_{\partial F_1} u_1 \partial_\tau \Psi_C \dcalH^1
\le c \|\nabla\Psi_C\|_{L^{2}(F_1)} \|\nabla u_1\|_{L^{2}(F_1)}
\le c \ln^{1/2}\frac1m
\mu^{-1/2} I(u)^{-1/2} 
\end{equation*}
which implies  $I(u)\ge c\mu\theta^2\ln \frac1m$ and concludes the proof of (\ref{eqstep1energyest}).
}

{\bf Step 2. Choice of the parameters.}\\
 {
We first remark that
\begin{equation}\label{eqminabx1}
  \min_{x\ge0} \big[ ax+\frac{b}{x+1}\big] =\min_{x\ge0} \big[a(x+1)+\frac{b}{x+1}\big] -a\ge 2a^{1/2}b^{1/2}-a\ge a^{1/2}b^{1/2}
  \quad\text{ whenever  $0<a\le b$ }.
\end{equation}
}

{
We use (\ref{eqstep1energyest}) with $m:=\max\{\theta,m_2\frac{\mu^2\theta^2}{\eps}\}\in[\theta,m_2]$.
If the first, the second or the last term are the smallest, the proof is concluded.  
Assume that the smallest is the third or the fourth one.
We remark that $\eps\le\mu\theta^2$ and $\mu\le\theta$ imply $\eps\le\theta^3$, hence $\eps\le \theta^2m\ln^2\frac1m$. Optimizing 
the third term in (\ref{eqstep1energyest})
in $q$  with (\ref{eqminabx1}) leads to
\begin{equation*}
  \frac1c I(u)\ge 
  \min\left\{
  {\eps^{1/2}\theta m^{1/2} }\ln\frac1m,
  \mu\theta^2 \ln\frac1m
  \right\}.
 \end{equation*}
 }
At this point we distinguish two cases. If 
{$\eps\le m_2\mu^2\theta$ then}
$m=m_2\frac{\mu^2\theta^2}{\eps}$, we insert and obtain 
$I(u)\ge c\mu\theta^2 \ln \frac{\eps}{m_2 \mu^2\theta^2}
\ge c \mu\theta^2 \ln (3+\frac{\eps}{\mu^2\theta^2})$, which concludes the proof.
If instead $m_2\mu^2\theta<\eps$ then $m=\theta$, and the above estimate gives
$ I(u)\ge 
 c\mu\theta^2 \ln\frac1\theta
 \ge c\mu\theta^2 \ln(3+\frac1{\theta^2}) $, which also concludes the proof.
\end{proof}

\begin{proof}[Proof of Proposition \ref{propinterp}]
{We first recall 
that by Lemma \ref{lem:gammawuerfel}(\ref{lem:gammawuerfelest}) we have, since $\eps\le\theta^2$,
\begin{equation}\label{eqsdfmujsth3}
\frac1c  I(u)\ge \mu\theta^2\ln\frac{L-\xi_1+1}{p+1} +\eps p.
 \end{equation}
The 
 interpolation estimate from Lemma \ref{lemmainterpolationestim} gives
\begin{equation}\label{eqinterpprop}
I(u)\ge c\min\{\eps^{1/2}\theta^{3/2},\mu\theta^2 p, \theta^2 p, \eps L\}.
\end{equation}
We treat the four cases separately.}

If the minimum in (\ref{eqinterpprop}) is $\eps L$, we are done. 

If the minimum in (\ref{eqinterpprop}) is $\mu\theta^2p$, then, 
as in (i) in the proof of  Prop. \ref{lem:logsammelnthetasmall},
\begin{equation*}
  \frac1c I(u)\ge \min_{p'\in[0,L-\xi_1]} \big(\mu \theta^2p' +\mu\theta^2\ln\frac{L-\xi_1+1}{p'+1}\big)
  =\mu\theta^2\ln(L-\xi_1+1)\ge c \mu\theta^2\ln(3+L),
 \end{equation*}
and we are done.

If the minimum in (\ref{eqinterpprop}) is $\theta^2p$, then {we can assume $\mu\ge 1$ and,}
as in (ii) in the proof of  Prop. \ref{lem:logsammelnthetasmall},
\begin{equation*}
  \frac1c I(u)\ge \min_{p'\in[0,L-\xi_1]} \left( \theta^2p'+\mu\theta^2\ln\frac{L-\xi_1+1}{p'+1}\right)
  \ge c\min\left\{\mu\theta^2\ln(3+L),\mu\theta^2\ln(3+\frac L\mu), {\theta^2L}\right\}.
 \end{equation*}
{Also in this case we are done. Indeed, 
recalling $\eps\le \theta^2$, we have 
{$\theta^2L\ge \eps L$ and}
$\eps^{1/2}\theta^{3/2}\le \theta^2\le \mu\theta^2$,
$\frac{\eps L}{\mu\theta^2}\le \frac{L}{\mu}$ and $\frac{\eps}{\mu^2\theta^2}\le \frac{1}{\mu^2}\le 1$, so that
$\mu\theta^2\ln(3+\frac{\eps L}{\mu\theta^2})+\mu\theta^2\ln(3+\frac{\eps}{\mu^2\theta^2})+\eps^{1/2}\theta^{3/2}
\le 4 \mu\theta^2\ln(3+\frac{L}{\mu})\le 4c I(u)$}.

We are left with the case that (\ref{eqinterpprop}) states $I(u)\ge c \eps^{1/2}\theta^{3/2}$.

We now show that (\ref{eqsdfmujsth3}) implies $I(u)\ge c \min\{\eps L, \mu\theta^2\ln (3+\frac{\eps L}{\mu\theta^2}),{\mu\theta^2(3+L)}\}$. Indeed, 
 the minimum of the expression in the right-hand side of (\ref{eqsdfmujsth3}) is attained at $p=0$, or at $p=L-\xi_1$, or at
 $p+1=\mu\theta^2/\eps$.
  If it is at $p=0$ then $I(u)\ge c\mu\theta^2\ln (L-\xi_1+1)\ge c\mu\theta^2\ln(3+L)$ and the proof is concluded.
  If it is at some $p\ge\frac15 L$  then $I(u)\ge c \eps p\ge c\eps L$ and the proof is concluded. 
  We are left with the case that 
   the first term 
  is at least $\mu\theta^2$ and  $p+1=\mu\theta^2/\eps$. Then, recalling the previous result from the interpolation estimate, we have
  \begin{equation}\label{eqendpf381}
 \frac1c I(u)\ge  \eps^{1/2}\theta^{3/2}+
 \mu\theta^2\ln(3+\frac{\eps L}{\mu\theta^2}).
  \end{equation}
  Let $m_2$ be as in  Lemma \ref{lemmabdryln}. {We next show that we can assume (with a constant $c$ depending on $m_2$) that
 \begin{equation}\label{eqendpf382}
\frac1c I(u)\ge   \min\left\{
\mu\theta^2 \ln(3+\frac1{\theta^2}),
  \mu\theta^2\ln(3+\frac\theta\mu),
{\mu\theta^2}\ln (3+\frac{\eps}{\mu^2\theta^2})
 \right\}.
 \end{equation}}
If at least one of 
 $m_2\le\theta$, 
  $\theta\le \mu$,
   $\eps\le\mu^2\theta^2$ holds then 
   {the minimum in (\ref{eqendpf382}) is below $c\mu\theta^2$ and  (\ref{eqendpf382})  follows from  (\ref{eqendpf381}).
   If instead} $\theta<m_2$, $\mu<\theta$, $\mu^2\theta^2<\eps$ then {Lemma \ref{lemmabdryln} shows that 
   either $\frac1c I(u)\ge\min\{\eps L, \mu\theta^2\ln(3+L)\}$, and we are done, or (\ref{eqendpf382})  holds.}
   
{It remains to show that (\ref{eqendpf381}) and (\ref{eqendpf382}) imply the assertion.
This is clear if $\eps\le \mu^2$, since in this case the first term in (\ref{eqendpf382})  is not relevant
{and $\eps^{1/2}\theta^{3/2}\ge0$.}
If instead $\mu^2<\eps$, {then $\mu^2\theta^2<\eps$ and therefore $(\frac{\mu^2\theta^2}{\eps})^{1/4}\ln(3+\frac{\eps}{\mu^2\theta^2})\le C$. This implies
$\mu\theta^2\ln(3+\frac{\eps}{\mu^2\theta^2})\le \mu^{1/2}\theta^{3/2}\eps^{1/4}C\le C \eps^{1/2}\theta^{3/2}$,} so that
 (\ref{eqendpf381}) concludes the proof.
} 
\end{proof}
{
\begin{rmrk}
In the last step of the proof, we just removed the scaling $\eps^{1/2}\theta^{3/2}$ in the regime $\mu\theta^2\ln (3+\frac{\eps L}{\mu\theta^2}) +
   \mu\theta^2\ln(3+\frac\theta\mu) +\eps^{1/2}\theta^{3/2}$. We note that this does not change the scaling behaviour of our lower bound: 
We distinguish two possibilities. If $\eps \leq\mu^2\theta$, then $\eps^{1/2}\theta^{3/2}\leq \mu\theta^2$, and we have 
 \[\mu\theta^2\ln (3+\frac{\eps L}{\mu\theta^2}) +
   \mu\theta^2\ln(3+\frac\theta\mu)\leq \mu\theta^2\ln (3+\frac{\eps L}{\mu\theta^2}) +
   \mu\theta^2\ln(3+\frac\theta\mu) +\eps^{1/2}\theta^{3/2}\leq 2\left( \mu\theta^2\ln (3+\frac{\eps L}{\mu\theta^2}) +
   \mu\theta^2\ln(3+\frac\theta\mu) \right).\]
 Otherwise, if $\eps>\mu^2\theta$, we have as in the proof of the upper bound (Theorem \ref{th:upperbound} (d))
   \begin{eqnarray*}
   \mu\theta^2\ln (3+\frac{\eps L}{\mu\theta^2}) +
   \mu\theta^2\ln(3+\frac\theta\mu) +\eps^{1/2}\theta^{3/2}\geq  \mu\theta^2\ln (3+\frac{\eps L}{\mu\theta^2}) +
   \mu\theta^2\ln(3+\frac\theta\mu) \geq c \mu\theta^2\ln(3+L).
   \end{eqnarray*}
\end{rmrk}
}

\subsection{Conclusion of the lower bound}\label{sec:lbconclusion}

We finally bring together the bounds proven in the previous Sections to obtain the desired lower bound.

\begin{thrm}[Lower bound] \label{th:lowerbound}
For any $\eps>0$, $\mu>0$, $L\ge \frac12$, $\theta\in (0,\frac12]$ and {any $u\in\mathcal{X}$} we have
\begin{equation*}
\begin{split}
 I(u)\ge 
   c\mathcal{I}({\mu,\eps,\theta,L}),
\end{split}
\end{equation*}
where $\mathcal I$ was defined in Theorem \ref{th:main}.
\end{thrm}

\begin{proof}
{We distinguish several cases. }
\begin{enumerate}
 \item  Assume {that} at least one of $\theta^2\le \eps$ and $\mu\theta^2\le \eps$ holds.
Then Proposition \ref{lem:logsammelnthetasmall}
gives
\begin{equation*}
 I(u)\ge c \min\left\{\mu\theta^2\ln (3+L), \theta^2L, \mu\theta^2\ln (3+\frac L\mu)+\eps\theta\right\},
\end{equation*}
and the proof is concluded.

 \item
From now on we have $\eps\le\theta^2$ and $\eps\le\mu\theta^2$.

We start by applying 
Proposition \ref{propinterp}, see also the remark afterwards. This   
gives that 
\begin{eqnarray*}
 I(u)\ge  c \min\left\{\eps L,\mu\theta^2 \ln(3+L),  \mu\theta^2\ln (3+\frac{\eps L}{\mu\theta^2}) +
   \eps^{1/2}\theta^{3/2} + 
   \mu\theta^2\ln (3+\frac{\eps}{\mu^2\theta^2}), 
   \mu\theta^2\ln (3+\frac{\eps L}{\mu\theta^2}) +
   \mu\theta^2\ln(3+\frac\theta\mu) \right\}.
\end{eqnarray*}
Note that the proof is concluded
unless the first term is the smallest.
Assume now that 
\begin{equation}\label{eqtheolbepsl}
 I(u)\ge c \eps L.
\end{equation}

If  $\theta^2\le \eps (2L)^2$, then
\begin{equation*}
 \frac1c I(u)\ge 2 \eps L 
 =\eps L + \eps^{2/3}L^{1/3}(\eps L^2)^{1/3}
 \ge \eps L + \frac12 \eps^{2/3}\theta^{2/3}L^{1/3}
\end{equation*}
concludes the proof. \\[.2cm]
In the following we can assume {that} $\eps (2L)^2<\theta^2$ and $\eps\le\mu\theta^2$ and (\ref{eqtheolbepsl}) {hold}.
 We distinguish more subcases, depending on the competition between the interfacial energy and the austenite elasticity. Specifically, the critical condition is whether
\begin{equation}\label{eqellessmutheta}
{2\eps L}\le \mu\theta^2 \min\left\{\ln(3+\frac{\eps}{{2\mu^3\theta^2L}}),
  \ln(3+\frac{1}{\theta^2})\right\}
\end{equation} 
holds or not. \\
Assume that (\ref{eqellessmutheta}) does not hold.  Then
  \begin{equation*}
2\eps L \ge   \min\left\{
 \mu^{1/2}\varepsilon^{1/2}\theta L^{1/2}  (\ln (3+ \frac 1 {\theta^{2}}))^{1/2}+\varepsilon L, 
 \mu^{1/2}\varepsilon^{1/2}\theta L^{1/2} (\ln (3+ \frac{\varepsilon}{ \mu^{3}\theta^{2}{L}}))^{1/2}+\varepsilon L\right\}
\end{equation*}
and with (\ref{eqtheolbepsl}) the proof is concluded.\\[.2cm]
Finally, assume that (\ref{eqellessmutheta}) holds. In this situation 
we can use Proposition \ref{lem:lbbranching} with $\ell=2L$. Recalling (\ref{eqtheolbepsl}), this gives
    \begin{equation*}
    \begin{split}
  \frac1c I(u) \geq \eps L + \min 
  \Big{\{}
&{ \mu \theta^2 \ln(3+L)}, 
  {\mu\theta^2\ln(3+\frac\theta\mu)},
{\varepsilon^{2/3}  \theta^{2/3}L^{1/3}},\\
&\mu^{1/2}\varepsilon^{1/2} 
{ \theta}L^{1/2} (\ln (3+ \frac 1 {\theta^2}))^{1/2}, \mu^{1/2}\varepsilon^{1/2}{\theta} L^{1/2}  (\ln (3+ \frac \varepsilon {\mu^3\theta^2{L}}))^{1/2}
  \Big{\}}.
         \end{split}
 \end{equation*}
We remark that the term $\mu\theta^2\ln(3+\frac\theta\mu)$ can be dropped. Indeed, if $\theta\ge \mu L$ then it is larger than $\ln(3+L)$. If instead $\theta<  \mu L$, then it is equivalent to the term $\mu\theta^2\ln(3+\frac{\eps L}{\mu\theta^2}) + \mu\theta^2\ln(3+\frac\theta\mu)+\eps^{1/2}\theta^{3/2}$, that we already included before. Indeed, in this case
$\eps^{1/2}\theta^{3/2}\le \eps^{1/2}L^{1/2}\mu^{1/2}\theta\le \eps L+\mu\theta^2$
 and, using $\ln(3+x)\le 2+x$ in the first term,
\[
 \mu\theta^2\ln(3+\frac{\eps L}{\mu\theta^2}) 
 + \mu\theta^2\ln(3+\frac\theta\mu) + \eps^{1/2}\theta^{3/2}
 \le 2\mu\theta^2+\eps L+ \mu\theta^2\ln(3+\frac\theta\mu)+\eps L +\mu\theta^2
 \le 4 \left( \mu\theta^2\ln(3+\frac\theta\mu)+\eps L\right).
\]
This concludes the proof.
\end{enumerate}
\end{proof}

\section*{Acknowledgements}
The authors are very grateful to an anonymous referee for the careful reading and the many useful comments, which {led} to a substantial improvement of the paper.



\end{document}